\documentclass[11pt]{amsart}

\usepackage{upgreek}
\usepackage{longtable}
\usepackage{ytableau}
\usepackage{shuffle}

\usepackage{amsmath, amsxtra, amsthm, amssymb,mathtools,bm,amsfonts}
\usepackage{mathrsfs}
\usepackage[normalem]{ulem}
\usepackage[mathscr]{euscript}
\usepackage{graphicx}
\usepackage{url}
\usepackage{color}
\usepackage{bbm}
\usepackage{tikz-cd}
\usepackage{tikz}

\usepackage[T1]{fontenc}
\graphicspath{/Figures/}

\usepackage{enumitem}
\usepackage{comment}
\usepackage[pdftex,hidelinks,backref=page]{hyperref}
\hypersetup{
    colorlinks,
    citecolor=magenta,
    filecolor=magenta,
    linkcolor=blue,
    urlcolor=black
}
\usepackage{cleveref}

\usepackage{xcolor}
\usepackage[colorinlistoftodos]{todonotes}
\renewcommand{\emptyset}{\varnothing}
\newcommand{\NN}{\mathbb N}
\newcommand{\QQ}{\mathbb Q}

\newcommand{\ZZ}{\mathbb Z}

\newcommand{\id}{\mathrm{id}}

\theoremstyle{definition}
\newtheorem{thm}{Theorem}[section]
\newtheorem{cor}[thm]{Corollary}
\newtheorem{lem}[thm]{Lemma}

\newtheorem{prop}[thm]{Proposition}
\newtheorem{defn}[thm]{Definition}

\newtheorem{eg}[thm]{Example}
\newtheorem{rem}[thm]{Remark}
\newtheorem{obs}[thm]{Observation}

\newtheorem{question}[thm]{Question}
\newtheorem{maintheorem}{Theorem}	\newtheorem{fact}[thm]{Fact}

\numberwithin{equation}{section}

\newcommand{\longword}{\widetilde{\omega}}

\newcommand{\indexedforests}{\operatorname{\mathsf{Forest}}}

\newcommand{\forestpoly}[1]{\mathfrak{P}_{#1}} 

\newcommand{\qsym}[2][]{
{\ifx&#1&%
  {\operatorname{QSym}_{#2}}
\else
  {{}^{#1}\!\operatorname{QSym}_{#2}}
\fi}
} 

\newcommand{\eqsym}[2][]{
{\ifx&#1&%
  {\operatorname{EQSym}_{#2}}
\else
  {{}^{#1}\!\operatorname{EQSym}_{#2}}
\fi}
} 

\newcommand{\qseq}[2][]{
{\ifx&#1&%
  {\operatorname{QSeq}_{#2}}
\else
  {{}^{#1}\!\operatorname{QSeq}_{#2}}
\fi}
}

\newcommand{\qsymide}[2][]{
{\ifx&#1&%
  {\operatorname{QSym}_{#2}^+}
\else
  {{}^{#1}\!\operatorname{QSym}_{#2}^+}
\fi}
} 
\newcommand{\sym}[1]{\operatorname{Sym}_{#1}} 
\newcommand{\esym}[1]{\operatorname{ESym}_{#1}} 
\newcommand{\symide}[1]{\sym{#1}^+} 
\newcommand{\supp}{\operatorname{supp}} 
\newcommand{\compatible}[2][]{
{\ifx&#1&%
  {\mathcal{C}(#2)}
\else
  {\mathcal{C}^{m}(#2)}
\fi}
} 
\newcommand{\internal}[1]{\operatorname{IN}(#1)} 
\newcommand{\suchthat}{\;|\;}
\newcommand{\nvect}{\mathsf{Codes}}

\newcommand{\schub}[1]{\mathfrak{S}_{#1}} 
\newcommand{\red}[1]{\operatorname{Red}(#1)} 
\newcommand{\des}[1]{\operatorname{Des}(#1)} 
\newcommand{\desnc}[1]{\operatorname{Des}_{\operatorname{NC}}(#1)} 
\newcommand{\ncpart}{\operatorname{NCPart}}

\newcommand{\ncperm}{\operatorname{NCPerm}}
\date{}

\newcommand{\idem}{\operatorname{id}} 
\newcommand{\slide}[2][]{
{\ifx&#1&%
  {\mathfrak{F}_{#2}}
\else
  {\mathfrak{F}_{#2}^{\underline{#1}}}
\fi}
} 
\newcommand{\qdes}[1]{\operatorname{LTer}(#1)} 
\newcommand{\tope}[2][]{
{\ifx&#1&%
  {\mathsf{T}_{#2}}
\else
  {\mathsf{T}_{#2}^{\underline{#1}}}
\fi}
} 
\newcommand{\bope}[2][]{
{\ifx&#1&%
  {\mathsf{B}_{#2}}
\else
  {\mathsf{B}_{#2}^{(#1)}}
\fi}
} 
\newcommand{\rope}[1]{\mathsf{R}_{#1}} 
\newcommand{\End}{\operatorname{End}} 
\newcommand{\Trim}[1]{\operatorname{Trim}({#1})} 
\newcommand{\Th}[1][]{\mathsf{ThMon}^{\underline{#1}}} 

\newcommand{\zigzag}[2][]
{
{\ifx&#1&%
  {\mathsf{ZigZag}_{#2}}
\else
  {\mathsf{ZigZag}_{#2}^{#1}}
\fi}
}

\newcommand{\ltfor}[2][] 
{
{\ifx&#1&%
  {\mathsf{LTFor}_{#2}}
\else
  {\mathsf{LTFor}_{#2}^{#1}}
\fi}
}

\newcommand{\rtfor}[2][] 
{
{\ifx&#1&%
  {\mathsf{RTFor}_{>#2}}
\else
  {\mathsf{RTFor}_{>#2}^{#1}}
\fi}
}

\newcommand{\suppfor}[2][] 
{
{\ifx&#1&%
  {\mathsf{Forest}_{#2}}
\else
  {\mathsf{For}_{#2}^{#1}}
\fi}
}

\newcommand{\binfor}[1][]{
{\ifx&#1&%
    \operatorname{\mathsf{BinFor}}
\else
    \operatorname{\mathsf{BinFor}}^{#1}
\fi}
} 

\newcommand{\hqsym}[2][]{
{\ifx&#1&%
  {\operatorname{HQSym}_{#2}}
\else
  {\operatorname{HQSym}_{#2}^{#1}}
\fi}
} 
\newcommand{\fl}[1]{\mathrm{Fl}_{#1}}
\newcommand{\coinv}[1]{\operatorname{Coinv}_{#1}} 

\newcommand{\qscoinv}[2][]{
{\ifx&#1&%
  {\operatorname{QSCoinv}_{#2}}
\else
  {{}^{#1}\!\operatorname{QSCoinv}_{#2}}
\fi}
}

\definecolor{ao}{rgb}{0.0, 0.5, 0.0}
\newcommand{\sfc}{\mathsf{c}}

\newcommand{\rt}{\Omega}

\newcommand{\reseq}{\mathrm{RESeq}}

\newcommand{\hhmp}{\mathrm{QFl}}

\newcommand{\nfor}{\operatorname{NestFor}} 
\newcommand{\mbnfor}{\operatorname{mBNestFor}} 
\newcommand{\mnfor}{\operatorname{mNestFor}} 

\newcommand{\mrefor}{\operatorname{mREFor}} 

\newcommand{\wt}[1]{\widetilde{#1}} 
\newcommand{\wh}[1]{\widehat{#1}} 
\newcommand{\ul}[1]{\underline{#1}} 
\newcommand{\rletter}[1]{\mathsf{r}_{#1}}

\newcommand{\eletter}[1]{\mathsf{e}_{#1}}

\newcommand{\tl}{\textbf{t}}
\newcommand{\xl}{\textbf{x}}

\newcommand{\ev}{\operatorname{ev}}

\newcommand{\NC}{\operatorname{NC}}
\newcommand{\ForToNC}{\operatorname{ForToNC}}
\newcommand{\sylv}[1]{\operatorname{Syl}(#1)} 

\newcommand{\eope}[1]{{\mathsf{E}_{#1}}}

\newcommand{\wteope}[1]{\wt{\mathsf{E}}_{#1}}

\newcommand{\poltx}{\ZZ[\tl][\xl]} 
\newcommand{\weight}{\operatorname{wt}} 
\newcommand{\mc}[1]{\mathcal{#1}}

\newcommand{\heap}[1]{\mathsf{Heap}(#1)}
\newcommand{\rednc}[1]{\operatorname{Red}_{\NC}(#1)} 
\newcommand{\invnc}[1]{\operatorname{Inv}_{\NC}(#1)} 
\newcommand{\ncrmin}[1]{\longword^{#1}_{\le vert}}
\newcommand{\sylmin}[1]{\longword^{#1}_{vert}}
\newcommand{\sylcont}[1]{\longword^{#1}_{\ge vert}}

\newcommand{\padded}[1]{\operatorname{Pad}_{#1}} 

\renewcommand\emph[1]{\textcolor{blue}{\textit{#1}}} 

\newcommand{\wnode}
{
\begin{tikzpicture}[baseline=-3,scale=1]
		\node at (0,0) {{$\wedge$}};
		\draw[fill=white] (0,.1) circle (.05);
\end{tikzpicture}
}
\newcommand{\bnode}
{
\begin{tikzpicture}[baseline=-3,scale=1]
		\node at (0,0) {{$\wedge$}};
		\draw[fill=black] (0,.1) circle (.05);
\end{tikzpicture}
}

\title{Equivariant quasisymmetry and noncrossing partitions}

\author{Nantel Bergeron}
\address{Dept. of Math. and Stat., York University, Toronto, Ontario M3J 1P3, Canada}
\email{\href{mailto:bergeron@yorku.ca}{bergeron@yorku.ca}}

\author{Lucas Gagnon}
\address{Dept. of Math. and Stat., York University, Toronto, Ontario M3J 1P3, Canada}
\email{\href{mailto:lgagnon@yorku.ca}{lgagnon@yorku.ca}}

\author{Philippe Nadeau}
\address{Universite Claude Bernard Lyon 1, CNRS, Ecole Centrale de Lyon, INSA Lyon, Université Jean Monnet, ICJ UMR5208, 69622 Villeurbanne, France}
\email{\href{mailto:nadeau@math.univ-lyon1.fr}{nadeau@math.univ-lyon1.fr}}

\author{Hunter Spink}
\address{Department of Mathematics,
University of Toronto, Toronto, ON M5S 2E4, Canada}
\email{\href{mailto:hunter.spink@utoronto.ca}{hunter.spink@utoronto.ca}}

\author{Vasu Tewari}
\address{Department of Mathematical and Computational Sciences, University of Toronto Mississauga, Mississauga, ON L5L 1C6, Canada}
\email{\href{mailto:vasu.tewari@utoronto.ca}{vasu.tewari@utoronto.ca}}

\thanks{
NB and LG were supported by the Natural Sciences and Engineering Research Council of Canada (NSERC) and York Research Chair in Applied Algebra.
PN was partially supported by French ANR grant ANR-19-CE48-0011 (COMBIN\'E). HS and VT acknowledge the support of the NSERC, respectively [RGPIN-2024-04181] and [RGPIN-2024-05433].}

\begin{document}

\begin{abstract}
We introduce a definition of ``equivariant quasisymmetry'' for polynomials in two sets of variables. 
Using this definition we define quasisymmetric generalizations of the theory of double Schur and double Schubert polynomials that we call double fundamental polynomials and double forest polynomials, where the subset of ``noncrossing partitions'' plays the role of $S_n$. 
In subsequent work we will show this combinatorics is governed by a new geometric construction we call the ``quasisymmetric flag variety'' which plays the same role for equivariant quasisymmetry as the usual flag variety plays in the classical story.
\end{abstract}

\maketitle
\setcounter{tocdepth}{1}
\tableofcontents

\section{Introduction}

Recall that a polynomial $f(x_1,\ldots,x_n)$ is \emph{quasisymmetric} if the coefficient of a monomial $x_{i_1}^{a_1}\cdots x_{i_k}^{a_k}$ is the same as $x_1^{a_1}\cdots x_k^{a_k}$ for all sequences $i_1<\cdots < i_k$, and that the quasisymmetric polynomials form a subring $\qsym{n}\subset  \ZZ[x_1,\ldots,x_n]$. The following definition is new and the crux of this paper. We say that a polynomial in two sets of variables $f(x_1,\ldots,x_n;\tl)$ with $\tl=(t_1,t_2,\ldots)$ is \emph{equivariantly quasisymmetric} if for $1\le i \le n-1$ we have
$$f(x_1,\ldots,x_{i-1},t_i,x_{i+1},x_{i+2},\ldots,x_{n};\tl)=f(x_1,\ldots,x_{i-1},x_{i+1},t_i,x_{i+2},\ldots,x_{n};\tl).$$
This is a restricted form of variable symmetry, and setting all $t_i=0$ recovers an equivalent  reformulation of the definition of quasisymmetric polynomials that we recall in \Cref{sec:what_is_equivariant_quasisymmetry?}. 
We write $\eqsym{n}$ for the set of equivariantly quasisymmetric polynomials, which we will show is a ring with a $\ZZ[\tl]$-basis of what we call \emph{double fundamental quasisymmetric polynomials} $\slide{c}(x_1,\ldots,x_n;\tl)$, indexed by padded compositions $c=0^{\ell}a_1\cdots a_{n-\ell}$ with $a_i>0$.
This family specializes when $\tl=0$ to the well-known basis $\slide{c}(x_1,\ldots,x_n)$ for $\qsym{n}$ of fundamental quasisymmetric polynomials.\footnote{Our double quasisymmetric polynomials are unrelated to work of Pechenik--Satriano \cite{PeSa23}.}

While a polynomial that is symmetric in $x_1,\ldots,x_n$ takes a single value at any specialization of the $x$-variables $x_{i} = t_{\sigma(i)}$ 
for a
$\sigma\in S_n$, an equivariantly quasisymmetric polynomial need only take the same values at $x_{i} = t_{\sigma(i)}$ for a \emph{noncrossing permutation} $\sigma\in \NC_n$.  
Noncrossing permutations are obtained by applying a backwards cycle to each block of a noncrossing partition of $\{1,\ldots,n\}$. 
For example, when $n=3$ we have
 \[
 f(t_3,t_2,t_1;\tl)=f(t_3,t_1,t_2;\tl)=f(t_1,t_3,t_2;\tl)=f(t_1,t_2,t_3;\tl)=f(t_2,t_1,t_3;\tl).
 \]
Notably, we are unable to reach $f(t_2,t_3,t_1;\tl)$ by applying (quasi)symmetries, and the corresponding permutation is the only element of $S_3\setminus \NC_3$.

In this paper we develop a combinatorial theory for equivariant quasisymmetry which parallels the theory of double Schur polynomials $s_\lambda(x_1,\ldots,x_n;\tl)$, and more generally double Schubert polynomials $\schub{w}(\xl;\tl)$ for $\xl=(x_1,x_2,\ldots)$. 
These will be played respectively by the {double fundamental polynomials}
 $\slide{c}(x_1,x_2,\ldots,x_n;\tl)$ and a family of polynomials we call \emph{double forest polynomials} $\forestpoly{F}(\xl;\tl)$ for $F$ a plane binary indexed forest. 
 We will show that the double forest polynomials give a $\ZZ[\tl]$-basis for $\ZZ[\tl][\xl]$
and specialize when $\tl=0$ to the recently introduced forest polynomial basis of $\ZZ[\xl]$ studied in \cite{NST_a, NT_forest}.

In the sequel to this paper \cite{BGNST_Toappear} we will use the combinatorial theory developed here as a backbone to construct a \emph{quasisymmetric flag variety} $\hhmp_n$ that plays the same role for the double forest polynomials and $\NC_n$ as the usual flag variety does for double Schubert polynomials and $S_n$. 
One of the culminating results of the sequel \cite{BGNST_Toappear} will be that if we define $\qsymide{n}\coloneqq \langle f(x_1,\ldots,x_n)-f(0,\ldots,0)\suchthat f\in \qsym{n}\rangle$, then the quasisymmetric flag has cohomology ring
$$H^\bullet(\hhmp_n)\cong \qscoinv{n}\coloneqq\ZZ[x_1,\ldots,x_n]/\qsymide{n},$$
the ring of \emph{quasisymmetric coinvariants}. 
This exactly parallels the fact that the complete flag variety $\fl{n}$ has cohomology ring the analogously defined symmetric coinvariants, and answers in the best possible way the question of finding a geometric model for the ring of quasisymmetric coinvariants. 
See \Cref{subsec:PriorWork} for further discussion.

\subsection{Double forest polynomials}
\label{subsec:double_forests_intro}

Let $S_{\infty}$ be the group of permutations of $\{1,2,\ldots\}$ that fix all but finitely many elements. 
This group is generated by the adjacent transpositions $s_i=(i,i+1)$, and contains the finite symmetric groups $S_n=\langle s_1,\ldots,s_{n-1}\rangle\subset S_{\infty}$. 
The group $S_{\infty}$ acts on polynomials $\ZZ[\tl][\xl]$ by permuting the $\xl$-variables $w\cdot (x_1,x_2,\ldots)=(x_{w(1)},x_{w(2)},\ldots)$, and the divided difference operation is the endomorphism of $\ZZ[\tl][\xl]$ defined by $$\partial_if=\frac{f-s_i\cdot f}{x_i-x_{i+1}}.$$
This operation characterizes symmetry in the $\xl$-variables, as a polynomial $f(x_1,\ldots,x_n;\tl)$ is symmetric in $x_1,\ldots,x_n$ if and only if $\partial_1f=\cdots=\partial_{n-1}f=0$. 
Furthermore, these operations satisfy the defining relations of the nil-Hecke algebra $\partial_i^2=0$, $\partial_i\partial_j=\partial_j\partial_i$ for $|i-j|\ge 2$, and $\partial_i\partial_{i+1}\partial_i=\partial_{i+1}\partial_i\partial_{i+1}$, which implies that the composite operations $\partial_w=\partial_{i_1}\cdots \partial_{i_{\ell(w)}}$ for any reduced word factorization $w=s_{i_1}\cdots s_{i_{\ell(w)}}$ are well defined.

The double Schubert polynomials $\{\schub{w}(\xl;\tl)\suchthat w\in S_{\infty}\}\subset \ZZ[\tl][\xl]$ of Lascoux--Sch\"utzenburger \cite{LS82} are the unique family of homogenous polynomials satisfying the normalization condition $\schub{w}(\tl;\tl)=\delta_{w,\idem}$ and the recursion
$\partial_i \schub{w}(\xl;\tl)=\delta_{\ell(ws_i)<\ell(w)}\schub{ws_i}$. 
They are equivalently characterized as the unique family of homogenous polynomials such that $(\partial_w\schub{w'})(\tl;\tl)=\delta_{w,w'}$. 
Classically, these polynomials are constructed by the ansatz $\schub{w_{0,n}}(\xl;\tl)=\prod_{i+j\le n}(x_i-t_j)$, where $w_{0,n}\in S_n$ is the longest element, and deriving all other Schubert polynomials by applying divided differences. They can also be defined via a subword model inducing the graphical ``pipe dreams'' ~\cite{Ber93, BJS93, FK96, KM05}, and directly shown to satisfy the recursive characterization from there.

In this paper we define a new operation $$\eope{i}f=\frac{f(x_1,\ldots,x_{i-1},x_i,t_{i},x_{i+1},\ldots,x_{n-1};\tl)-f(x_1,\ldots,x_{i-1},t_i,x_i,x_{i+1},\ldots,x_{n-1};\tl)}{x_i-t_i}$$
we call the \emph{equivariant quasisymmetric divided difference}. This operation characterizes equivariant quasisymmetry, as a polynomial $f(x_1,\ldots,x_n;\tl)$ is equivariantly quasisymmetric if and only if $\eope{1}f=\cdots=\eope{n-1}f=0$.

Analogously to the descent set $\des{w}$ which consists of those $i$ for which $w\mapsto ws_i$ decreases length, indexed forests have a notion $\qdes{F}$ which consists of the labels of the terminal nodes, and we have an analogous ``trimming'' operation $F\mapsto F/i$ to a forest with one less internal node \cite[\S 3]{NST_a}.
The double forest polynomials will be axiomatically characterized in terms of their interaction with $\eope{i}$ by requiring the normalization condition $\forestpoly{F}(\tl;\tl)=\delta_{F,\emptyset}$, and $\eope{i}\forestpoly{F}(\xl;\tl)=\delta_{i\in \qdes{F}}\forestpoly{F/i}(\xl;\wh{\tl}_{i})$ where $\wh{\tl}_{i}=(t_1,\dots,t_{i-1},t_{i+1},\ldots)$. 
We note in particular that this recursion asserts that $\eope{i}\forestpoly{F}$ does not depend on the variable $t_i$. 
This recursion is subtle; for example, if $j\in \qdes{F/i}$ then $\eope{j}\eope{i}\forestpoly{F}$ is not necessarily a variable transformation of $\forestpoly{(F/i)/j}$ due to the $\tl$-variable reindexing in $\eope{i}\forestpoly{F}$. We give some examples of applications of this operation to forest polynomials in \Cref{fig:eq_trims}.

\begin{figure}[!h]
    \centering
    \includegraphics[width=\linewidth]{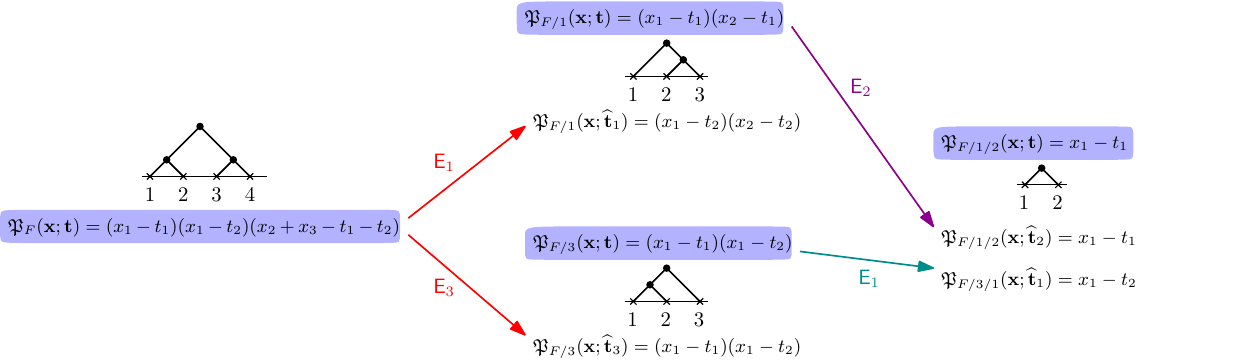}
    \caption{Applying $\eope{i}$ operations for the forest $F$ with code $c(F)=(2,0,1,0,\ldots)$}
    \label{fig:eq_trims}
\end{figure}

\begin{maintheorem}[\Cref{thm:ForestDesiderata},  \Cref{sec:subword_models}]
\label{thm:thmA}
    Double forest polynomials $\forestpoly{F}(\xl;\tl)$ exist, and are computed by a subword model that is graphically represented by certain \emph{vine diagrams}.
\end{maintheorem}
 
 We also show our subword model and vine diagrams can also be adapted to compute double Schubert polynomials. These models do not appear to specialize to pipe dreams.

\begin{rem}
    As noted in \cite{NST_a}, it appears that no simple ansatz works even for ordinary forest polynomials. As we will see, the existence of vine diagrams follows from the axiomatic definition of forest polynomials assuming they do exist, and we will directly verify the axiomatic recursion on these polynomials.
\end{rem}

\subsection{Evaluations at noncrossing permutations}

Given a polynomial $f(\xl;\tl)$, we can consider its  specializations $\ev_wf\coloneqq f(\tl_{w};\tl)$ at permutations of the $t$-variables $\tl_{w}=(t_{w(1)},t_{w(2)},\ldots)$. The double Schubert polynomials satisfy an upper-triangular property with respect to these evaluations under the Bruhat order
$$\ev_\sigma\schub{w}=\begin{cases}0&\sigma \not \ge w\\\displaystyle\prod_{i<j\text{ and }w(i)>w(j)}(t_{w(i)}-t_{w(j)})& \sigma=w\\
\text{Formula of AJS--Billey \cite{AJS,Bil99}}&\sigma>w.\end{cases}$$
 Furthermore, the formula of AJS-Billey is always nonzero and so this shows that double Schubert polynomials evaluations characterize the Bruhat order via 
 \[ 
 w\le \sigma \Longleftrightarrow \ev_{\sigma}\schub{w}\ne 0.
 \]

We show that double forest polynomials behave analogously with respect to $\NC_n$. We construct a bijection $\ForToNC:\suppfor{n}\to \NC_{n}$ from forests supported on $\{1,\ldots,n\}$ to noncrossing partitions
such that the following is true.
\begin{maintheorem}[\Cref{sec:evaluations}]
\label{thm:mainC}
For $\sigma\in \NC_{n}$ and $F\in \suppfor{n}$ we have
    $$\ev_\sigma \forestpoly{F}=\begin{cases}0&\sigma \not \ge \ForToNC(F)\\\displaystyle\prod_{i<j\text{ and }\sigma(i)>\sigma(j)\text{ and }\sigma(i,j)\in \NC_n}(t_{\sigma(i)}-t_{\sigma(j)})& \sigma=\ForToNC(F)\\\text{Formula in \Cref{thm:AJS_Billey_for_Forests}}&\sigma>\ForToNC(F).\end{cases}$$
\end{maintheorem}
We will furthermore show that the formula in \Cref{thm:AJS_Billey_for_Forests} is always nonzero, and so double forest polynomials evaluations characterize the Bruhat order restricted to $\NC_n$ via 
\[
\ForToNC(F)\le \sigma \Longleftrightarrow \ev_{\sigma}\forestpoly{F}\ne 0.
\]


\subsection{$\star$-compositions and combinatorial positivity}

Under a modified notion of composition we call \emph{$\star$-composition}, the $\eope{i}$ operations satisfy the defining relations of the positive Thompson monoid
$$[\eope{i}\star \eope{j}]=[\eope{j}\star \eope{i+1}]\text{ for }i>j,$$
meaning that the theory of their $\star$-compositions is governed by the monoid structure on $\indexedforests$. 
From these $\star$-compositions we can build algebraic operations $[\ev\star \eope{F}]\colon\ZZ[\tl][\xl]\to \ZZ[\tl]$ with the property that $[\ev \star \eope{F}]\forestpoly{G}=\delta_{F,G}$. This will in particular imply for any $f(\xl;\tl)\in \ZZ[\tl][\xl]$ that
$$f(\xl;\tl)=\sum_{F\in \indexedforests}a_F(\tl)\forestpoly{F}(\xl;\tl)\text{ where }a_F(\tl)=[\ev\star \eope{F}]f.$$
Say that $a(\tl)\in \ZZ[\tl]$ is \emph{Graham-positive} if it lies in $\ZZ_{\ge 0}[t_2-t_1,t_3-t_2,\ldots]$. First studied by Graham \cite{Gra01}, this is the natural torus-equivariant notion of positivity in ``type $A$'', specializing when $\tl=0$ to the usual notion of positivity. Using these operations we are able to show the following result.
\begin{maintheorem}[\Cref{thm:schubertexpandpos}, \Cref{thm:forestmultpos}]\label{thm:mainB}
    Double Schubert polynomials have a decomposition
    $$\schub{w}(\xl;\tl)=\sum_{F\in \indexedforests}a_F^w(\tl) \forestpoly{F}(\xl;\tl)\text{ with }a_F^w(\tl)\text{ Graham positive},$$
    and for forests $F,G$ we have
    $$\forestpoly{F}(\xl;\tl)\forestpoly{G}(\xl;\tl)=\sum_{H\in \indexedforests}a^H_{F,G}(\tl)\forestpoly{H}(\xl;\tl)\text{ with }a^H_{F,G}(\tl)\text{ Graham positive}.$$
    Furthermore, the Graham-positivity is realized by an explicit combinatorial algorithm for computing the coefficients.
\end{maintheorem}
We note that as a special case of this result we obtain a Graham-positive decomposition of double Schur polynomials into double fundamental quasisymmetric polynomials, as well as a Graham-positive multiplication rule for double fundamental quasisymmetric polynomials. 
Although our algorithm is manifestly Graham-positive, we understand almost nothing about how it behaves in even the simplest cases. 
We thus pose the following tasks.
\begin{question}
\begin{enumerate}[align=parleft,left=0pt,label=(\arabic*)]
    \item Give an explicit combinatorial interpretation for the $a^w_F(\tl)$.
    \item As a special case, give an explicit combinatorial interpretation for the decomposition of double Schur polynomials into double fundamental quasisymmetric polynomials.
    \item Determine the ``equivariant shuffle product rule'' for multiplying double fundamental quasisymmetric polynomials.
    \item As a special case, determine an ``equivariant Monk's rule'' corresponding to multiplication by $\slide{0^{n-1}1}=x_1+\cdots+x_n-t_1-\cdots-t_n$.
\end{enumerate}
\end{question}

The interested reader can find tables containing the values from (1)-(4) at the end of the paper.


\subsection{Prior work}
\label{subsec:PriorWork}
This work has a number of antecedents. By work of Borel \cite{Bor53}, the complete flag variety $\fl{n}$ has cohomology ring given by the symmetric coinvariants
$$H^\bullet(\fl{n})=\coinv{n}\coloneqq \ZZ[x_1,\ldots,x_n]/\symide{n},$$
where $\symide{n}\coloneqq \langle f(x_1,\ldots,x_n)-f(0,\ldots,0)\suchthat f\in \sym{n}\rangle$.
Subsets of Schubert polynomials form bases both for $\coinv{n}$ and $\symide{n}$, and the interplay between the geometry of the flag variety and the combinatorics of Schubert polynomials is one of the central themes in algebraic combinatorics.

Inspired by this, in \cite{ABB04} Aval--Bergeron--Bergeron introduced the ring of quasisymmetric coinvariants and showed that
$$\operatorname{rank}(\qscoinv{n}^{(i)})=\frac{n-i}{n+i}\binom{n-i}{i}$$
for $i=0,\ldots,n-1$. For fixed $n$ this sequence is increasing and in particular fails to be symmetric, which rules out the possibility of a smooth projective variety $X$ for which the quasisymmetric coinvariants may be realized as $H^\bullet(X)$. Later work of the third and fifth authors \cite{NT_forest} showed that forest polynomials give bases for $\qsymide{n}$ and $\qscoinv{n}$, analogously to how Schubert polynomials give bases for $\symide{n}$ and $\coinv{n}$. In \cite{NST_a} the third, fourth, and fifth authors developed a combinatorial theory for forest polynomials based on the trimming operation $\tope{i}$ (obtained by setting $\tl=0$ in $\eope{i}$) which paralleled the divided difference theory of Schubert polynomials. This was used by the same authors in \cite{NST_c} to show that there was a geometric theory paralleling the combinatorial theory based on certain toric Richardson varieties forming a toric complex we called $\operatorname{HHMP}_n$, and it was shown using a spectral sequence argument that
$$\qscoinv{n}\subset H^\bullet(\operatorname{HHMP}_n).$$ In the sequel \cite{BGNST_Toappear} we will show that to have the geometry match the equivariant combinatorial theory (and to have an equality of the cohomology ring with the quasisymmetric coinvariants) it is necessary to work with a new much more combinatorially complicated toric complex that we will call the quasisymmetric flag variety $\hhmp_n$, formed as the union of the $X(F)$ varieties that were first introduced in \cite{NST_c} as canonical translations of the Richardsons appearing in $\hhmp$. It is a primary goal of this paper to supply the necessary combinatorics to carry this out.

Finally, in work of the first two authors \cite{BeGa23} it was shown that if we consider the ``orbit harmonics'' associated to the set of points $$\{(\sigma(1),\ldots,\sigma(n))\suchthat \sigma\in \NC_n\}\subset \QQ^n,$$
i.e. the cokernel of the map $\QQ[x_1,\ldots,x_n]\to \QQ^{|\NC_n|}$ taking $f\mapsto (f_{\sigma})_{\sigma\in \NC_n}$ then we recover a ring whose associated graded is isomorphic to the quasisymmetric coinvariants with $\QQ$-coefficients. 
These results, when generalized appropriately to our context, will be established in the sequel \cite{BGNST_Toappear} using \Cref{thm:mainC}, and will correspond geometrically to the determination of the $T$-equivariant cohomology ring $H^\bullet_T(\hhmp_n)$ via the combinatorial graph cohomology ring of Goresky--Kottwitz--MacPherson \cite{GKM98}.

\subsection{Outline of the article}
In \Cref{sec:what_is_equivariant_quasisymmetry?} we define the notion of equivariant quasisymmetry and compare it to how symmetric polynomials are considered in an equivariant context. 
In \Cref{sec:preliminaries} we go over combinatorial preliminaries needed for the remainder of the article.
Moving beyond preliminaries, the article has three main parts.  

In the first part, we study equivariant quasisymmetry using forests and double forest polynomials.  
In \Cref{sec:equivariant_forest_polynomials} we introduce the axiomatic definition of double forest polynomials. 
In \Cref{sec:subword_models} we create subword models and graphical vine models for double Forest polynomials and prove that double forest polynomials exist using the subword model as an ansatz.
In \Cref{sec:VineSchub} we show that the vine model computes double Schubert polynomials and state some related results.  

In the second part of the paper we use the combinatorics of noncrossing partitions to understand the evaluations of double forest polynomials at noncrossing partitions.  
In \Cref{sec:EQNP} we show that noncrossing partitions characterize equivariant quasisymmetry and give a natural bijection between indexed forests and noncrossing partitions.  
In \Cref{sec:evaluations} we prove the analogue of the AJS--Billey formula for double forest polynomials and noncrossing partitions.  

In the final part of the article we introduce monoids which interpolate between forests and noncrossing partitions in order to give in-depth descriptions of certain structure constants related to equivariant quasisymmetry.  
In \Cref{sec:monoids} we introduce some monoids that are used in the remainder of the article.
In \Cref{sec:coeffextraction} we introduce the $\star$-composition for the $\eope{i}$ operations, and show how to algebraically extract coefficients in double forest polynomial decompositions. 
In \Cref{sec:graham} we show how certain positive straightening rules allow us to verify double Schuberts expand Graham-positively into double forests, and that the structure coefficients for double forest multiplication are Graham-positive.

For the convenience of the reader, we include three tables of computed values at the end of the article.  
Table~\ref{table:forestpolys} contains examples of double forest polynomials.  
Table~\ref{table:Schubtofor} contains the expansion of double Schubert polynomials into double forest polynomials. 
Table~\ref{table:fundamentalprod} gives the multiplicative structure constants for the basis of double fundamental polynomials in $\eqsym{4}$.




\subsection*{Acknowledgements}
We would like to thank Fr\'ed\'eric Chapoton and Allen Knutson for helpful correspondence/conversations.
We are very grateful to the Fields Institute for providing a fantastic work environment.

\section{Equivariant quasisymmetry}
\label{sec:what_is_equivariant_quasisymmetry?}
We will work with polynomials $\ZZ[\tl][\xl]$ in two infinite sets of variables
\[
\xl = \{x_{1}, x_{2}, \ldots\}
\qquad\text{and}\qquad
\tl = \{t_{1}, t_{2}, \ldots \}.
\]
For geometric reasons we call $\tl$ the \emph{equivariant variables} and $\xl$ the \emph{non-equivariant variables}.
Given a nonnegative integer $n$ we shall denote the truncated sets of variables $\{x_1,\dots,x_n\}$ and $\{t_1,\dots,t_n\}$ by $\xl_n$ and $\tl_n$ respectively.

\subsection{Equivariantly quasisymmetric polynomials}
\label{section:equivariant_quasisymmetrics}
Our definition of equivariant quasisymmetry generalizes an algebraic formulation of quasisymmetry that we now recall. For $f\in \ZZ[\xl]$, the $i$th \emph{Bergeron--Sottile} map \cite{BS98,NST_c,NST_a}
\[
\rope{i}f=f(x_1,\ldots,x_{i-1},0,x_i,x_{i+1},\ldots),
\]
makes the substitutions $x_i\mapsto0$ and $x_j \mapsto x_{j-1}$ for all $j\ge i$. 
Then we say that a polynomial $f\in \ZZ[\xl_n]$ is \emph{quasisymmetric} if $\rope{i+1}f=\rope{i}f$ for all $1 \le i < n$.

We denote
    $$\qsym{n}= \{f(\xl_n)\in \ZZ[\xl_n]\suchthat \rope{i+1}f=\rope{i}f\text{ for }1\le i < n\},$$
    the ring of \emph{quasisymmetric polynomials}. 
    This algebraic formulation is equivalent to the usual definition of quasisymmetry for  $f \in \ZZ[\xl_n]$ which says that the coefficients of the monomials $x_1^{a_1}\cdots x_k^{a_k}$ and  $x_{i_1}^{a_1}\cdots x_{i_k}^{a_k}$ are equal for every increasing sequence $(i_1,\ldots,i_k)$ of distinct indices.

\begin{defn}
\label{defn:equivariant_bergeron_sotille}
    For $f(\xl;\tl)\in \ZZ[\tl][\xl]$, we define the \emph{equivariant Bergeron--Sottile maps} to be
    \begin{align*}
    \rope{i}^-f(\xl;\tl)&=f(x_1,\ldots,x_{i-1},t_i,x_i,x_{i+1},\ldots;\tl)\\
    \rope{i}^+f(\xl;\tl)&=f(x_1,\ldots,x_{i-1},x_i,t_i,x_{i+1},\ldots;\tl)
    \end{align*}
     We say that $f\in \ZZ[\tl][\xl_n]$ is \emph{equivariantly quasisymmetric} if 
     $\rope{i}^+f=\rope{i}^-f$ for $1 \le i < n$.
\end{defn}

We denote
\[
\eqsym{n} = \{ f(\xl_n; \tl) \in \ZZ[\tl][\xl_{n}] \;|\;\text{$\rope{i}^+f=\rope{i}^-f$ for $1 \le i < n$}\},
\]
the \emph{equivariantly quasisymmetric polynomials}. 
Note that if we set $\tl=0$ then the condition for equivariant quasisymmetry becomes the algebraic formulation for quasisymmetry, and therefore
$$
f(\xl_n;\tl)\in \eqsym{n} \implies f(\xl_n;\bm{0})\in \qsym{n}.
$$

\begin{thm}
    For all $n \ge 0$, $\eqsym{n}$ is a $\ZZ[\tl]$-algebra.
\end{thm}
\begin{proof}
    The equalizer of the algebra morphisms $\rope{i}^+$ and $\rope{i}^-$ is a $\ZZ[\tl]$-algebra, and so $\eqsym{n}$ is the intersection of these $\ZZ[\tl]$-algebras.
\end{proof}

\subsection{Comparison with equivariantly symmetric polynomials}

Our definition of equivariant quasisymmetry is closely related to what one might call the equivariantly symmetric polynomials.  
Recall the action of $S_{\infty}$ (and $S_n$) on $\ZZ[\tl][\xl]$ from \S\ref{subsec:double_forests_intro}.
Say that an element $f \in \ZZ[\tl][\xl_n]$ is \emph{equivariantly symmetric} if $f = w \cdot f$ for all $w \in S_{n}$.  
For each $n$, the equivariantly symmetric polynomials in $x_{1}, \ldots, x_{n}$ form a subring $\esym{n}$ of $\ZZ[\tl][\xl_n]$.

The defining condition for equivariant quasisymmetry can be rephrased as $\rope{i}^-(f-s_i\cdot f)=0$, which weakens the condition for equivariant symmetry.  
Therefore, we have the containments
\[
\esym{n}\subset \eqsym{n}\subset \ZZ[\tl][\xl_n].
\]

There are also important differences between equivariant symmetry and equivariant quasisymmetry.  
The ordinary symmetric polynomials $\sym{n}$ are exactly the equivariantly symmetric polynomials in the subring $\ZZ[\xl_n]$, and moreover $\esym{n} = \ZZ[\tl] \otimes \sym{n}$.  
In contrast, ordinary quasisymmetric polynomials are typically not even contained in $\eqsym{n}$, except when $n=0$ or $n=1$.

\begin{eg}
The polynomial $f=x_1^2x_2+x_1^2x_3+x_2^2x_3$ lies in $\qsym{3}$, but is not equivariantly quasisymmetric:
\[
\rope{1}^{+}(f) = t_1x_1^2+x_1^2x_2+t_1^2x_2
\qquad\text{and}\qquad
\rope{1}^{-}(f) = t_1^2x_1+t_1^2x_2+x_1^2x_2.
\]
\end{eg}

We will see in \Cref{cor:basiscor} that every element of $\qsym{n}$ can be homogenously deformed using the equivariant variables to produce an element of $\eqsym{n}$. For example, one such deformation of this polynomial would be
\begin{multline*}
(x_2^2x_3 + x_1^2x_3 + x_1^2x_2) - (x_2x_3+x_1x_3+x_1x_2)t_2  - (x_1^2+x_2^2)t_2 - x_1^2t_1 \\  + (x_1+x_2+x_3)t_1t_2 + (x_2 + x_1)t_2^2 - (x_3 + x_2)t_1^2 - t_1t_2^2 + t_1^3.
\end{multline*}
Note that for example the coefficient of $t_1^2$ is not quasisymmetric, so $\eqsym{3}\not\subset \ZZ[\tl]\otimes \qsym{3}$.

\section{Combinatorial preliminaries}
\label{sec:preliminaries}
Throughout we set $[n]\coloneqq \{1,\dots,n\}$ for all nonnegative integers $n$. 
We let $\NN$ denote the set of positive integers. In this section we introduce our main combinatorial objects.

\subsection{Binary trees and indexed forests}
\label{subsec:trees_and_forest}
We quickly recall several relevant notions in the context of our primary combinatorial object: indexed forests.
In order to keep the exposition brief, we refer the reader to \cite{NST_c, NST_a} for a more detailed investigation of the associated combinatorics and only record the facts that we shall need.

A \emph{binary plane tree} is a rooted tree $T$ in which each node $v$ is an \emph{internal node} with exactly $2$ children $v_L$ and $v_R$ (the left and right child), or $v$ is a \emph{leaf} with zero children. 
Going forward, all trees will be binary plane trees, so we shall omit these qualifiers. We write $\internal{T}$ for the set of internal nodes of $T$, and we define the size of a tree to be $|T|\coloneqq |\internal{T}|$. For $v\in \internal{T}$, we call the descendants of $v_{L}$ the \emph{left descendants} of $v$ and the descendants of $v_{R}$ the \emph{right descendants} of $v$.

We write $\ast$ for the trivial tree with one node. This node is both the root and a leaf, and $\internal{\ast}=\emptyset$.

\begin{defn}\label{def:indexed_forest}
    An \emph{indexed forest} is an infinite sequence  $F = (T_{1}, T_{2}, \ldots)$ of binary trees where all but finitely many of the trees are $\ast$.
    We write $\indexedforests$ for the set of all indexed forests.
\end{defn}

Given Definition~\ref{def:indexed_forest}, we extend our notation and terminology for binary trees to indexed forests.  For $F = (T_{1}, T_{2}, \ldots)$, the internal nodes of $F$ are
\[
\internal{F}=\bigcup_{i=1}^{\infty}\internal{T_i},
\]
and we write $|F|=|\internal{F}|$.  Similarly, the leaves of $F$ are the union of the leaves of the $T_{i}$.  
We say that $v$ is \emph{terminal} if $v_{L}$ and $v_{R}$ are leaves. 
The forest with all constituent trees trivial is called the \emph{empty} forest and denoted by $\emptyset$. 

We use distinct labeling conventions for the leaves 
and the internal vertices of an indexed forest $F = (T_{1}, T_{2}, \ldots)$.  
The leaves of $F$ will be absolutely identified with $\NN$ from left to right, so that $T_{1}$ has leaves $1$ through $|T_{1}|+1$, $T_{2}$ has leaves $|T_{1}| + 2$ through $|T_{1}| + |T_{2}| + 2$, and so on.  
The \emph{canonical label} of $v \in \internal{F}$ will be the value of the rightmost leaf descendant of $v_L$.

We define its \emph{support} $\supp(F)$ to be the set of leaves in $\NN$ that appear in the nontrivial trees in $F$, and for fixed $n\ge 1$ we denote the subset of forests supported on $[n]$ by 
\begin{align*}
\suppfor{n}=\{F\in \indexedforests\suchthat \supp(F)\subset [n]\}.
\end{align*}

A \emph{Sylvester word} of a forest $F\in \indexedforests$ is an ordered sequence comprising the canonical labels of the vertices in $\internal{F}$ such that the label of $v\in \internal{F}$ appears later than the labels of its children $v_L,v_R$.
Let
\[
\sylv{F} = \{\text{Sylvester words for $F$}\}.
\]
See Figure~\ref{fig:indexed_forest_sylvester_words} for an example.
\begin{figure}[!h]
    \centering
    \includegraphics[width=\linewidth]{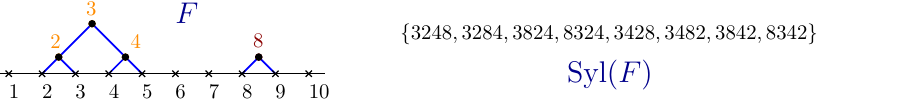}
    \caption{An indexed forest $F$ and its associated Sylvester words}
    \label{fig:indexed_forest_sylvester_words}
\end{figure}

A word $w$ in the alphabet $\NN$ is \emph{injective} (in the sense of \cite{BW83}) if each of its letters are distinct. We write $\supp{w}\subset \NN$ for the set of letters that $w$ is a permutation of. Every finite injective word $w$ is a Sylvester word for a unique indexed forest, which is determined as follows.

For $w$ a permutation of a finite linear order $I$, we define the \emph{binary search tree} of $w$ to be the labeled binary tree $T(w)$ obtained by the following insertion procedure.
If $w$ is empty, then $T(w)$ is the empty tree.  
So suppose $w=w'a$ and let $T'\coloneqq T(w')$. 
If $T'$ is empty, then $T(w)$ is the tree with a single node labeled $a$.
If $T'$ is not empty, let $b$ be the label of its root node. 
Now $T(w)$ is recursively obtained by inserting $a$ in the right (resp. left) subtree of $T'$ if $a>b$ (resp. $a<b$).
When this procedure terminates, $a$ is necessarily the label of a terminal node.

For $w$ an injective word, decompose $\supp{w}=I_1\sqcup \cdots \sqcup I_k$ as a disjoint union of maximal disjoint contiguous intervals in $\NN$, and let $w_1,\ldots,w_k$ be the associated permutations of $I_1,\ldots,I_k$ as they appear in $w$ (so that $w$ is a shuffle of the letters of $w_1,\ldots,w_k$).
Then $w \in \sylv{F}$, where $F\in \indexedforests$ has its nontrivial trees $T(w_1), \dots, T(w_k)$ with canonical labels that agree with the associated binary search labelings. 
For example the injective word $w=3482$ has $w_1=342$ and $w_2=8$, and these insert to the binary trees constituting the forest $F$ in \Cref{fig:indexed_forest_sylvester_words}, showing that $3482\in \sylv{F}$.

\begin{rem}
\label{rem:NewSylv}
Our choice of terminology is inspired by Hivert--Novelli--Thibon's  Sylvester congruence \cite[\S 3]{HNT05}.  
This is an equivalence relation on injective words\footnote{It is defined more generally for words but it suffices to restrict to this setting for our purposes.} generated by the relation $w_1\, b\, w_2\, ac\, w_3\sim w_1\, b\, w_2\, ca\, w_3$ whenever  $a<b<c$. 
One of the main results in \cite{HNT05} states that the binary search tree of an injective word $w$, treated as a permutation of $\supp{w}$ treated as a linear sub-order of $\NN$ uniquely determines its equivalence class in the Sylvester congruence.  
Our sets $\sylv{F}$ partition the set of all injective words into coarser equivalence classes under the modified version of the Sylvester relation that
\begin{align}
\label{eqn:forsylvrelations}
w_1acw_2\sim w_1caw_2\text{ whenever there exists some $a<b<c$ with $b\not\in w_2$}.
\end{align} 
The additional equivalences we allow are when there is such a $b$ between $a$ and $c$ which is not present in the word at all, in which case $a,c$ are in distinct trees of the indexed forest and we allow them to freely commute.
\end{rem}

Indexed forests can be encoded by sequences of nonnegative numbers.
We let $\nvect$ denote the set of sequences $(c_i)_{i\geq 1}$ of nonnegative integers where all but finitely many entries are $0$.
For $F\in \indexedforests$ we define the \emph{flag} $\rho_F\colon \internal{F}\to \NN$ by setting $\rho_F(v)$ to be  the label of the leaf obtained by going down left edges starting from $v$.
The \emph{code} of $F$, denoted by $\sfc(F)$, is defined as
\begin{align*}
    \sfc(F)=(c_i)_{i\in \NN} \text{,\ \ \  where } c_i=|\{v\in \internal{F}\suchthat \rho_F(v)=i\}|.
\end{align*}
This gives a bijection $\sfc\colon \indexedforests\to \nvect$.
The flag also allows us to define another relevant notion. 
For $F\in \indexedforests$, let
\begin{align*}
    \qdes{F}\coloneqq \{\rho_F(v)\suchthat v \text{ a terminal node in } F\}.
\end{align*}
Of particular interest is the case where $\qdes{F}\subset \{n\}$ for some $n\in \NN$. 
We call such an $F$ a \emph{zigzag forest}, and write $\zigzag{n}$ for the collection of all such forests.
Figure~\ref{fig:zigzag} depicts a zigzag forest in $\zigzag{6}$. Observe that a zigzag forest necessarily has a unique Sylvester word, and it is obtained by reading the canonical labels starting at the root node and ending at the terminal node.
\begin{figure}
    \centering
    \includegraphics[width=0.75\linewidth]{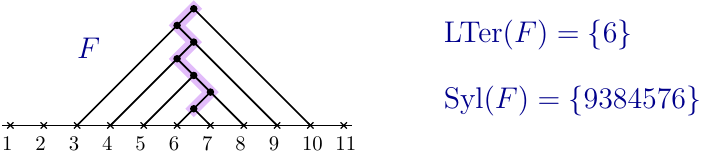}
    \caption{A zigzag forest in $\zigzag{6}$}
    \label{fig:zigzag}
\end{figure}

We define the \emph{$i$th elementary forest} $\underline{i}$ as shown in Figure~\ref{fig:ith_elementary}.
It is the unique indexed forest of size $1$ with $i\in \qdes{F}$.
The elementary forests generate the monoid on $\indexedforests$ \cite[\S 4]{NST_a} in which $F\cdot G$ is the forest obtained by attaching, for each $i$, the $i$th leaf of $F$ to the $i$th root of $G$. 
This monoid is isomorphic to the \emph{Thompson monoid}
\[
\Th\coloneqq\langle 1,2,\ldots \suchthat i\cdot j=j\cdot (i+1)\text{ for all }i>j\rangle,
\]
under the map $i\mapsto \underline{i}$, and we will always make this identification in what follows. Every element $F$ of $\Th$ has a unique representative $F=1^{a_1}\cdot 2^{a_2}\cdots$ where the generators are multiplied in weakly increasing order, and $c(F)=(a_1,a_2,\ldots)$.

\begin{figure}[!h]
    \centering
    \includegraphics[scale=1]{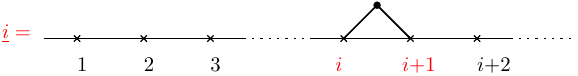}
    \caption{The $i$th elementary forest}
    \label{fig:ith_elementary}
\end{figure}

Given $F\in\indexedforests$ and $i\in \qdes{F}$, we define the \emph{trimmed forest} $F/i$ to be the indexed forest obtained by deleting the terminal node $v\in \internal{F}$ satisfying $\rho_F(v)=i$. Equivalently, $F/i$ is the unique forest such that $F=(F/i)\cdot i$. As shown in \cite{NST_a} if $c=c(F)$ then  $\qdes{F}=\{i\suchthat c_i>0\text{ and }c_{i+1}=0\}$, and for $i\in \qdes{F}$ we have $c(F/i)=(c_1,\ldots,c_{i-1},c_i-1,c_{i+2},c_{i+3},\ldots)$.

The analogue of reduced words for forests are the \emph{trimming sequences}, defined by $$\Trim{F}=\{(i_1,\ldots,i_{|F|})\suchthat F=i_1\cdots i_{|F|}\}.$$

While the trimming sequences for $F/i$ are obtained by deleting the last letter from all trimming sequences of $F$ that end in $i$, the analogous fact for Sylvester words involves a nontrivial shift.
\begin{obs}
\label{obs:SylvesterTrimming}
    If $F\in \indexedforests$ and $j\in \qdes{F}$, then the Sylvester words for $F/j$ are obtained by deleting the last letter from each Sylvester word for $F$ which ends in $j$ and replacing $k\mapsto k-\delta_{k>j}$.
\end{obs}

\subsection{Noncrossing partitions and noncrossing permutations}
\label{subsec:noncrossing_stuff}

A \emph{set partition} $\mc{P}$ of $[n]$ is a collection of disjoint nonempty sets $\{ B_1, B_{2}, \ldots, B_{k}\}$---called \emph{blocks}---whose union is $[n]$.  
We depict set partitions as \emph{arc diagrams}, placing elements of $[n]$ along the positive $x$-axis and connecting sequential elements of each block with an arc above the axis; see Example~\ref{eg:noncrossing}.
A set partition $\mc{P}$ of $[n]$ is \emph{noncrossing} if, for every pair $P, Q$ of distinct blocks of $\mc{P}$ with $a,b\in P$ and $c,d\in Q$, we do not have $a<c<b<d$; this condition ensures that no arcs in the arc diagram cross.
We denote the set of noncrossing set partitions $\ncpart_n$. 

Each noncrossing partition $\mc{P}$ of $[n]$ determines a unique \emph{noncrossing permutation} $\sigma(\mc{P})\in S_{n}$ as follows.  First write $\mc{P}$ as $B_1/B_2/\cdots/B_k$ with the convention that $\max(B_1)<\max(B_2)<\cdots < \max(B_k)$. 
Associate to each block $B=\{a_p>a_{p-1}>\cdots >a_1\}$ in $\mc{P}$ the long backwards cycle $c_{B}\coloneqq (a_p\, a_{p-1}\, \cdots \, a_1)$, and define $\sigma(\mc{P})\in S_{n}$ as the product of disjoint cycles 
\[
\sigma(\mc{P})\coloneqq \prod_{1\leq i\leq k}c_{B_i}.
\]

Given this correspondence, we use the terms `noncrossing partition' and `noncrossing permutation' interchangeably when this causes no confusion. 
We shall also treat the terms `blocks' and `cycles' as synonyms.

\begin{eg}
\label{eg:noncrossing}
Let $\mc{P}= 138/2/45/67 \in \ncpart_8$.  Then $\mc{P}$ is noncrossing, with arc diagram 
\[
\mathcal{P}= \begin{tikzpicture}[scale = 0.75, baseline = 0.75*-0.2]
\foreach \x in {1, ..., 8}{\draw[fill] (\x - 1, 0) node[inner sep = 2pt] (\x) {$\scriptstyle \x$};}
\foreach \i\j in {1/3, 3/8, 4/5, 6/7}{\draw[thick] (\i) to[out = 35, in = 145] (\j);}
\end{tikzpicture}.
\]
The associated permutation in $\NC_7$ has cycle notation $(8\,3\,1)(2)(4\,5)(7\,6)$ and in one line notation is the more opaque $82154763$.
\end{eg}

\subsection{Nested Forests}
\label{sec:nestfor}

Say that a partition $\mc{P}$ of the set $\NN$ is \emph{finite noncrossing} if there exists an $N \in \NN$ for which $\{n\} \in \mc{P}$ for all $n > N$ and $\{B \in \mc{P} \suchthat \max{B} \le N\}$ is a noncrossing set partition of $[N]$.  
A \emph{nested \textup{(}indexed\textup{)} forest} $\wh{F}$ is a family of binary trees $\left(T_B\right)_{B\in \mc{P}}$ where $\mc{P}$ is a finite noncrossing set partition of $\NN$ and each $T_{B}$ has $|B|$-many leaves.

Every indexed forest is a nested forest, and we extend  our conventions for indexed forests to nested forests.  
For each tree $T_{B}$ in a nested forest $\wh{F}$, we absolutely identify the leaves of $T_{B}$ with $B$ in increasing order from left to right, and define the support of $\wh{F}$ to be the set $\supp{\wh{F}}$ of all leaves of nontrivial trees in $\wh{F}$.
We also write $\internal{\wh{F}}$ to denote the set of internal nodes in all $T_{B}$.  

We denote by $\nfor$ the set of nested forests, and we write $\nfor_n$ for the subset of nested forests with $\supp(\wh{F}) \subseteq [n]$.  
For $\wh{F} = \left(T_B\right)_{B\in \mc{P}} \in \nfor_n$ we let $\ncperm(\wh{F}) \in S_{n}$ be the noncrossing partition whose cycles are the sets $B \in \mc{P}$ which are contained in $[n]$.

\begin{figure}[!h]
    \centering
    \includegraphics[width=0.75\linewidth]{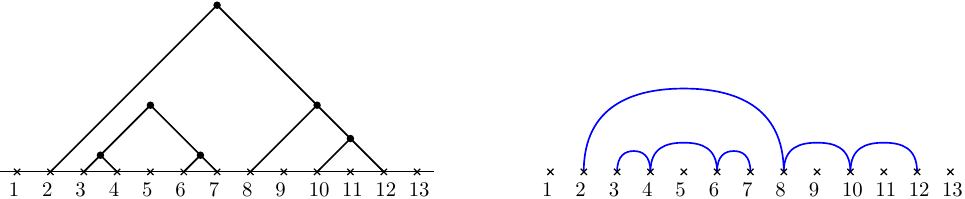}
    \caption{A nested indexed forest and the noncrossing partition obtained from its support}
    \label{fig:nested_forest_noncrossing_via_support}
\end{figure}


\section{Double forest polynomials}
\label{sec:equivariant_forest_polynomials}

Recall the (equivariant) Bergeron--Sottile operators from Section~\ref{section:equivariant_quasisymmetrics}.  
It was shown in \cite{NST_a} that if we define the ``trimming operations'' $$\tope{i}\coloneqq\frac{\rope{i+1}f-\rope{i}f}{x_i}=\rope{i+1}\partial_i=\rope{i}\partial_i,$$ then the forest polynomials $\{\forestpoly{F}(\xl)\suchthat F\in \indexedforests\}$ are the unique family of homogeneous polynomials such that $\forestpoly{\emptyset}=1$ and
$\tope{i}\forestpoly{F}=\delta_{i\in \qdes{F}}\forestpoly{F/i}$. 
The forest polynomials form a basis of $\ZZ[\xl]$, and if we restrict to the forests $\{F\in \indexedforests\suchthat \qdes{F}\subset [n]\}$ then we obtain a basis of $\ZZ[\xl_{n}]$. 
 
We define equivariant forest polynomials via their interactions with a new analogue of the trimming operation using the equivariant Bergeron--Sottile operators $\rope{i}^-$ and $\rope{i}^+$; see \Cref{defn:equivariant_bergeron_sotille}.

\begin{defn}
    We define the \emph{equivariant trimming operation} $\eope{i}$ by
$$\eope{i}f\coloneqq\frac{\rope{i}^+f-\rope{i}^-f}{x_i-t_i}=\rope{i}^+\partial_if=\rope{i}^-\partial_i f.$$
\end{defn}
    Unlike the $\tope{i}$ operations which satisfy the Thompson monoid relations
    $$\tope{i}\tope{j}=\tope{j}\tope{i+1}\text{ for }i>j,$$
    the $\eope{i}$ do not satisfy these relations. 
    In \Cref{sec:coeffextraction} we will define a modified composition under which the  $\eope{i}$ operators satisfy these relations.
\begin{thm}
\label{thm:ForestDesiderata}
There is a unique family of homogenous polynomials $\{\forestpoly{F}(\xl;\tl)\suchthat F\in \indexedforests\}\subset \poltx$ such that, denoting $\wh{\tl}_i=(t_1,t_2,\ldots,t_{i-1},t_{i+1},\ldots)$, we have 
\[
\forestpoly{F}(\tl;\tl) = \begin{cases}
1 & \text{if $F = \emptyset$} \\
0 & \text{otherwise.}
\end{cases}
\qquad\text{and}\qquad
\eope{i}\forestpoly{F}= \begin{cases}
\forestpoly{F/i}(\xl;\wh{\tl}_i) & \text{if $i\in \qdes{F}$} \\
0 & \text{otherwise.}
\end{cases}
\]
\end{thm}
\begin{proof}
    Deferred to \Cref{sec:ForestProof}.
\end{proof}
\begin{defn}
    We call the polynomials $\forestpoly{F}\in \poltx$ in Theorem~\ref{thm:ForestDesiderata} the \emph{double forest polynomials}.
\end{defn}
\begin{eg}
In lowest degrees we have $\forestpoly{\emptyset}(\xl;\tl)=1$ and $\forestpoly{\underline{i}}=x_1+\cdots+x_i-t_1-\cdots-t_i$.
\end{eg}
\begin{rem}
\label{rem:uniqueforest}
    If a polynomial $f$ depends nontrivially on $x_n$ but not $x_m$ for any $m>n$, then $\eope{n}f\ne 0$.  It follows that 
    \[
    \bigcap_{i=1}^{\infty}\ker(\eope{i})=\ZZ[\tl].
    \] 
    Therefore by induction on $|F|$ it is clear that   double forest polynomials must be unique provided they exist, which we prove in \Cref{sec:ForestProof}.
\end{rem}
\begin{cor}
\label{cor:basiscor}
    For $F\in \indexedforests$ we have $\forestpoly{F}(\xl;\bm{0})=\forestpoly{F}(\xl)$, the ordinary forest polynomial. Furthermore,
    \begin{enumerate}[label=(\arabic*)]
    \item \label{basis:it1} the double forest polynomials $\{\forestpoly{F}(\xl;\tl)\suchthat F\in \indexedforests\}$ are a $\ZZ[\tl]$-basis for $\poltx$, and
    \item \label{basis:it2}
        $\{\forestpoly{F}(\xl;\tl)\suchthat \qdes{F}\subset [n]\}$ is a $\ZZ[\tl]$-basis for $\ZZ[\tl][\xl_n]$.
    \end{enumerate}
\end{cor}
\begin{proof}
When $\tl=\bm{0}$, the $\eope{i}$ become the trimming operations $\tope{i}$, and Theorem~\ref{thm:ForestDesiderata} translates to the recursive characterization of forest polynomials in terms of the $\tope{i}$ \cite[Theorem 6.1]{NST_a}, which shows $\forestpoly{F}(\xl;\bm{0})=\forestpoly{F}(\xl)$. 
    Now ~\ref{basis:it1}  follows from the fact that $\{\forestpoly{F}(\xl)\suchthat F\in \indexedforests\}$ is a $\ZZ$-basis of $\ZZ[\xl]$.

    For ~\ref{basis:it2}, begin by noting that
    \[
        f\in \ZZ[\tl][\xl_n]\Longleftrightarrow \eope{i}f=0\text{ for all }i>n.
    \]
    If we now write $f\in \poltx$ uniquely as $f=\sum a_F(\tl)\forestpoly{F}(\xl;\tl),$ then 
    \[
        \eope{i}f=\sum_{i\in \qdes{F}}a_F(\tl)\,\forestpoly{F/i}(\xl;\wh{\tl}_i),
    \]
    which is equal to zero exactly when $a_F=0$ for all $F$ with $i\in\qdes{F}$.
    Demanding $\eope{i}f=0$ for $i>n$ now immediately yields~\ref{basis:it2}.
\end{proof}
The following corollary is proved nearly identically, so we omit the proof.
\begin{cor}\leavevmode
    \begin{enumerate}
        \item $\{\forestpoly{F}(\xl;t_1,\ldots,t_n,0,\ldots)\suchthat F\in \indexedforests\}$ is a $\ZZ[\tl_n]$-basis for $\ZZ[\tl_n][\xl]$.
        \item $\{\forestpoly{F}(\xl;t_1,\ldots,t_n,0,\ldots)\suchthat \qdes{F}\subset [n]\}$ is a $\ZZ[\tl_n]$-basis for $\ZZ[\tl_n][\xl_n]$.
    \end{enumerate}
\end{cor}

\begin{cor}
\label{cor:forestdepend}
    A forest polynomial $\forestpoly{F}(\xl;\tl)$ depends only on the variables $x_1,\ldots,x_{\max \qdes{F}}$ and the variables $t_1,\ldots,t_{(\max \supp F)-1}$. In particular, if $F\in \indexedforests_n$ then $\forestpoly{F}(\xl;\tl)\in \ZZ[\tl_{n-1}][\xl_n]$.
\end{cor}
\begin{proof}
To show the $x$-variable dependence, we note that if $k$ is the largest index such that $\forestpoly{F}(\xl;\tl)$ depends on $x_k$ then $$\eope{k}\forestpoly{F}(\xl;\tl)=\frac{1}{x_k-t_k}(\forestpoly{F}(x_1,\ldots,x_k;\tl)-\forestpoly{F}(x_1,\ldots,t_k;\tl))\ne 0,$$
so we must have $k\in \qdes{F}$.

To show the $t$-variable dependence we proceed by induction on $|F|$. Suppose $\forestpoly{F}$ depends on $t_k$ with $k\ge \max \supp F$. Write $\forestpoly{F}=\sum h_i(\xl;\wh{\tl}_k) t_k^i$ where $h_i$ is the coefficient of $t_k^i$. If for some $i\ge 1$ we have that $h_i$ has a nontrivial dependence of $x_j$, then choosing $j$ maximal we have $\eope{j}\forestpoly{F}=\delta_{j\in \qdes{F}}\forestpoly{F/j}(\xl;\wh{\tl}_j)$ has nonzero $t_k^i$ coefficient, contradicting our induction hypothesis as $t_k$ appears as the $k-1$th equivariant variable in $\wh{\tl}_j$ and $k-1\ge \max \supp{F/j}$. Thus only $h_0$ may have a nontrivial $x$-dependence. But now if $h_i\ne 0$ for any $i\ge 1$ then $\forestpoly{F}(\tl;\tl)$ has a nontrivial $t_k^i$-coefficient, contradicting that it is equal to $0$.
\end{proof}

A few more basic properties of double forest polynomials, particularly that if $F,G$ are forests with disjoint sets of leaves that $\forestpoly{F\sqcup G}(\xl;\tl)=\forestpoly{F}(\xl;\tl)\forestpoly{G}(\xl;\tl)$, are easier to establish with the combinatorial model we derive later rather than from first principles. We defer these to \Cref{subsec:Furtheerproperties}.

\subsection{Double fundamental quasisymmetric polynomials}
\label{subsec:double_fund}
 
The fundamental quasisymmetric polynomials are a well-known $\ZZ$-basis of ordinary quasisymmetric polynomials $\qsym{n}$. We now recall from \cite{NST_a} how they are realized as the forest polynomials associated to forests $F\in \zigzag{n}$, and use this to obtain a $\ZZ[\tl]$-basis of $\eqsym{n}$ given by a new family of ``double fundamental quasisymmetric polynomials''.

\begin{defn}  
\label{def:qseq}
A \emph{padded composition} of length $n$ is a sequence $(0^{n-\ell},a_1,\dots,a_{\ell})$ where $\ell\ge 0$ and $a_i>0$ for $1\leq i\leq \ell$. 
The set of padded compositions of length $n$ will be denoted by $\padded{n}$.
For $c=(0^{n-\ell},a_1,\cdots, a_\ell)\in \padded{n}$, let $\operatorname{Set}(c)=\{a_1,a_1+a_2,\dots,a_1+\cdots+a_{\ell-1}\}$.
The associated \emph{fundamental quasisymmetric polynomial} is the generating function 
\[
    \slide{c}(\xl_n)=\sum_{\substack{1\leq i_1\leq \cdots \leq i_m\leq n\\ j\in \operatorname{Set}(c) \Longrightarrow i_j<i_{j+1}
    }}x_{i_1}\cdots x_{i_m},
\]
where $m\coloneqq a_1+\cdots+a_{\ell}$.
\end{defn}
With this indexing convention, $\slide{c}(\xl_n)$ is the fundamental quasisymmetric polynomial whose reverse-lexicographic leading term $\prod_{1\leq i\leq \ell}x_{n-\ell+i}^{a_i}$. 
This would traditionally be indexed by the ``strong composition'' $(a_1,\dots,a_{\ell})$.

By interpreting a padded composition $c$ as an element of $\nvect$ (by appending $0$s at the end), we can associate the unique indexed forest $F_c$ whose code is $c$. 
It was shown in \cite[Theorem 8.3]{NST_a} that
\begin{equation}
\label{eq:fundamental_as_forest}
\slide{c}(\xl_n)=\forestpoly{F_c}(\xl_n),    
\end{equation}
and the map $c\mapsto F_c$ gives a bijection from $\padded{n}$ to $\zigzag{n}$. 
For the zigzag forest in Figure~\ref{fig:zigzag}, the associated padded composition $c\in \padded{6}$ is $c=(0,0,2,2,1,2)$.
\begin{defn}
\label{def:double_fundamental}
     Given $c\in \padded{n}$, we define the associated \emph{double fundamental quasisymmetric polynomial} by
    $$
        {\mathfrak{F}_{c}(\xl_n;\textbf{t})}\coloneqq \forestpoly{F_c}(\xl_n;\textbf{t}).
    $$
\end{defn}

\begin{eg}
\label{eg:double_fundamental}
We give an example of a double fundamental quasisymmetric polynomial computed from the combinatorial model appearing later in Section~\ref{sec:subword_models}.
Consider $c=(0,2,1)\in \padded{3}$. Then the associated forest is given by $1^02^23^1=2\cdot 2\cdot 3\in \indexedforests_5$, which has the unique Sylvester word $\{4,2,3\}$. By \Cref{thm:subwordformula_for_forests} we can compute
\[
    \slide{c}(\xl;\tl)=(x_1-t_4)(x_1-t_1)(x_2+x_3-t_1-t_2)+(x_1+x_2-t_1-t_4)(x_2-t_2)(x_3-t_2).
\]
Observe that the specialization $\tl=\bm{0}$  is $x_1^2x_2+x_1^2x_3+x_2^2x_3+x_1x_2x_3$ which is the fundamental quasisymmetric polynomial indexed by $c$.

\end{eg}
\begin{thm}
\label{thm:qsymbasis}
For $c\in \padded{n}$ we have $\slide{c}(\xl_n;\bm{0})=\slide{c}(\xl_n)$, the ordinary fundamental quasisymmetric polynomial. 
Furthermore, the set $$\{\forestpoly{F}(\xl_n;\tl)\suchthat F\in \zigzag{n}\}=\{\slide{c}(\xl_n;\tl)\suchthat c\in \padded{n}\}$$ is a $\ZZ[\tl]$-basis for $\eqsym{n}$.
\end{thm}
\begin{proof}
The specialization at $\tl=\bm{0}$ to ordinary fundamental quasisymmetric polynomials follows from \Cref{cor:basiscor} in conjunction with~\eqref{eq:fundamental_as_forest}.

We now establish the second half.
Note that \[
f\in \eqsym{n}\Longleftrightarrow \eope{i}f=0\text{ for all }i\ne n.
\]
In particular $\forestpoly{F}(\xl;\tl)\in \eqsym{n}$ for any $F\in \zigzag{n}$.
By Corollary~\ref{cor:basiscor}, the double forest polynomials $\{\forestpoly{F}(\xl;\tl)\suchthat F\in \zigzag{n}\}$ are $\ZZ[\tl]$-linearly independent, hence it remains to show that they $\ZZ[\tl]$-linearly span $\eqsym{n}$. If we write $f\in \poltx$ uniquely as $f=\sum a_F\forestpoly{F}(\xl;\tl),$ then 
\[
\eope{i}f=\sum_{i\in \qdes{F}}a_F\,\forestpoly{F/i}(\xl;\wh{\tl}_i),
\]
which is equal to zero exactly when $a_F=0$ for all $F$ with $i\in\qdes{F}$. 
If $f\in \eqsym{n}$, then we infer that the nonzero $a_F$ necessarily satisfy $\qdes{F}\subset \{n\}$, or equivalently, $F\in \zigzag{n}$.
This concludes the proof.
\end{proof}
We now describe a version of this result where we work only with equivariant variables $t_1,\ldots,t_n$.

\begin{defn}
    We let $\eqsym{[n]}\coloneqq\eqsym{n}\cap \ZZ[\tl_n][\xl_n]$
\end{defn}

Since the definition of equivariant quasisymmetry only involves the variables $t_1,\ldots,t_{n-1}$ it is straightforward to verify that
$$\eqsym{n}=\eqsym{[n]}\otimes \ZZ[t_{n+1},t_{n+2},\ldots].$$

The following is proved almost identically to \Cref{thm:qsymbasis}, so we omit the proof.
\begin{thm}
    The set $$\{\forestpoly{F}(\xl_n;t_1,\ldots,t_n,0,\ldots)\suchthat F\in \zigzag{n}\}=\{\slide{c}(\xl_n;t_1,\ldots,t_n,0,\ldots)\suchthat c\in \padded{n}\}$$ is a $\ZZ[\tl_n]$-basis for $\eqsym{[n]}$.
\end{thm}

\section{A combinatorial model for forest polynomials}
\label{sec:subword_models}

Following~\cite{KM05}, we will use the term \emph{subword model} to refer to the realization of a family of polynomials $\mathcal{F}$ as uniformly-defined generating functions for subwords in a fixed word.
A well-known example of a subword model is the celebrated pipe dream formula for single and double Schubert polynomials~\cite{Ber93, BJS93, FK96, KM05, KnUd23}.
In this section we define a subword model which allows us to realize double forest polynomials, and in the next section double Schubert polynomials.  
While both models are reminiscent of the pipe dream formula and we have a similar diagrammatic depiction, the resulting formulas seem to be novel.  

We now define a subword model formally (see \cite[Definition 1.8.1]{KM05} for what guides us).
A \emph{word} is an ordered sequence $\omega = (\omega_{1}, \ldots, \omega_{\ell})$. 
For convenience we shall often omit parentheses and commas in writing our words.
An ordered subsequence $\pi$ of $\omega$ is called a \emph{subword} of $\omega$. 
We shall abuse notation on occasion and denote this by $\pi\subset \omega$.
A \emph{weight function} on $\omega$ is a function $
\weight:[\ell]\to \poltx$,
which we informally view as assigning an element of $\poltx$ to each letter in $\omega$. 
This given, the weight $\weight(\pi)$ of a subword $\pi$ of $\omega$ is the product of the weights of the letters in $\pi$.
For a collection of words $D$, we define 
\[
\mc{R}(\omega;D)=\{\pi\subset \omega\suchthat \pi \in D\}.
\]
Then the subword generating function associated to $D$ is
\begin{equation}
\label{eq:subword_gf}
    \sum_{\pi\in \mathcal{R}(\omega;D)}\weight(\pi).    
\end{equation}

The relevant collections of words for us are $D = \red{w}$ for some permutation $w$ or $D = \sylv{F}$ for some indexed forest $F$.
In these cases, we will write 
\[
\mc{R}(\omega;F)\coloneqq \mc{R}(\omega;\sylv{F}),
\quad\text{ and }\quad
\mc{R}(\omega;w)\coloneqq \mc{R}(\omega;\red{w}).
\]
In Section~\ref{subsec:subword} we define a word $\longword_{[n]}$ and weight function $\weight$ such that the following result holds.

\begin{thm}
\label{thm:subwordformula_for_forests}
For $F\in \indexedforests_{n+1}$ we have
\[
\forestpoly{F}(\xl;\tl)=\sum_{\pi\in \mathcal{R}(\longword_{[n]};F)}\weight(\pi).
\]
\end{thm}

The proof of Theorem~\ref{thm:subwordformula_for_forests} is deferred to Section~\ref{sec:ForestProof}, after which we give a  diagrammatic construction of our subword model in Section~\ref{subsec:vines}.  The same subword model can be used to compute double Schubert polynomials in Section~\ref{sec:VineSchub}.

\begin{figure}[!h]
    \centering
    \includegraphics[scale=1]{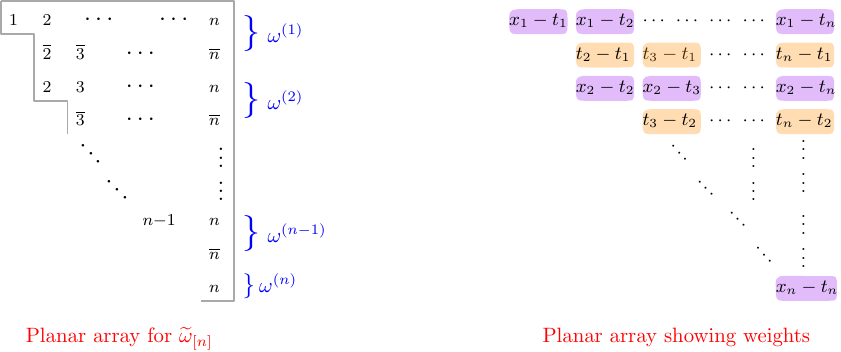}
    \caption{The long word $\longword$ as a planar array}
    \label{fig:longword}
\end{figure}

\subsection{The vine subword model}
\label{subsec:subword}

We now define the subword model used in Theorem~\ref{thm:subwordformula_for_forests}, which we call the \emph{vine subword model}.  Proofs are deferred to Section~\ref{sec:ForestProof}.
For each $1\le k\le n$ let
\[
\omega_{[n]}^{(k)} = (n, n-1, \dots, k+1, k, \overline{k+1}, \dots, \overline{n-1}, \overline{n}),
\]
where the bar is used only as a decoration to distinguish the two instances of $i$ in $\omega^{(k)}$.
To distinguish the letters in $\omega^{(k)}$ we sometimes denote these letters as $i^{(k)}$ for $k\le i \le n$ and $\displaystyle\smash{\overline{i}^{(k)}}$ for $k+1\le i \le n$; when using this convention we say that $i$ is the \emph{value} of both $i^{(k)}$ and $\overline{i}^{(k)}$.

We write $\longword_{[n]}$ for the concatenation
\[
\longword_{[n]} = \omega_{[n]}^{(1)} \omega_{[n]}^{(2)} \omega_{[n]}^{(3)} \cdots \omega_{[n]}^{(n-1)} \omega_{[n]}^{(n)},
\]
omitting the $[n]$ when it is clear from context.  
In Section~\ref{subsec:vines} we will consider $\longword_{[n]}$ as coming from a planar array, which may provide some useful intuition here; see Figure~\ref{fig:longword}.  
We refer to the terms from $\omega^{(k)}$ in $\longword$ as the \emph{$k$th syllable} of $\longword$.
For example $\longword_{[4]}$ consists of four syllables,
\begin{equation}
\label{eq:long_word}
    \longword_{[4]} =4\,3\,2\,1\,\overline{2}\,\overline{3}\,\overline{4}\,\, \,4\,3\,2\,\overline{3}\,\overline{4}\,\,\,4\,3\,\overline{4}\,\,\,4
\end{equation}
When finding subwords we only consider the value of each letter.  For example $|\mathcal{R}(\longword_{[4]}; \{2\})|=3$, counting two instances of $2$ and one instance of $\overline{2}$.

We assign each letter of $\longword$ a polynomial weight in $\xl$ and $\tl$ based on its position, value, and the syllable in which it occurs:
\begin{align*}
\weight(j^{(i)})= 
(x_{i} - t_{j})
\quad\text{ and }\quad
\weight(\overline{j}^{(i)})=(t_{j} - t_{i}).
\end{align*}
For example, the two instances of $2$ in~\eqref{eq:long_word} have weights $x_1-t_2$ and $x_2-t_2$ and the solitary instance of $\overline{2}$ has weight $t_2-t_1$.
This in turn means that the subword generating function from~\eqref{eq:subword_gf} for $D=\{2\}$ equals $(x_1-t_2)+(x_2-t_2)+(t_2-t_1)=x_1+x_2-t_1-t_2$.

\subsection{Proof that double forest polynomials exist}
\label{sec:ForestProof}
We are now ready to prove \Cref{thm:ForestDesiderata}, that double forest polynomials exist, and \Cref{thm:subwordformula_for_forests}, that double forest polynomials are computed by the subword model.  
We prove both results by showing that the recursive characterization $(\rope{i}^{+} - \rope{i}^{-})\forestpoly{F}(\xl;\tl) = (x_{i} - t_{i}) \forestpoly{F/i}(\xl;\hat{\tl}_{i})$ of forest polynomials also holds for the subword model 
\begin{equation}
\label{eq:forest_subword_model}
\sum_{\pi\in \mathcal{R}(\longword_{[n]};F)}\weight(\pi),
\end{equation}
as well as the normalization condition $\forestpoly{F}(\tl;\tl)=\delta_{F,\emptyset}$.
Our approach is to first devise a sign-reversing involution which simplifies the application of $(\rope{i}^{-} - \rope{i}^{+})$ to the expression in~\eqref{eq:forest_subword_model}, and then to introduce a weight-preserving bijection between the surviving terms and the elements of $\mathcal{R}(\longword_{[n]};F/i)$.

To begin, say that a word with any two consecutive letters distinct is \emph{Smirnov}; see~\cite{ShWa16} for instance.  
(In particular note that reduced words and Sylvester words are Smirnov.)  
In the following, denote by $\longword^-_{[n],i}$ (resp. $\longword^+_{[n],i}$) the subword of $\longword_{[n]}$  obtained by deleting $\omega^{(i)}_{[n]}$ (resp. the barred letters from $\omega^{(i)}_{[n]}$ and the unbarred letters from $\omega^{(i+1)}_{[n]}$).
For instance, if $n=6$ and $i=3$, we have
\begin{align*}
    \longword_{[6],3}^{-}&=6\,5\,4\,3\,2\,1\,\overline{2}\,\overline{3}\,\overline{4}\,\overline{5}\,\overline{6}\,\,\,\, 6\,5\,4\,3\,2\,\overline{3}\,\overline{4}\,\overline{5}\,\overline{6}\,\,\,\,
    \textcolor{gray!50}{6\,5\,4\,3\,\overline{4}\,\overline{5}\,\overline{6}}\,\,\,\,
    6\,5\,4\,\overline{5}\,\overline{6}\,\,\,\,
    6\,5\,\overline{6}\,\,\,\,
    6
    \\
    \longword_{[6],3}^{+}&=6\,5\,4\,3\,2\,1\,\overline{2}\,\overline{3}\,\overline{4}\,\overline{5}\,\overline{6}\,\,\,\, 6\,5\,4\,3\,2\,\overline{3}\,\overline{4}\,\overline{5}\,\overline{6}\,\,\,\,
    6\,5\,4\,3\,\textcolor{gray!50}{\overline{4}\,\overline{5}\,\overline{6}\,\,\,\,
    6\,5\,4}\,\overline{5}\,\overline{6}\,\,\,\,
    6\,5\,\overline{6}\,\,\,\,
    6,
\end{align*}
where the gray letters are meant to be ignored.
In terms of the planar array in Figure~\ref{fig:longword}, these words are obtained by striking out two consecutive rows. For $\longword_{[n],i}^{-}$ we strike out the consecutive rows of lengths $i+1$ and $i$, whereas for $\longword_{[n],i}^+$ we strike out the two rows of length $i$.

\begin{prop}
\label{prop:Rsubword}
Let $D$ be a finite set of Smirnov words.
  Then 
  \begin{align*}
      \rope{i}^-\sum_{\pi\in \mathcal{R}(\longword_{[n]};D)}\weight(\pi)&=\rope{i}^-\sum_{\pi\in \mathcal{R}(\longword^-_{[n],i};D)}\weight(\pi)\\
      \rope{i}^+\sum_{\pi\in \mathcal{R}(\longword_{[n]};D)}\weight(\pi)&=\rope{i}^+\sum_{\pi\in \mathcal{R}(\longword^+_{[n],i};D)}\weight(\pi)
  \end{align*}
\end{prop}
\begin{proof}
    Throughout the proof we write $\longword$ for $\longword_{[n]}$.
    The first equality can be rewritten as 
    \begin{align}
    \label{eq:first_rhs_0}
        \rope{i}^-\left(\sum_{\pi\in \mc{R}(\longword;D)\setminus\mc{R}(\longword^-_{i};D) }\weight(\pi)\right)=0.
    \end{align}
    To establish this we consider the involution $\iota$ on $\mc{R}(\longword;D)\setminus \mc{R}(\longword^-_{i};D)$ defined as follows. 
    If $\pi$ contains $i^{(i)}$, then we declare $\iota(\pi)=\pi$.
    Otherwise, we consider $j$ minimal so that $\pi$ contains exactly one of $\{j^{(i)},\overline{j}^{(i)}\}$. 
    Such a $j$ exists as $\pi$ is Smirnov.
    Now $\iota(\pi)$ is obtained by swapping $j^{(i)}$ for $\overline{j}^{(i)}$ (or vice versa). 
    Next observe that $\rope{i}^-\weight(\pi)=0$ if $\pi$ contains $i^{(i)}$ and $\rope{i}^{-}\weight(\pi)=-\rope{i}^{-}\weight(\iota(\pi))$ otherwise.
    It follows that~\eqref{eq:first_rhs_0} holds.

    The proof of the second equality is similar.
    This time we consider the involution on $\mc{R}(\longword;D)\setminus \mc{R}(\longword^+_{i};D)$ that takes the maximal $j$ such that $\pi$ contains exactly one of $\{\overline{j}^{(i)},j^{(i+1)}\}$ and then swaps $\overline{j}^{(i)}$ for $j^{(i+1)}$ (or vice versa). 
    Again, the existence of such a $j$ follows from the fact that $\pi$ is Smirnov.
    This involution has the property that it negates each $\rope{i}^+\weight(\pi)$
    implying that
    \begin{align}
    \label{eq:second_rhs_0}
        \rope{i}^+\left(\sum_{\pi\in \mc{R}(\longword;D)\setminus\mc{R}(\longword^+_{i};D) }\weight(\pi)\right)=0,
    \end{align}
    which is clearly equivalent to the second equality.
\end{proof}

\begin{cor}
\label{cor:esubword}
Let $D$ be a finite set of Smirnov words. 
Then
    $$\eope{i}\sum_{\pi\in \mc{R}(\longword_{[n]};D)}\weight(\pi)=\rope{i}^+\sum_{\substack{\pi \in \mc{R}(\longword^+_{[n],i};D)\\ i^{(i)}\in \pi}}\weight(\pi\setminus \{i^{(i)}\}).$$
\end{cor}
\begin{proof}
As before, we write $\longword$ for $\longword_{[n]}$.
By its definition we have $(x_i-t_i)\eope{i}=\rope{i}^+-\rope{i}^{-}$.
From Proposition~\ref{prop:Rsubword} it then follows that
\begin{align}
\label{eq:intermediate_step_existence_forest}
    (x_i-t_i)\eope{i}\sum_{\pi\in \mc{R}(\longword;D)}\weight(\pi)=\rope{i}^+\sum_{\pi\in \mc{R}(\longword^+_{i};D)}\weight(\pi)-
    \rope{i}^-\sum_{\pi\in \mathcal{R}(\longword^-_{i};D)}\weight(\pi).   
\end{align}
Now consider the injection from $\mc{R}(\longword^-_{i};D)$ to $\mc{R}(\longword^+_{i};D)$ that takes a subword $\pi$ of $\longword^-_{i}$ and replaces all instances of unbarred letters of the form $j^{(i+1)}$ by $j^{(i)}$.
Let $\pi'$ be the resulting word; this is clearly in $\mc{R}(\longword^+_{i};D)$. 
Furthermore any word in $\mc{R}(\longword^+_{i};D)$  not containing $i^{(i)}$ can be obtained uniquely as the image of this injection. 
Finally note that
\[
    \rope{i}^-\weight(\pi)=\rope{i}^+\weight(\pi').
\]
Equation~\ref{eq:intermediate_step_existence_forest} now becomes
\begin{align}
    (x_i-t_i)\eope{i}\sum_{\pi\in \mc{R}(\longword;D)}\weight(\pi)=\rope{i}^+\sum_{\substack{\pi\in \mc{R}(\longword^+_{i};D)\\ i^{(i)}\in \pi}}\weight(\pi)
\end{align}
The claim follows from observing that the weight of $i^{(i)}$ equals $x_i-t_i=\rope{i}^+(x_i-t_i)$.
\end{proof}

\begin{lem}
   \label{le:insertsylvester}
    Let $F\in \indexedforests$. The following are equivalent.
    \begin{enumerate}
        \item $i\in\qdes{F}$. 
        \item Neither $i-1$ nor $i+1$ come after $i$ in some  $w\in \sylv{F}$.
        \item Neither $i-1$ nor $i+1$ come after $i$ in any  $w\in \sylv{F}$.
    \end{enumerate} 
\end{lem}
\begin{proof}
 Note that $i\in \qdes{F}$ is equivalent to saying that $i$ is the terminal letter of some word in $\sylv{F}$ (by \Cref{obs:SylvesterTrimming} for example).
    
    The relations \eqref{eqn:forsylvrelations} from \Cref{rem:NewSylv} do not allow consecutive adjacent numbers in a Sylvester word to commute past each other. Hence if either of $\{i-1,i+1\}$ comes after $i$ in some $w\in \sylv{F}$, then $i$ cannot commute to the end of the Sylvester word, preventing it from lying in $\qdes{F}$.
    On the other hand, if neither $i-1$ nor $i+1$ comes after $i$ in some $w\in \sylv{F}$, then we claim that $i\in\qdes{F}$. Indeed, if $j$ is the letter to the right of $i$, then one of $i-1,i+1$ has value between $i$ and $j$, so by \eqref{eqn:forsylvrelations} $i$ can commute past $j$ to the right. Applying this repeatedly we can move $i$ to the end of the Sylvester word and hence $i\in \qdes{F}$.
\end{proof}

We are now ready to address the existence of double forest polynomials.
\begin{proof}[Proof of \Cref{thm:ForestDesiderata} and \Cref{thm:subwordformula_for_forests}]
Suppose $F\in \indexedforests_{n+1}$.
Again we write $\longword$ for $\longword_{[n]}$.
By \Cref{rem:uniqueforest} it suffices to show that the ansatz 
\[
\forestpoly{F}'(\xl;\tl)=\sum_{\pi\in \mc{R}(\longword;F)}\weight(\pi)
\]
satisfies 
\begin{enumerate}[label=(\roman*)]
    \item \label{it:ansatz_1}$\forestpoly{F}'(\tl;\tl)=\delta_{F,\emptyset}$, and
    \item \label{it:ansatz_2}$\eope{i}\forestpoly{F}'(\xl;\tl)=\forestpoly{F/i}'(\xl;\wh{\tl}_{i})$ if $i\in \qdes{F}$, and $0$ otherwise. 
\end{enumerate} 
For~\ref{it:ansatz_1}, note that there is an involution on $\mc{R}(\longword;F)$ where we consider the minimal $j$ such that $j^{(i)}$ or $\overline{j}^{(i)}$ appears, and either fix the word if $j=i$, and otherwise swap $j^{(i)}$ for $\overline{j}^{(i)}$ and vice versa (in this latter case exactly one of $j^{(i)}$ and $\overline{j}^{(i)}$ occurs as the words in $\mathcal{R}(\longword;F)$ are Smirnov). 
This is a sign-reversing involution on the weights $\weight(\pi)(\tl;\tl)$ so we conclude that $\forestpoly{F}'(\tl;\tl)=\delta_{F,\emptyset}$.

For~\ref{it:ansatz_2}, we consider cases. 
By Corollary~\ref{cor:esubword} we are interested in the right-hand side of
\[
    \eope{i}\forestpoly{F}'(\xl;\tl)=\rope{i}^+\sum_{\substack{\pi \in \mc{R}(\longword^+_{i};F)\\ i^{(i)}\in \pi}}\weight(\pi\setminus \{i^{(i)}\}).  
\]
If $i\not\in \qdes{F}$, then we claim that there does not exist $\pi\in \mc{R}(\longword^+_i;F)$ so that $i^{(i)}\in \pi$.
 Indeed, in $\longword_i^{+}$ all letters after $i^{(i)}$ have value $\ge i+2$, so by \Cref{le:insertsylvester} we have $i\in \qdes{F}$ which is a contradiction. 
 Therefore $\eope{i}\forestpoly{F}'(\xl;\tl)=0$ in this case, as desired.

On the other hand, suppose $i\in \qdes{F}$. 
We aim to establish that
\begin{align}
\label{eq:intermediate_Riplus_Ei}
    \rope{i}^+\sum_{\substack{\pi \in \mc{R}(\longword^+_{i};F)\\ i^{(i)}\in \pi}}\weight(\pi\setminus \{i^{(i)}\})= \forestpoly{F/i}'(\xl;\wh{\tl}_{i}).
\end{align}
Note that Sylvester words have the property that any letter appears at most once. 
So if we assume that $i^{(i)}$ appears in $\pi\in \sylv{F}$ then we are guaranteed that $\pi$ cannot contain letters in $\{i^{(i-1)},\overline{i}^{(i-1)},i^{(i-2)},\overline{i}^{(i-2)},\dots,i^{(1)},\overline{i}^{(1)}\}$.

For clarity we now make the dependence on $n$ explicit.
There is a bijection
$$\Phi:\longword_{[n],i}^+\setminus \{i^{(i)},i^{(i-1)},\overline{i}^{(i-1)},i^{(i-2)},\overline{i}^{(i-2)},\dots,i^{(1)},\overline{i}^{(1)}\}\to \longword_{[n-1]}$$
obtained by applying the following transformations.
\begin{align*}
    j^{(k)} \mapsto \left\lbrace \begin{array}{ll}
    (j-1)^{(k-1)} & k>i\\
    (j-1)^{(k)} & j>i, k\leq i\\
    j^{(k)} & j<i,k\leq i\\
    \end{array}\right.
    \qquad\qquad 
    \overline{j}^{(k)} \mapsto \left\lbrace \begin{array}{ll}
    (\overline{j-1})^{(k-1)} & k>i\\
    (\overline{j-1})^{(k)} & j>i, k\leq i\\
    \overline{j}^{(k)} & j<i,k\leq i\\
    \end{array}\right.
\end{align*} 
The map $\Phi$ has the property that
\[
    \rope{i}^+\weight(\pi\setminus \{i^{(i)}\})=\weight(\Phi(\pi\setminus \{i^{(i)}\}))(\xl;\wh{\tl}_{i}).
\]
This given, to conclude~\eqref{eq:intermediate_Riplus_Ei} it remains to show that $\pi\mapsto \Phi(\pi\setminus \{i^{(i)}\})$ is a bijection  
\[
\{\pi\in \mathcal{R}(\longword_{[n],i}^+;F)\suchthat i^{(i)}\in \pi\} \longleftrightarrow \mc{R}(\longword_{[n-1]};F/i).
\]

To see that this map is well defined, note that in any $\pi\in \mc{R}(\longword_{[n],i}^+;F)$ containing $i^{(i)}$ we must have all letters after this one have value $\ge i+2$. 
Hence by \Cref{le:insertsylvester} we have $i\in \qdes{F}$. 
Therefore removing $i^{(i)}$ and then applying $\Phi$ gives an element of $\mc{R}(\longword_{[n-1]};F/i)$. 
This map is furthermore clearly injective, so it remains to show that it is surjective. 

For $\pi'\in \mathcal{R}(\longword_{[n-1]};F/i)$, if we append $i$ to the end of the string of values of the letters in $\Phi^{-1}(\pi')$ then we obtain an element of $\sylv{F}$. 
Repeatedly applying the forest Sylvester relation $w_1aiw_2 \sim w_1iaw_2$ when $a\ge i+2$ and $i+1\not\in w_2$ (see \eqref{eqn:forsylvrelations}),
the string of values for $\Phi^{-1}(\pi')\cup \{i^{(i)}\}\subset \longword_{[n]}$ also give an element of $\sylv{F}$ as all letters of $\Phi^{-1}(\pi')$ after $i^{(i)}$ are $\ge i+2$ in value. 
The subword $\Phi^{-1}(\pi')\cup \{i^{(i)}\}\subset \longword_{[n]}$ is then the desired preimage.
\end{proof}

\subsection{Vine diagrams}
\label{subsec:vines}

We conclude the section by giving an alternate presentation of our subword model which we call the \emph{pictorial vine model}.  
We associate each letter in $\longword$ with an entry of the planar array from Figure~\ref{fig:longword}, moving row-by-row from top to bottom, reading odd (unbarred) rows from right to left and even (barred) rows from left to right, so that with the exception of $\omega^{(n)}$ each syllable of $\longword$ corresponds to exactly two consecutive rows.  

We refer to the places that the individual letters occupy as \emph{boxes}.  
Every subword of $\longword$ is uniquely determined by selecting a subset of boxes, and the weight of the selection is the product of the weights of the selected boxes. 
The weight of each box is equal to $x_{i} - t_{j}$ if the box is in the $j$th column of the $i$th odd (unbarred) row and $t_{j} - t_{i}$ if the corresponding box is in the $j$th column of the $i$th even (barred) row.

The upshot of this approach is that we can give a diagrammatic criterion for whether a subword (seen as a subset of boxes) belongs to a particular set $D$, which for our purposes will either be $\sylv{F}$ or, in Section~\ref{sec:VineSchub}, $\red{w}$.  
We do so by placing a pictorial \emph{tile} in each box according to a fixed set of rules and considering the resulting structure; see Figure~\ref{fig:tiles_for_vines} for the tiles used in the vine diagram.  
This parallels the diagrammatic realization of the pipe dream formula for Schubert polynomials, in which each subword is realized as an eponymous diagram~\cite{Ber93,BJS93} .  

For forest polynomials, we first note that a selection of boxes corresponds to the Sylvester word of some forest (i.e. no letter is repeated) if and only if we have chosen at most one box from each column of the diagram. 
If a subset $S$ of the diagram is selected, then we put a $T$-shape into the selected boxes, a horizontal bar in any unselected box above a $T$-shape, and for the remaining boxes we put an odd elbow in odd numbered rows and an even elbow in even numbered rows (see \Cref{fig:tiles_for_vines}). A selection of boxes corresponds to an element of $\sylv{F}$, and hence contributes to $\forestpoly{F}$ if and only if the diagram is combinatorially equivalent to $F$ (see \Cref{fig:vine_take_2}).

\begin{figure}
    \centering
    \includegraphics[scale=1.5]{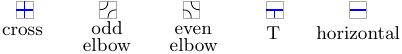}
    \caption{The various tiling pieces. Pieces 2,3,4,5 are used for the forest vine model and pieces 1,2,3 are used for the Schubert vine model.}
    \label{fig:tiles_for_vines}
\end{figure}

\begin{figure}[!h]
    \centering
    \includegraphics[width=\textwidth]{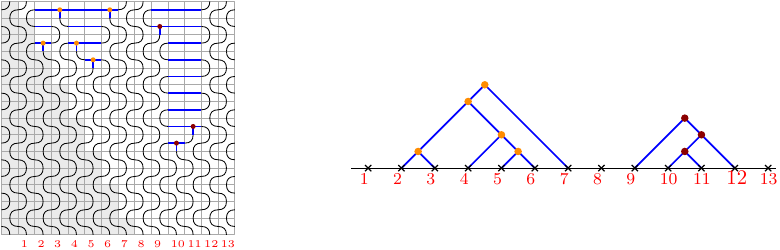}
    \caption{A vine diagram and its associated indexed forest}
    \label{fig:vine_take_2}
\end{figure}

\section{The vine model for Schubert polynomials}
\label{sec:VineSchub}

In this section we state further properties of the vine model from Section~\ref{sec:subword_models}, and in particular the following result.

\begin{thm}
\label{thm:subwordformula_for_schubs}
    For $w\in S_{n+1}$ we have
    \[
        \schub{w}(\xl;\tl)=\sum_{\pi\in \mathcal{R}(\longword_{[n]};w)}\weight(\pi)
    \]
\end{thm}

We prove this result in Section~\ref{subsec:Furtheerproperties} after giving a diagrammatic interpretation of this formula.  
Under the assumption that forest polynomials exist, the same argument leads to a second proof of Theorem~\ref{thm:subwordformula_for_forests}; see Remark~\ref{rem:forest_varswap_proof}. 
This second proof is noteworthy as it illuminates our initial intuition for the vine subword model.

\subsection{The vine model for double Schubert polynomials}
\label{subsec:vinemodel}

Before proving Theorem~\ref{thm:subwordformula_for_schubs} we briefly discuss the new pictorial vine model for double Schubert polynomials which it implies. 

The elements of $\mathcal{R}(\longword_{[n]};w)$ correspond exactly to realizations of the permutation $w$ in the planar array from Figure~\ref{fig:longword} as follows.  
Referring to the first three tiling pieces in \Cref{fig:tiles_for_vines}, given a subset $D$ of boxes we place a cross into the selected boxes, and for the remaining boxes we put an odd elbow in odd-numbered rows and an even elbow in even-numbered rows. Then $D$ belongs to $\mathcal{R}(\longword_{[n]};w)$ if and only if the resulting tiles connect column $i$ at the bottom of the diagram to column $w(i)$ at the top, and furthermore no two strands cross more than once; see \Cref{fig:vine_schubert}. Thus if we sum over all such configurations the product of the weights associated to the boxes containing crosses, we recover the formula in \Cref{thm:subwordformula_for_schubs}.

\begin{figure}
    \centering
    \begin{tikzpicture}[scale = 0.35, baseline = 0]
\path[fill = lightgray!30!white, xshift = -1cm] (0, 0) -- (1, 0) -- (1, -1) -- (2, -1) -- (2, -3) -- (3, -3) -- (3, -5) -- (4, -5) -- (4, -7) -- (5, -7) -- (5, -8) -- (0, -8) -- cycle;
\draw[color = gray] (-1, 0) grid (5, -8);
\foreach \x in {1,...,5}{
    \draw[thin, red] (\x-0.5, 0.5) node {$\x$};
    }
\foreach \x/\sigmai in {1/2, 2/5, 3/4, 4/3, 5/1}{
    \draw[thin, red] (\x-0.5, -8.5) node {$\sigmai$};
    }
\foreach \x/\y in {0/1, 2/1, 5/1, 0/7, 1/7, 2/7, 3/7, 5/7}{
    \draw[thick] (\x-1, 0.5-\y) to[out = 0, in = -90] (\x-0.5, 1-\y);
    \draw[thick] (\x, 0.5-\y) to[out = -180, in = 90] (\x-0.5, -\y);
    }
\foreach \fullrow in {3, 5}{
    \foreach \x in {0,...,5}{
        \draw[thick] (\x-1, 0.5-\fullrow) to[out = 0, in = -90] (\x-0.5, 1-\fullrow);
        \draw[thick] (\x, 0.5-\fullrow) to[out = -180, in = 90] (\x-0.5, -\fullrow);
        }
    }
\foreach \x/\y in {0/2, 1/2, 3/2, 5/2, 0/4, 1/4, 2/4, 4/4, 5/4}{
    \draw[thick] (\x-1, 0.5-\y) to[out = 0, in = 90] (\x-0.5, -\y);
    \draw[thick] (\x, 0.5-\y) to[out = -180, in = -90] (\x-0.5, 1 -\y);
    }
\foreach \fullrow in {6, 8}{
    \foreach \x in {0,...,5}{
        \draw[thick] (\x-1, 0.5-\fullrow) to[out = 0, in = 90] (\x-0.5, -\fullrow);
        \draw[thick] (\x, 0.5-\fullrow) to[out = -180, in = -90] (\x-0.5, 1 -\fullrow);
        }
    }
\foreach \x/\y in {1/1, 3/1, 4/1, 2/2, 4/2, 3/4, 4/7}{
    \draw[very thick, blue] (\x, 0.5-\y) -- (\x-1, 0.5-\y);
    \draw[very thick, blue] (\x-0.5, 1-\y) -- (\x-0.5, -\y);
    }
\end{tikzpicture}
    \caption{A vine diagram for the permutation $25431=s_4s_3s_1s_2s_4s_3s_4$.}
    \label{fig:vine_schubert}
\end{figure}
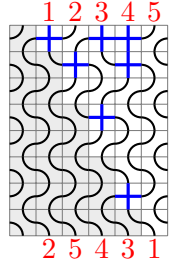

\begin{rem}
    If we set $\tl=\bm{0}$ in the vine model for Schuberts, then we have effectively forbidden cross tiles in even numbered rows, rendering these rows superfluous.
    It is then easily checked that the vine model recovers the well-known pipe dream model for ordinary Schubert polynomials \cite{Ber93,BJS93}. 
    Additionally, our vine model is also compatible with the ``back stable limit'' and gives a new formula for back stable double Schubert polynomials. 
    As discussed in \cite[\S 4.4]{Lam18} this is not the case with the pipe dream formula of Fomin--Kirillov \cite{FK96}  for double Schubert polynomials; indeed the bumpless pipe dreams introduced in \cite[\S 5]{Lam18} fix this but do not directly recover the pipe dream formula when $\tl=\bm{0}$. 
    Setting $\tl=\bm{0}$ in the vine model for double forest polynomials produces the diagrammatic interpretation in \cite[\S 4.2]{NST_2}. 
\end{rem}

We now turn to the proof of Theorem~\ref{thm:subwordformula_for_schubs}.  Throughout we work with a single, fixed value of $n$.  We begin by carefully choosing a sequence of variable substitutions which transitions from $(x_1,\ldots,x_{n})$ to $(t_1,\ldots,t_{n})$.  

\begin{defn}
For $1\le i \le j$ we consider two ordered variable sets:
\begin{align*}
    \overline{B}_{j,i}=&(x_1,\ldots,x_{i-1}, x_{i},t_{i+1},\ldots,t_{j},t_i,t_{j+1},\ldots, t_{n-1}),\,\text{and}\\
    B_{j,i} =&(x_1,\ldots,x_{i-1},t_i,t_{i+1},\ldots,t_{j-1},x_i,t_{j},\ldots,t_{n-1}).
\end{align*}

To each prefix  $L\subset \longword_{[n]}$ we associate one of the sets defined above: let
\[
B(L)=\begin{cases}
\overline{B}_{j,i}&\text{if $\overline{j}^{(i)}$ is the final letter of $L$,}\\
B_{j,i}&\text{if $j^{(i)}$ is the final letter of $L$,}\\
(t_1,\ldots,t_n)& \text{if $L=\emptyset$},
\end{cases}
\]
so that $B(\longword_{[n]})=(x_1,\ldots,x_{n})$.
\end{defn}

Our proof depends on the following analysis of how $B(L)$ changes as $L$ is shortened.  

\begin{fact}\label{fact:alphabet}
Let $L$ be a nonempty prefix of $\longword_{[n]}$ with final letter $a$, and write $L = L'a$.
    \begin{enumerate}
        \item If $a = n^{(i)}$ (resp. $a = \overline{n}^{(i)}$), then $B(L')$ is obtained from $B(L)$ by replacing $x_{i}$ (resp. $t_{n}$) by $t_{n}$ (resp. $t_{i}$) in position $n$.  
        \item If $a = \overline{j}^{(i)} \neq  \overline{n}^{(i)}$, then $B(L')$ is obtained from $B(L)$ by swapping $t_{j}$ in position $j$ with $t_{i}$ in position $j+1$.
        \item If $a = j^{(i)} \neq n^{(i)}$, then $B(L')$ is obtained from $B(L)$ by swapping $x_{i}$ in position $j$ with $t_{j+1}$ in position $j+1$.
    \end{enumerate}
\end{fact}
 For example for $n=3$, we have $$\longword_{[3]}=3^{(1)}2^{(1)}1^{(1)}\overline{2}^{(1)}\overline{3}^{(1)}\,\,\,\,\,3^{(2)}2^{(2)}\overline{3}^{(2)}\,\,\,\,\,3^{(3)}$$
 and the variable sets
\begin{multline*}
 B^3(\longword_{[3]})=\underbrace{(x_1,x_2,x_3)}_{B_{3,3}},\quad 
 \underbrace{(x_1,x_2,t_3)}_{\overline{B}_{3,2}}, \,
 \underbrace{(x_1,x_2,t_2)}_{B_{2,2}}, \,
 \underbrace{(x_1,t_2,x_2)}_{B_{3,2}},\\[1.5ex]
 \underbrace{(x_1,t_2,t_3)}_{\overline{B}_{3,1}},\, 
 \underbrace{(x_1,t_2,t_1)}_{\overline{B}_{2,1}}, \,
 \underbrace{(x_1,t_1,t_2)}_{B_{1,1}},\,
 \underbrace{(t_1,x_1,t_2)}_{B_{2,1}}, \,
 \underbrace{(t_1,t_2,x_1)}_{B_{3,1}}, \quad
 \underbrace{(t_1,t_2,t_3)}_{B(\emptyset)}.
\end{multline*}
\begin{proof}[Proof of \Cref{thm:subwordformula_for_schubs}]
For $w\in S_{n+1}$, so that $\des{w}\subset [n]$, we have $\partial_{n+1}\schub{w}=\partial_{n+2}\schub{w}=\cdots=0$.  
Therefore $\schub{w}$ only depends on $x$-variables $x_1,\ldots,x_{n}$, so we can write $\schub{w}(x_1,\ldots,x_n;\tl) = \schub{w}(\xl;\tl)$ without ambiguity.
We will show that for any prefix $L\subset  \longword_{[n]}$ that 
\[
\schub{w}(B(L);\tl)=\sum_{\pi\in \mathcal{R}(L;w)}\weight(\pi).
\]
We show this by induction on $|L|$. 
For $L=\emptyset$ we have $\schub{w}(B(\emptyset);\tl)=\schub{w}(\tl;\tl)=\delta_{w,\idem}$.  
Otherwise write $L=L'a$, where $a$ is $j^{(i)}$ or $\overline{j}^{(i)}$.  
The recursion for $\partial_{j} \schub{w}(\xl;\tl)$ and \Cref{fact:alphabet} gives
 \begin{align}
 \label{eq:barred_recusion}
 \schub{w}(B(L);\tl)&=\left\lbrace
    \begin{array}{ll}
    \schub{w}(B(L');\tl)+\weight(a)\schub{ws_j}(B(L');\tl) & j\in \des{w}\\
    \schub{w}(B(L');\tl) & \text{otherwise.}
    \end{array}\right.
 \end{align}
(Note that in case (1) of \Cref{fact:alphabet}, we set $x_{n+1} = t_{n}$ or $t_{i}$ before applying the recursion for $\partial_{n}$.)
This exactly matches the recursion on subwords $\pi\in \mathcal{R}(L;w)$ given by choosing whether or not $\pi$ has final letter $a$:
\[
\sum_{\pi\in \mathcal{R}(L;w)}\weight(\pi)=\sum_{\pi\in \mathcal{R}(L';w)}\weight(\pi)+\quad\weight(a)\hskip-10pt\sum_{\pi=\pi'a\in \mathcal{R}(L;w)}\weight(\pi').\qedhere
\]
\end{proof}

As an example, in \Cref{fig:recursion_tree_schubert} we fully expand the recursion for $\schub{3142}$, where each node keeps track of the current variable set and permutation.

\begin{figure}[!ht]
    \centering
    \includegraphics[scale=0.8]{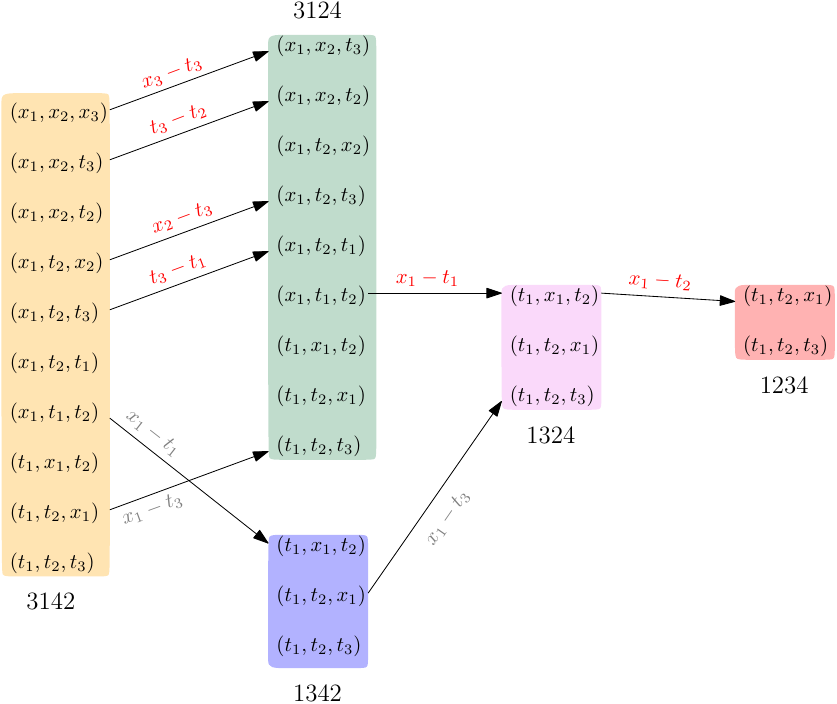}
    \caption{The recursion for $\schub{3142}=(x_2+x_3-t_1-t_2)(x_1-t_1)(x_1-t_2)$ fully expanded out.  Each arrow represents a nonzero second term in our recursion, and the weights of the vine model are recovered by multiplying the labels of any path which reaches the bottom of column $1234$ by moving only down and to the right.}
    \label{fig:recursion_tree_schubert}
\end{figure}

\begin{rem}
    The pipe dream formula for double Schubert polynomials can be derived analogously. Recall that for $w\in S_{n+1}$ this formula states that $\schub{w}(\xl; \tl) = \sum_{\pi\in \mathcal{R}(\longword^{pipe}_{[n]};w)}\operatorname{wt}^{pipe}(\pi)$ for
    $$\longword_{[n]}^{pipe}\coloneqq n^{(1)}\,\,(n-1)^{(2)}n^{(2)}\,\,\cdots \,\,2^{(n-1)}\cdots n^{(n-1)}\,\,1^{(n)}\cdots n^{(n)}$$
    and $\operatorname{wt}^{pipe}(j^{(i)})=x_{j+i-n}-t_{n+1-i}$. One can prove this subword model
    using the variable sets
\begin{align*}A_{j,i}=&(t_1,\ldots,t_{n-i},x_1,\ldots,x_{j+i-n-1},t_{n-i+1},x_{j+i-n},\ldots,x_{i-1})\text{ for }1\le j \le i \le n\end{align*}
and the recursion
\begin{align*}
    \schub{w}(A_{j+1,i};\tl)=&\schub{w}(A_{j,i};\tl)+(x_{j+i-n}-t_{n+1-i})\delta_{j\in \des{w}}\schub{ws_j}(A_{j,i};\tl).
\end{align*}
\end{rem}

\begin{rem}
\label{rem:forest_varswap_proof}
Provided one is willing to assume that forest polynomials exist, one can also derive \Cref{thm:subwordformula_for_forests} using the argument above.  
This proof is our original inspiration for the vine model and establishes the stronger statement that for any prefix $L$ of $\longword_{[n]}$, we have 
\[
\forestpoly{F}(B(L);\tl)=\sum_{\pi\in \mathcal{R}(L;F)}\weight(\pi),
\]
so we sketch it here.

Pursuing a similar inductive argument as we do for double Schubert polynomials, we consider the case $|F|, |L| > 1$ and write $L = L' a$.  Then with 
\[
B(L/a) = \begin{cases}
B(L')(\xl, \hat{\tl}_{j}) & \text{if $a = \overline{j}^{(i)}$} \\
B(L)(\xl, \hat{\tl}_{j}) & \text{if $a = j^{(i)}$}
\end{cases},
\]
the recursion for $\eope{j} \forestpoly{F}(\xl, \tl)$ and \Cref{fact:alphabet} give that
\[
\forestpoly{F}(B(L);\tl)=\begin{cases}\forestpoly{F}(B(L');\tl)+\operatorname{wt}(a)\forestpoly{F/i}(B(L/a) ;\wh{\tl}_j)&j\in \qdes{F}\\
\forestpoly{F}(B(L');\tl)&\text{otherwise.}\end{cases}
\]
The final step of the proof amounts to showing that with our inductive hypothesis on $|L|$ and $|F|$, the terms above correspond exactly and in a weight-preserving manner to elements $\pi \in \mathcal{R}(L;F)$, divided by whether the final letter of $\pi$ has value $j$ or not.  
While accounting for depleted variable sets make this correspondence more cumbersome to describe, it is morally equivalent to the bijection used in the proof of \Cref{thm:subwordformula_for_schubs}.
\end{rem}

\subsection{Further properties of the vine model}
\label{subsec:Furtheerproperties}
We close this section with a couple of elementary results on double forest polynomials that are not obvious from their definition, in contrast to the analogous results in the nonequivariant case, but follow easily from the vine subword model.

\begin{cor}
    If $F,G\in \indexedforests$ have $\supp{F}\cap \supp{G}=\emptyset$, then $\forestpoly{F}\forestpoly{G}=\forestpoly{F\sqcup G}$ where $F\sqcup G$ is the indexed forest obtained by overlaying $F$ and $G$ on a common set of leaves.\footnote{In terms of codes, $F\sqcup G$ is the indexed forest with code $\sfc(F)+\sfc(G)$ where addition is component-wise.} 
\end{cor}
\begin{proof}
It suffices to assume that $F$ and $G$ are indexed trees. 
Now note that $\sylv{F\sqcup G}=\sylv{F}\shuffle \sylv{G}$ where $\shuffle$ denotes the shuffle operation.
Given $\pi \in \sylv{F\sqcup G}$, let $\pi_F$ (resp. $\pi_G$) denote the subword of $\pi$ consisting solely of letters in $\supp(F)$ (resp. $\supp(G)$).
We then have $\weight(\pi)=\weight(\pi_F)\weight(\pi_G)$ which by Theorem~\ref{thm:subwordformula_for_forests} implies the claim. 
\end{proof}

\begin{cor}
\label{cor:coxeter}
    If $w=s_{i_1}\cdots s_{i_k}$ with all $i_j$ distinct, then $\schub{w}(\xl;\tl)$ is a positive,
    multiplicity-free sum of double forest polynomials.
    As a special case, the hook Schur $s_{(m-\ell,1^{\ell})}(\xl_n;\tl)$ decomposes explicitly as 
   \[
   s_{(m-\ell,1^{\ell})}(\xl_n;\tl)=\sum \slide{c}(\xl_n;\tl)
   \]
    where the sum is over all $(0^{n-\ell-1},c_1,\dots,c_{\ell+1})\in \padded{n}$ where all $c_i>0$ and $\sum c_i=m$.
\end{cor}
\begin{proof}
    For the first half, note that $\red{w}$ is a disjoint union of $\sylv{F}$ over some indexed forests $F$.
    This is because $w$ is fully commutative, i.e. any two elements of $\red{w}$ can be related by a sequence of commutation moves, and any two elements of $\sylv{F}$ are also related via commutation moves \cite[Definition 8]{HNT05}.

    Regarding the second half of the claim, note that $s_{(m-\ell,1^{\ell})}(x_1,\dots,x_n;\tl)=\schub{w}$ where $w=s_{n-\ell}s_{n-\ell+1}\cdots s_{n-1} s_{n+m-\ell-1}\cdots s_{n+1}s_{n}$.
    The explicit expansion holds, as it is true under the specialization $\tl=0$; the latter is a special case of the well-known expansion of Schur polynomials into fundamental quasisymmetric polynomials \cite[Chapter 7]{EC2}.
\end{proof}

The column shape and the row shape are special cases of hook shapes.
In these special cases one has that the double fundamental quasisymmetric polynomial $\slide{0^{n-a}1^a}$ (resp. $\slide{0^{n-1}a}$) equals the double elementary (resp. complete) symmetric polynomial $e_a(\xl_n;\tl)$ (resp. $h_a(\xl_n;\tl)$).
In particular, $\slide{0^{n-1}1}=x_1+\cdots+x_n-t_1-\cdots-t_n$.


\section{Equivariant quasisymmetry and noncrossing partitions}
\label{sec:EQNP}
For each $\sigma \in S_{\infty}$, we define the evaluation map
\[
\begin{array}{rcl}
\ev_{\sigma}: \ZZ[\tl][\xl] & \to & \ZZ[\tl] \\
f(\xl; \tl) & \mapsto & f(t_{\sigma(1)}, t_{\sigma(2)}, \ldots; \tl)
\end{array}.
\]
For $\lambda$ a nonempty integer partition with at most $n$ parts, the non-constant double Schur polynomials $s_\lambda(x_1,\ldots,x_n;\tl)$ have the property that $\ev_{\sigma}s_\lambda=0$ for all $\sigma\in S_n$. 
This follows immediately from the fact that $s_\lambda$ is symmetric in the $x$-variables, so $\ev_{\sigma}s_\lambda=\ev_{\idem}s_\lambda=0$. The analogous vanishing for equivariant quasisymmetric polynomials involves the noncrossing partitions $\NC_n\subset S_n$.

\begin{thm}
\label{prop:qsymvanishesonNc}
If $f(x_1,\ldots,x_n;\tl)\in \eqsym{n}$ and $\sigma\in \NC_n$, then $\ev_{\sigma}f=\ev_{\idem} f$. 
In particular, if $\slide{c}(\xl;\tl)$ is a nonconstant double fundamental quasisymmetric polynomial (so that $c \in \padded{n} \setminus \{\emptyset\}$), then $\ev_{\sigma}\slide{c}=0$ for all $\sigma\in \NC_n$.
\end{thm}

As an example, we evaluate the equivariant fundamental $\slide{(0,2,1)}(\xl;\tl)$ computed in Example~\ref{eg:double_fundamental} according to the noncrossing partition $\sigma=213$:
\[
\slide{(0,2,1)}(t_2,t_1,t_3;\tl)=(t_2-t_4)(t_2-t_1)(t_1+t_3-t_1-t_2)+(t_2+t_1-t_1-t_4)(t_1-t_2)(t_3-t_2)=0.
\]
The reader is invited to check that the $\slide{(0,2,1)}(\xl;\tl)$ vanishes at each of the four remaining elements of  $\NC_3=S_3\setminus \{231\}$.  

We prove Theorem~\ref{prop:qsymvanishesonNc} in Section~\ref{sec:noncrossing_evaluation_proof} by characterizing $\NC_n$ in terms of an equivalence relation on $S_{n}$.  
We then extend the relationship between noncrossing partitions and forest polynomials in Section~\ref{subsec:relating_nc_and_forests} by constructing a bijection between $\indexedforests_{n}$ and $\NC_{n}$ which, in relation to our proof, is canonical.

\subsection{Noncrossing word combinatorics and Theorem~\ref{prop:qsymvanishesonNc}}
\label{sec:noncrossing_evaluation_proof}
We now introduce an equivalence relation which encodes noncrossing permutations.
\begin{defn}
Let $\sim_{\NC}$ denote the equivalence relation on $S_{n}$ generated by the relations $\sigma\sim_{\NC}\sigma s_i$ if $i\in\{\sigma(i),\sigma(i+1)\}$. 
\end{defn}

The relation $\sim_{\NC}$ reflects the structure of equivariant quasisymmetry in the sense that if $\sigma \sim_{\NC} \tau$, we must have $\ev_{\sigma}f=\ev_{\tau} f$ for every $f(x_1,\ldots,x_n;\tl)\in \eqsym{n}$.  In this section we show that $\NC_n$ is the equivalence class of $\sim_{\NC}$ containing $\idem\in S_n$, see~\Cref{prop:NCequiv}, and use this to deduce \Cref{prop:qsymvanishesonNc}.

\begin{eg}
    For $n=3$ the equivalence classes are
    $$\NC_3 = \{123,132,213,312,321\}\text{ and }\{231\}.$$
    For $n=4$ the equivalence classes are
    \begin{align*}
    \NC_4 = \{1234,1324,2134,3124,3214,1243,2143,1423,1432,4123,4132,4213,4231,4321\},\text{ and }\\
    \{1342,3142\},
    \{2314\},\{2341\},\{2431,2413\},\{3412\},\{3421,3241\},\{4312\}\end{align*}
    $\NC_4$ is pictured in \Cref{fig:NC4} along with the generating equivalences for $\sim_{\NC}$.
\end{eg}

\begin{figure}[!ht]
    \centering
    \includegraphics[scale=0.8]{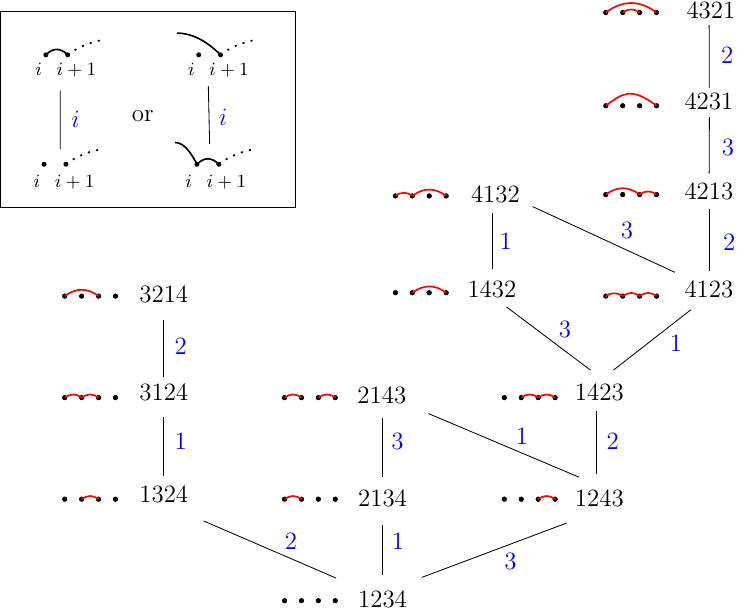}
    \caption{$\NC_4$ with the generating equivalences labeled.
    \label{fig:NC4}
    }
\end{figure}

We first show that $\NC_n$ is the equivalence class containing $\idem$ under $\sim_{\NC}$. \footnote{See also \cite[\S 3.7]{biane:hal-04891027} where Biane states essentially this criterion.}


\begin{lem}
\label{lem:NCsi}
    Let $\sigma\in \NC_{n}$. 
    Then $\sigma s_i\in \NC_{n}$ if and only if $i\in \{\sigma(i),\sigma(i+1)\}$.
\end{lem}
\begin{proof}
    Let $\mathcal{P}=\mathcal{P}(\sigma)$ be the set partition given by the blocks of $\sigma$. We proceed by taking cases.
    
    If $i=\sigma(i)$, then $\{i\}$ is a singleton block in $\mc{P}$. 
    Now consider the block $B$ in $\mc{P}$ containing $i+1$.
    Since $i\not\in B$ we have 
    \[
    c_Bs_i=c_{B\sqcup \{i\}}
    \]
    where $\sqcup$ denotes disjoint union. 
    Replacing $\{i\}$ and $B$ by $B\sqcup \{i\}$ cannot violate the noncrossing property, so $\sigma s_i\in \NC_{n}$ in this case.
    
    Next suppose that $i=\sigma(i+1)$. 
    This then means that both $i$ and $i+1$ belong to the same block, say $B$, in $\mc{P}$.
    It then follows that 
    \[
    c_B s_i=c_{B\cap \{1,\ldots,i\}}c_{B\cap \{i+1,\dots,n\}},
    \] splitting $B$ into blocks $B\cap \{1,\dots,i\}$ and $B\cap \{i+1,\dots,n\}$. 
    Again we see that $\sigma s_i \in \NC_{n}$.

    Finally, if $i\not\in\{\sigma(i),\sigma(i+1)\}$, then $i\in A\in \mathcal{P}$ with $|A|\ne 1$ and $i+1\in B\in \mathcal{P}$ with $A\ne B$. 
    Multiplying $\sigma$ by $s_{i}$ merges the cycles $c_A$ and $c_B$ into a single cycle $c_Ac_Bs_i$ containing $A$ and $B$. 
    However, this merged cycle is not equal to $c_{A\sqcup B}$: 
    \[
    (c_Ac_Bs_i)(i+1)=\sigma s_i(i+1)=\sigma(i)\ne i=c_{A\sqcup B}(i+1),
    \] 
    so $\sigma s_i\not\in \NC_{n}$.
\end{proof}

In order to conduct a more careful analysis of the $\sim_{\NC}$ classes, we develop a combinatorics of reduced words for noncrossing partitions; these tools will be used again in later sections.

\begin{defn}
For $\sigma\in \NC_n$, 
\begin{enumerate}
\item a \emph{noncrossing descent} of $\sigma$ is an $1 \le i < n$ such that $\sigma s_i\in \NC_n\text{ and }\ell(\sigma s_i)<\ell(\sigma)$; and

\item a \emph{noncrossing reduced word} for $\sigma \in \NC_{n}$ is a word $(i_{1}, i_{2}, \ldots, i_{\ell})$ such that for all $1\le k \le \ell$, the permutation $s_{i_{1}} s_{i_{2}} \cdots s_{i_{k}}$ is noncrossing and satisfies $i_k\in \desnc{s_{i_1}\cdots s_{i_{k}}}$.

\end{enumerate}
We write $\desnc{\sigma}$ and $\rednc{\sigma}$ for the sets of noncrossing descents and noncrossing reduced words for $\sigma$, respectively.
\end{defn}
We note that $\desnc{\sigma}$ is always a subset of the usual descent set $\des{\sigma} = \{i\suchthat \ell(\sigma s_i)<\ell(\sigma)\}$. Similarly, $\rednc{\sigma}$ is always a subset of the set of ordinary reduced words $\red{\sigma}$.

Furthermore, by \Cref{lem:NCsi} each noncrossing descent corresponds to a generating equivalence for $\sim_{\NC}$ that decreases the length of $\sigma$, and each noncrossing reduced word corresponds to a length-increasing chain of generating equivalences from $\idem$ to $\sigma$.

\begin{eg}
\label{eg:noncrossing_reduced_words}
Taking $\sigma = 82154763 \in \NC_{7}$ from Example~\ref{eg:noncrossing} we see that $\desnc{\sigma}=\{2, 4, 6\}$. Indeed one may check that $\sigma s_2=81254763$, 
$\sigma s_4=82145763$, and
$\sigma s_6=82154673$ belong to
belongs to $\NC_8$; for instance 
\[
\mc{P}(\sigma s_2) = \begin{tikzpicture}[scale = 0.75, baseline = 0.75*-0.2]
\foreach \x in {1, ..., 8}{\draw[fill] (\x - 1, 0) node[inner sep = 2pt] (\x) {$\scriptstyle \x$};}
\foreach \i\j in {1/2, 2/3, 3/8, 4/5, 6/7}{\draw[thick] (\i) to[out = 35, in = 145] (\j);}
\end{tikzpicture}
\quad\text{and}\quad
\mc{P}(\sigma s_6) = \begin{tikzpicture}[scale = 0.75, baseline = 0.75*-0.2]
\foreach \x in {1, ..., 8}{\draw[fill] (\x - 1, 0) node[inner sep = 2pt] (\x) {$\scriptstyle \x$};}
\foreach \i\j in {1/3, 3/8, 4/5}{\draw[thick] (\i) to[out = 35, in = 145] (\j);}
\end{tikzpicture}.
\]
The remaining descents at position $1$ and $7$ give permutations $28154763$ and $82154736$, which are not noncrossing.
Considering noncrossing reduced words, this means that every element of $\rednc{\sigma}$ must end in $2$, $4$, or $6$; for example
\[
(7, 6, 5, 4, 3, 2, 1, 4, 5, 4, 6, 7, 6, 2)\in\rednc{\sigma}.
\]
One can check by computer that of the 183365 elements of $\red{\sigma}$, only 336 are noncrossing.
\end{eg}
 
\begin{prop}
\label{prop:NCsidec}
    Let $\sigma\in \NC_{n}$. Then $i\in \desnc{\sigma}$ exactly when $i$ is the minimal element of its cycle and $i+1$ is not.  
    In particular for each $\sigma\in \NC_{n}\setminus\{\idem\}$ we have $\desnc{\sigma}\ne \emptyset$.
    
\end{prop}
\begin{proof}
    The first part follows from using \Cref{lem:NCsi} together with the fact that $\ell(\sigma s_i)<\ell(\sigma)$ if and only if $\sigma(i+1)<\sigma(i)$. 
    Now consider a non-identity $\sigma\in \NC_n$.
    Find the maximal $i$ with the property that for all $1\leq j\leq i$ we have that $j$ is the minimal element in its block. 
    Since $\sigma\neq \idem$ we know that $i<n$.
    By its definition, we must have that $i\in \desnc{\sigma}$.
\end{proof}

\begin{cor}
\label{prop:NCRW}
Every noncrossing partition has a noncrossing reduced word.
\end{cor}
\begin{proof}
    This follows from iteratively applying \Cref{prop:NCsidec}.
\end{proof}

\begin{cor}
\label{prop:NCequiv}
The equivalence class of $\sim_{\NC}$ containing $\idem\in S_n$ is $\NC_n$.
\end{cor}
\begin{proof}
    By \Cref{lem:NCsi} we know that $\NC_n$ is a union of equivalence classes, and by \Cref{prop:NCRW} every noncrossing partition is equivalent to the identity.
\end{proof}

\begin{proof}[Proof of \Cref{prop:qsymvanishesonNc}]
By the definition of equivariant quasisymmetry, $f$ takes a constant value on all $(t_{\sigma(1)},\ldots,t_{\sigma(n)})$ for $\sigma$ in a fixed equivalence class of $\sim_{\NC}$, and $\NC_n$ is the equivalence class containing $\idem$ by \Cref{prop:NCequiv}, so $\ev_{\sigma}f=\ev_{\idem}f=0$. If $f=\slide{c}(x_1,\ldots,x_n;\tl)$ with $c\ne \emptyset$ then $\ev_{\idem}f=0$ by the normalization condition on double forest polynomials, so $\ev_{\sigma} f=0$.
\end{proof} 

\subsection{A bijection $\ForToNC$ from indexed forests to noncrossing permutations}
\label{subsec:relating_nc_and_forests}

In this section we describe a bijection $\ForToNC: \indexedforests_{n} \to \NC_{n}$ and relate it to the combinatorial structure of noncrossing reduced words.  This bijection will be used in later sections and in particular Theorem~\ref{thm:forestatownperm}.

We begin with a graphical description of our bijection. 
Recall the definition of a nested forest from Section~\ref{sec:nestfor} and the map $\ncperm : \nfor_{n} \to \NC_{n}$ defined therein.

Given a nested forest $\wh{F} \in \nfor_{n}$, define an operation of `deleting the left edge of $v$' for $v \in \internal{F}$ as follows.  
If $v$ is a root, delete $v$ as well as the two edges incident to it. 
If $v$ is the child of $u$, then delete node $v$ as well as the three edges incident to it, and subsequently insert a new edge $uv_R$ from $u$ to the right child $v_{R}$ of $v$, making $v_{R}$ the left (resp.~right) child of $u$ if $v$ was the left (resp.~right) child of $u$ originally.  
The result will be another nested forest.  

\begin{defn}
\label{defn:LC}
For $F\in \indexedforests$, let $\operatorname{LC}(F)\in \nfor$ be the nested forest obtained by deleting the left edge of each right child $v\in\internal{F}$. 
We define $\ForToNC:\suppfor{n}\to \NC_n$ by $$\ForToNC(F)=\ncperm(\operatorname{LC}(F)).$$
\end{defn}

The map $\ForToNC$ is a bijection.  As shown in Figure~\ref{fig:nc_partition_to_indexed_forest}, we can recover the code $c(F)=(c_1,c_2,\ldots)$ by setting, for an element $i$ contained in the cycle $B$ of $\ncperm(\operatorname{LC}(F))$,
\[
c_{i} = \begin{cases} 
|B| & \text{if $i = \min{B}$ and $B$ is nested in some other cycle,} \\
|B| - 1 & \text{if $i = \min{B}$ and $B$ is not nested in any other cycle,} \\
0 & \text{otherwise.}
\end{cases}
\]

\begin{figure}[ht!]
    \centering
    \includegraphics[width=\textwidth]{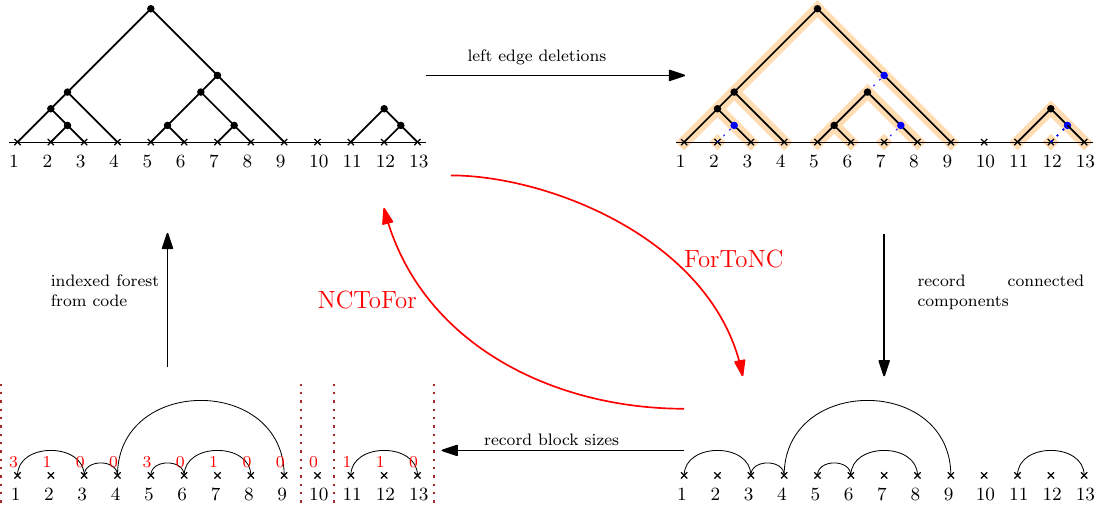}
    \caption{An example of the map $\ForToNC$ and its inverse for $n = 13$}
    \label{fig:nc_partition_to_indexed_forest}
\end{figure}

\begin{eg}
\label{re:unique_noncrossing_descent}
    Recall that the permutations in $S_n$ that have at most one descent comprise the distinguished class of Grassmannian permutations, so it is natural to consider the $\sigma \in \NC_n$ for which $|\desnc{\sigma}|\leq 1$. 
     Proposition~\ref{prop:NCsidec} shows that this occurs exactly when there exists an $1\leq i\leq n$ so that every element in $\{1,\dots,i\}$ is the minimal in its cycle in $\sigma$ and no element in $\{i+1,\dots,n\}$ is minimal in its cycle. It is not hard to show that under the $\ForToNC$ bijection these noncrossing partitions are the ones that map to the zigzag forests in $\zigzag{i}$ that are supported on $[n]$, and the associated permutations are the unique elements of $\NC_n$ for which there is a unique noncrossing reduced word.
\end{eg}

\subsection{The $\ForToNC$ bijection via heaps}
We now discuss how this bijection naturally arises in the study of noncrossing reduced words.  
For $\sigma \in S_{n}$, the Tits--Matsumoto theorem \cite{Mat64,Tits69} states that any two elements of $\red{\sigma}$ are connected by a sequence of commutation relations $\mathbf{a}ij\mathbf{b} \Leftrightarrow \mathbf{a}ji\mathbf{b}$ for $|j-i|>1$ and braid relations $\mathbf{a}i(i + 1)i\mathbf{b} \Leftrightarrow \mathbf{a}(i + 1)i(i + 1)\mathbf{b}$. 
The \emph{commutation class} of a reduced word $\omega \in \red{\sigma}$ is the subset $\mathcal{C}(\omega) \subseteq \red{\sigma}$ of words which can be obtained from $\omega$ by applying exclusively commutation relations.  

\begin{prop}
\label{prop:noncrossing_commutation_class}
For $\sigma \in \NC_{n}$, the set $\rednc{\sigma}$ is a commutation class in $\red{\sigma}$.
\end{prop}

In order to prove~\Cref{prop:noncrossing_commutation_class} we collect the following simple consequences of~\Cref{lem:NCsi}.

\begin{fact}
\label{fact:NCcomm}
For $\sigma \in \NC_{n}$ and $a \in \desnc{\sigma}$,
\begin{enumerate}
\item for any other $b \in\desnc{\sigma}$, we have $|b-a| > 1$,

\item if $|b - a| > 1$, then $b \in \desnc{\sigma s_{a}}$ if and only if $b \in \desnc{\sigma}$, and 

\item there exists an $\omega \in \rednc{\sigma}$ whose final letter is $a$.

\end{enumerate}
\end{fact}

\begin{proof}[Proof of Proposition~\ref{prop:noncrossing_commutation_class}]
We first show that applying a commutation relation to a  noncrossing reduced word produces another noncrossing reduced word.  As this is a recursive condition on prefixes, it is enough to check commutations of the final two letters in any $\omega \in \rednc{\sigma}$.  
To this end write $\omega =  \omega_{1}\cdots \omega_{\ell-2} \omega_{\ell-1} \omega_{\ell}  \in \rednc{\sigma}$ and assume that $|\omega_{\ell} - \omega_{\ell-1}| > 1$.  
We aim to show that $\omega' =  \omega_{1}\cdots  \omega_{\ell-2}  \omega_{\ell}  \omega_{\ell-1}$ is a noncrossing reduced word.  
As $\omega_{1}\cdots  \omega_{\ell-2} $ is by assumption a noncrossing reduced word, we need only show that $\omega_{\ell-1} \in \desnc{\sigma}$ and $\omega_{\ell} \in \desnc{\sigma s_{\omega_{\ell-1}}}$.
To see this, note that $\omega_{\ell-1}$ is a noncrossing descent for $\sigma s_{\omega_{\ell}}$, so by Fact~\ref{fact:NCcomm} (2) it is also a noncrossing descent for $\sigma$.  
Thus $\sigma s_{\omega_{\ell-1}}$ is a noncrossing partition, so by Fact~\ref{fact:NCcomm} (2) again $\omega_{\ell}$ is a noncrossing descent for $\sigma s_{\omega_{\ell-1}}$.  

We now prove that any two noncrossing reduced words for $\sigma$ are related by commutation moves.  
We prove this by induction on $\ell(\sigma)$, with base case $\ell(\sigma) \le 1$ holding vacuously, so take $\omega, \psi \in \rednc{\sigma}$ for $\sigma$ having length at least two.  
If the final letters of $\omega$ and $\psi$ are equal then our inductive hypothesis completes the proof. 
So assume that $\omega$ and $\psi$ have distinct final letters $u$ and $v$.  
We then have that $u, v \in \desnc{\sigma}$ and by Fact~\ref{fact:NCcomm} (a), $|u-v| > 1$.  
By Fact~\ref{fact:NCcomm} (b), $\tau = \sigma s_{u} s_{v} = \sigma s_{v} s_{u} \in \NC_{n}$, so we can take $\phi \in \rednc{\tau}$ so that $\phi v \in \rednc{\sigma s_{u}}$, $\phi u \in \rednc{\sigma s_{v}}$, and both $\phi uv$ and $\phi vu$ belong to $\rednc{\sigma}$.  

Now write $\omega = \omega'u$ and $\psi = \psi'v$, so that $\omega' \in \rednc{\sigma s_{u}}$ and $\psi' \in  \rednc{\sigma s_{v}}$.  As $\ell(\sigma s_{u}) = \ell(\sigma s_{v}) = \ell(\sigma) - 1$, our inductive hypothesis makes it possible to reach $\phi v$ from $\omega'$ using commutation moves, and likewise $\phi u$ from $\psi'$.  Thus, we can go from $\omega = \omega' u$ to $\phi v u$ in the same way, then apply the commutation of $u$ and $v$ to get $\phi u v$, and finally go to $\psi' v = \psi$.
\end{proof}

Using~\Cref{prop:noncrossing_commutation_class} we can apply the theory of heaps to understand noncrossing reduced words.  
What follows is a straightforward application of the theory developed in~\cite{Stem96}, following~\cite{Vien86heaps}.  

\begin{defn}
\label{defn:heaporder}
For $\omega = \omega_{1} \cdots \omega_{\ell} \in \red{\sigma}$, the \emph{heap order} is the partial order $\preceq$ on $[\ell]$ generated by taking $i \prec j$ whenever $i < j$ and $|\omega_{i} - \omega_{j}| \le 1$.
The associated \emph{heap} $\heap{\omega}$ of $\omega$ is the labeled poset $([\ell],\preceq)$ together with the labeling $i\mapsto \omega_i$.
\end{defn}

We visualize the heap $\heap{\omega}$ by embedding the Hasse diagram of $\preceq$ into the first quadrant so that node $i$ sits in the line $x = \omega_{i}$, with the largest elements under $\preceq$ appearing at the bottom; see Figure~\ref{fig:NC_heap}.

\begin{prop}[{\cite[Proposition 2.2]{Stem96}}]
For $\omega \in \red{\sigma}$, the heap $\heap{\omega}$ is uniquely determined by the commutation class $\mc{C}(\omega)$, and 
\[
\mc{C}(\omega) = \{ \omega_{f(1)} \omega_{f(2)} \cdots \omega_{f(\ell)} \suchthat \text{$f: [\ell] \to [\ell]$ is a linear extension of $\preceq$} \}.
\]
\end{prop}

\begin{figure}
\begin{center}
\begin{tikzpicture}[scale = 0.75, baseline = 0]
\draw (0, 0) -- (8, 0);
\foreach \x in {1, 2, 3, 4, 5, 6, 7, 8}{
	\draw[thin] (\x + 0.05 - 0.5, 0.05) -- (\x - 0.05 - 0.5,- 0.05);
	\draw[thin] (\x - 0.05 - 0.5, 0.05) -- (\x + 0.05 - 0.5, - 0.05);
	\draw[thin, gray, dashed] (\x-0.5, 0) -- (\x-0.5, 0.5*13.5);
	}
\foreach \x in {1, 2, 3, 4, 5, 6, 7}{
	\node[above] at (\x, 0.5*13.5) {$\scriptstyle \x$};
	}
\foreach \x/\y/\c in {6/1/black, 7/2/black, 6/3/white, 5/4/black, 
4/5/white, 3/6/black, 4/7/white, 5/8/white, 6/9/white, 7/10/white,
4/3/black, 2/3/black, 1/4/black, 2/5/white}{
	\draw[fill = \c] (\x, 0.6*\y) circle (2pt) node[inner sep = 2pt] (\x\y) {};
	}
\foreach \y/\p in {61/72, 72/63, 63/54, 54/45, 45/36, 36/47, 47/58, 58/69, 69/710, 43/54, 23/14, 14/25, 25/36}{
	\draw[<-] (\y) -- (\p);
	}
\end{tikzpicture}
\qquad\qquad
\begin{tikzpicture}[scale = 0.75, baseline = 0]
\draw (0, 0) -- (8.1, 0);
\foreach \x in {1, 2, 3, 4, 5, 6, 7, 8}{
	\draw (\x + 0.05 - 0.5, 0.05) -- (\x - 0.05 - 0.5,- 0.05);
	\draw (\x - 0.05 - 0.5, 0.05) -- (\x + 0.05 - 0.5, - 0.05);
	\node[inner sep = 0, outer sep = 0pt] (leaf\x) at (\x - 0.5, 0) {};
	\node[below] at (leaf\x) {$\scriptstyle \x$};
	}
\foreach \x/\y/\c in {6/1/black, 7/2/black, 5/4/black, 3/6/black, 
4/3/black, 2/3/black, 1/4/black}{
	\draw[fill = \c] (\x, 0.6*\y) circle (2pt) node[inner sep = 0, outer sep = 0] (\x\y) {};
	}
\foreach \y/\p in {61/72, 72/54, 54/36, 43/54, 23/14, 14/36}{
	\draw[-] (\y.center) -- (\p.center);
	}
\draw (14.center) -- (leaf1.center);
\draw (23.center) -- (leaf2.center);
\draw (23.center) -- (leaf3.center);
\draw (43.center) -- (leaf4.center);
\draw (43.center) -- (leaf5.center);
\draw (61.center) -- (leaf6.center);
\draw (61.center) -- (leaf7.center);
\draw (72.center) -- (leaf8.center);
\end{tikzpicture}
\end{center}
\caption{The heap associated to $\rednc{\sigma}$ alongside the indexed forest $\ForToNC^{-1}(\sigma)$ for $\sigma = 82154763$.
The final nodes in each column of the heap (colored black) correspond exactly to the internal nodes of $\ForToNC^{-1}(\sigma)$.}
\label{fig:NC_heap}
\end{figure}
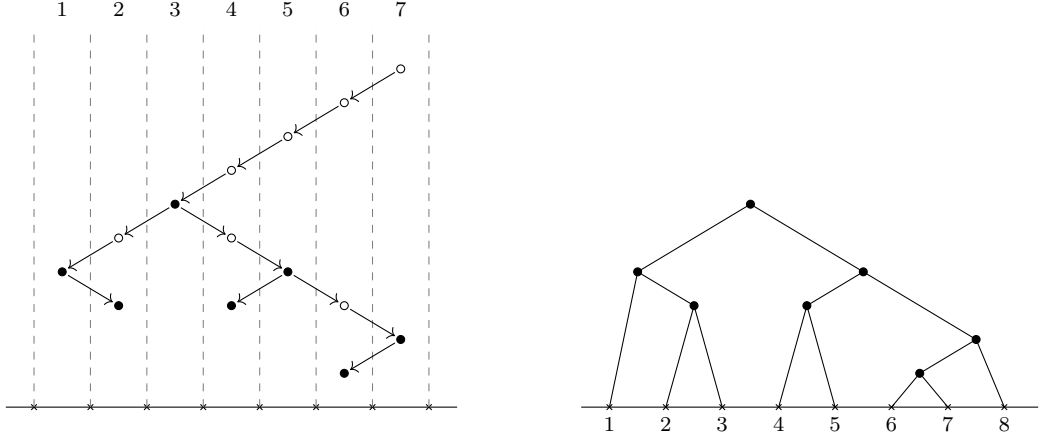

\begin{eg}
\label{eg:heaps}
Taking $\sigma = 82154763 \in \NC_{8}$ with $(7, 6, 5, 4, 3, 2, 1, 4, 5, 4, 6, 7, 6, 2) \in \rednc{\sigma}$ from Example~\ref{eg:noncrossing_reduced_words}, the heap order is shown in Figure~\ref{fig:NC_heap} alongside the indexed forest which maps to $\sigma$ under $\ForToNC$. 
The cardinality of $\rednc{\sigma}$ is the number of linear extensions of this poset, which we calculate to be 336.  
\end{eg}

The following proposition shows how the heap associated to $\sigma\in \NC_n$ naturally encodes the forest $\ForToNC^{-1}(\sigma)$.

\begin{prop}
\label{prop:heapclaims}
For $F \in \suppfor{n}$, write $\sigma = \ForToNC(F)$ and $\preceq$ for the heap order determined by the commutation class $\rednc{\sigma}$.  
\begin{enumerate}
\item The Hasse diagram of $\preceq$ is a forest $F$.

\item Removing the leaves of $F$ produces the Hasse diagram for $\preceq$ restricted to the final occurrence of each letter.

\end{enumerate}
\end{prop}

We defer the proof of~\Cref{prop:heapclaims} to Section~\ref{sec:NoncrossingCombi}, as it requires an in-depth study of a particular family of noncrossing reduced words that we have not yet introduced.  

\section{Evaluations at noncrossing partitions}
\label{sec:evaluations}

There is a well-known combinatorial formula for the image of a double Schubert polynomial under $\ev_{w}$, which we restate in Theorem~\ref{th:AJSB} below.  
This formula was independently discovered by Andersen--Jantzen--Soergel \cite{AJS} and Billey \cite{Bil99}, so we refer to it as the \emph{AJS--Billey formula}.  
In what follows, we remind the reader that the notations $\mathcal{R}(\omega; w)$ and $\mathcal{R}(\omega; F)$ were defined in \Cref{sec:subword_models}.

\begin{thm}[{\cite{AJS,Bil99}}]
\label{th:AJSB}
Let $v, \sigma \in S_{n}$ and fix a reduced word  $\omega = (i_{1}, \ldots, i_{\ell}) \in \red{\sigma}$.
Then letting $\sigma^{(p)}=s_{i_1}\cdots s_{i_p}$ we have
\[
\ev_{\sigma}(\schub{v})=\sum_{\pi \in \mathcal{R}(\omega; v)} \weight_\pi^{\circ}
\qquad\text{where}\qquad
\weight_{\pi}^{\circ} =\prod_{i_p\in \pi}(t_{\sigma^{(p-1)}(i_{p}+1)}-t_{\sigma^{(p-1)}(i_{p})}).
\]
\end{thm}
We refer the reader to \cite{TymSur} or \cite[Chapter 18]{AF24} for a geometric perspective on this result; see also \cite{GK21} for a generalization.

\begin{eg}
    For the permutations $v=312$ and $\sigma=321$, we have $\schub{312}(\xl_3;\tl_3)=(x_1-t_1)(x_1-t_2)$ hence $\ev_{\sigma}(\schub{v})=(t_3-t_2)(t_3-t_1)$. For $\omega =(2,1,2)=(i_1,i_2,i_3)$ we must have $\pi=(i_1,i_2)$ so we alternately verify by Theorem~\ref{th:AJSB} that $\ev_{\sigma}(\schub{v})=(t_3-t_2)(t_{s_2\cdot 2}-t_{s_2\cdot 1})=(t_3-t_2)(t_3-t_1)$.
\end{eg}

In this section we give a similar, AJS--Billey type formula for $\ev_{\sigma}$ applied to each double forest polynomial.

\begin{thm}
\label{thm:AJS_Billey_for_Forests}
Let $F \in \indexedforests_n$, $\sigma \in \NC_n$, and fix $\omega = (i_{1}, \ldots, i_{\ell})\in \rednc{\sigma}$.
Then letting $\sigma^{(p)}=s_{i_1}\cdots s_{i_p}$ we have
\[
\ev_{\sigma}(\forestpoly{F})=\sum_{\pi \in \mathcal{R}(\omega;F)} \weight_\pi^{\circ}
\qquad\text{where}\qquad
\weight_{\pi}^{\circ}=\prod_{i_p\in \pi}(t_{\sigma^{(p-1)}(i_{p}+1)}-t_{\sigma^{(p-1)}(i_{p})}).
\]
\end{thm}

The proof of Theorem~\ref{thm:AJS_Billey_for_Forests} is given in Section~\ref{subsec:ajs_billey_forests}.

\begin{eg}
    Consider $\sigma=4321 \in \NC_4$. Pick $\omega=(3,2,1,2,3,2)$, the solitary element in $\rednc{\sigma}$.
    Now suppose $F$ has code $(2,0,1,0,\dots)$.
    From Table~\ref{table:forestpolys} we know that
    \[
        \forestpoly{F}(\xl;\tl)= (x_1-t_1)(x_1-t_2)(x_2+x_3-t_1-t_2)
    \]
    In particular $\ev_\sigma\forestpoly{F}=(t_4-t_1)(t_4-t_2)(t_3-t_1)$. Let us check that we obtain the same result by invoking Theorem~\ref{thm:AJS_Billey_for_Forests}.
    We have $\sylv{F}=\{(2,1,3),(2,3,1)\}$. 
    Only $(2,1,3)$ appears as the subword $\pi=(i_2,i_3,i_5)\subset  \omega$.
    We thus compute 
    \[
        \ev_{\sigma}\forestpoly{F}=\weight_{\pi}^{\circ}=(t_{s_3\cdot 3}-t_{s_3\cdot 2})(t_{s_3s_2\cdot 2}-t_{s_3s_2\cdot 1})(t_{s_3s_2s_1s_2\cdot 4}-t_{s_3s_2s_1s_2\cdot 3})=(t_4-t_2)(t_4-t_1)(t_3-t_1).
    \]
\end{eg}

The rest of the section is laid out as follows: Section~\ref{sec:NoncrossingCombi} states some combinatorial results about noncrossing reduced words, after which Section~\ref{subsec:ajs_billey_forests} proves the main result.  
We conclude in Section~\ref{sec:AJS_consequences} by proving some consequences of Theorem~\ref{thm:AJS_Billey_for_Forests} which have parallels for Schubert polynomials, including Graham-positivity and upper-triangularity of evaluations.

\subsection{Bruhat combinatorics of noncrossing partitions}
\label{sec:NoncrossingCombi}

Recall that the (strong) Bruhat order on $S_{n}$ is given by $\sigma \le \tau$ if and only if some (equivalently any) reduced word $\omega \in \red{\sigma}$ occurs as a subword in some $\psi \in \red{\tau}$.  
In Section~\ref{subsec:relating_nc_and_forests} we introduced the notion of a noncrossing reduced word and established some basic properties of this definition.  
However, we have not yet considered (strong) Bruhat order in relation to noncrossing partitions, which is a key aspect of reduced word combinatorics.     
In this section we extend the combinatorics of noncrossing reduced words to the Bruhat order and develop what we need for the remainder of the paper.

We begin by reviewing the work of Gobet--Williams \cite{GobetWilliams}, who use the subword combinatorics of our long word $\longword$ to determine the Bruhat order on noncrossing partitions.  
A particularly nice feature of this result is that it (inadvertently) identifies a canonical noncrossing reduced word for each noncrossing partition.
Given $\sigma \in \NC_{n}$ and $i \in [n]$, let
\[
\operatorname{vert}_{i}(\sigma) = 
\begin{cases}
2 |\{ \text{cycles in $\sigma$ nesting $i$}\}| & \text{if $i$ is maximal in its cycle, and} \\
2 |\{ \text{cycles in $\sigma$ nesting $i$}\}| + 1 & \text{otherwise.}
\end{cases}
\]
For example, if $\sigma = 82154763$ as in Example~\ref{eg:noncrossing}, then $\operatorname{vert}(\sigma) = (1\, 2\, 1\, 3\, 2\, 3\, 2\, 0)$.

It will always be the case that $\operatorname{vert}_i(\sigma)\le 2i-1$. Since the total number of $i$'s (barred or unbarred) in $\longword$ is $2i-1$, the following are well-defined.
\begin{defn}
We define  $\ncrmin{\sigma}\subset \longword$ as the subword obtained by selecting, for each $i$ such that $\operatorname{vert}_i(\sigma)\ne 0$, the first $\operatorname{vert}_i(\sigma)$ instances of $i$.  
Similarly, define $\sylcont{\sigma}\subset \longword$ as the subword obtained by selecting, for each $i$ such that $\operatorname{vert}_i(\sigma)\ne 0$, all instances of $i$ at or beyond the first $\operatorname{vert}_i(\sigma)$ instances of $i$.  
Finally, let
\[
\sylmin{\sigma}=\ncrmin{\sigma}\cap \sylcont{\sigma},
\]
so that $\sylmin{\sigma}\subset \longword$ contains exactly the $\operatorname{vert}_i(\sigma)$'th instance of each letter $i$ for which $\operatorname{vert}_i(\sigma)\ne 0$.
\end{defn}

An example with $\sigma = 82154763$ as in Example~\ref{eg:noncrossing} is shown in Figure~\ref{fig:vert_code_nantel}; see Figure~\ref{fig:sylmin_example} for an accompanying pictorial description.

\begin{figure}[!h]
    \centering
 \begin{align*}
 \ncrmin{\sigma}=&\,
 {\color{purple!40!blue!80}7}\, 
 {\color{purple!40!blue!80}6}\, 
 {\color{purple!40!blue!80}5}\, 
 {\color{purple!40!blue!80}4}\, 
 {\color{purple!40!blue!80}3}\, 
 {\color{purple!40!blue!80}2}\, 
 {\color{purple!40!blue!80}1}\, 
 {\color{purple!40!blue!80}\overline{2}}\, {\color{lightgray}\overline{3}}\, 
 {\color{purple!40!blue!80}\overline{4}}\, 
 {\color{purple!40!blue!80}\overline{5}}\, 
 {\color{purple!40!blue!80}\overline{6}}\, 
 {\color{purple!40!blue!80}\overline{7}}\,\,\, 
 {\color{lightgray}7}\, 
 {\color{purple!40!blue!80}6}\, 
 {\color{lightgray}5}\, 
 {\color{purple!40!blue!80}4}\, 
 {\color{lightgray}3}\, 
 {\color{lightgray}2}\, 
 {\color{lightgray}\overline{3}}\, 
 {\color{lightgray}\overline{4}}\, 
 {\color{lightgray}\overline{5}}\, 
 {\color{lightgray}\overline{6}}\, {\color{lightgray}\overline{7}}\,\,\, 
 {\color{lightgray}7}\, 
 {\color{lightgray}6}\, 
 {\color{lightgray}5}\, 
 {\color{lightgray}4}\, 
 {\color{lightgray}3}\, 
 {\color{lightgray}\overline{4}}\, 
 {\color{lightgray}\overline{5}}\, 
 {\color{lightgray}\overline{6}}\, {\color{lightgray}\overline{7}}\,\,\, 
 {\color{lightgray}7}\, 
 {\color{lightgray}6}\, 
 {\color{lightgray}5}\, 
 {\color{lightgray}4}\, 
 {\color{lightgray}\overline{5}}\, 
 {\color{lightgray}\overline{6}}\, {\color{lightgray}\overline{7}}\,\,\, 
 {\color{lightgray}7}\, 
 {\color{lightgray}6}\, 
 {\color{lightgray}5}\, 
 {\color{lightgray}\overline{6}}\, {\color{lightgray}\overline{7}}\,\,\, 
 {\color{lightgray}7}\, 
 {\color{lightgray}6}\, 
 {\color{lightgray}\overline{7}} \,\,\,
 {\color{lightgray}7}\, \\[1ex]
 \sylcont{\sigma}=&\,
 {\color{lightgray}7}\, 
 {\color{lightgray}6}\, 
 {\color{lightgray}5}\, 
 {\color{lightgray}4}\, 
 {\color{black}3}\, 
 {\color{lightgray}2}\, 
 {\color{black}1}\, 
 {\color{black}\overline{2}}\, 
 {\color{black}\overline{3}}\, 
 {\color{lightgray}\overline{4}}\, 
 {\color{black}\overline{5}}\, 
 {\color{lightgray}\overline{6}}\, 
 {\color{black}\overline{7}}\,\,\, 
 {\color{black}7}\, 
 {\color{black}6}\, 
 {\color{black}5}\, 
 {\color{black}4}\, 
 {\color{black}3}\, 
 {\color{black}2}\, 
 {\color{black}\overline{3}}\, 
 {\color{black}\overline{4}}\, 
 {\color{black}\overline{5}}\, 
 {\color{black}\overline{6}}\, 
 {\color{black}\overline{7}}\,\,\, 
 {\color{black}7}\, 
 {\color{black}6}\, 
 {\color{black}5}\, 
 {\color{black}4}\, 
 {\color{black}3}\, 
 {\color{black}\overline{4}}\, 
 {\color{black}\overline{5}}\, 
 {\color{black}\overline{6}}\, 
 {\color{black}\overline{7}}\,\,\, 
 {\color{black}7}\, 
 {\color{black}6}\, 
 {\color{black}5}\, 
 {\color{black}4}\, 
 {\color{black}\overline{5}}\, 
 {\color{black}\overline{6}}\, 
 {\color{black}\overline{7}}\,\,\, 
 {\color{black}7}\, 
 {\color{black}6}\, 
 {\color{black}5}\, 
 {\color{black}\overline{6}}\, 
 {\color{black}\overline{7}}\,\,\, 
 {\color{black}7}\, 
 {\color{black}6}\, 
 {\color{black}\overline{7}} \,\,\,
 {\color{black}7}\, \\[1ex]
 \sylmin{\sigma}=&\,
 {\color{lightgray}7}\, 
 {\color{lightgray}6}\, 
 {\color{lightgray}5}\, 
 {\color{lightgray}4}\, 
 {\color{red}3}\, 
 {\color{lightgray}2}\, 
 {\color{red}1}\, 
 {\color{red}\overline{2}}\, 
 {\color{lightgray}\overline{3}}\, 
 {\color{lightgray}\overline{4}}\, 
 {\color{red}\overline{5}}\, 
 {\color{lightgray}\overline{6}}\, 
 {\color{red}\overline{7}}\,\,\, 
 {\color{lightgray}7}\, 
 {\color{red}6}\, 
 {\color{lightgray}5}\, 
 {\color{red}4}\, 
 {\color{lightgray}3}\, 
 {\color{lightgray}2}\, 
 {\color{lightgray}\overline{3}}\, 
 {\color{lightgray}\overline{4}}\, 
 {\color{lightgray}\overline{5}}\, 
 {\color{lightgray}\overline{6}}\, 
 {\color{lightgray}\overline{7}}\,\,\, 
 {\color{lightgray}7}\, 
 {\color{lightgray}6}\, 
 {\color{lightgray}5}\, 
 {\color{lightgray}4}\, 
 {\color{lightgray}3}\, 
 {\color{lightgray}\overline{4}}\, 
 {\color{lightgray}\overline{5}}\, 
 {\color{lightgray}\overline{6}}\, 
 {\color{lightgray}\overline{7}}\,\,\, 
 {\color{lightgray}7}\, 
 {\color{lightgray}6}\, 
 {\color{lightgray}5}\, 
 {\color{lightgray}4}\, 
 {\color{lightgray}\overline{5}}\, 
 {\color{lightgray}\overline{6}}\, 
 {\color{lightgray}\overline{7}}\,\,\, 
 {\color{lightgray}7}\, 
 {\color{lightgray}6}\, 
 {\color{lightgray}5}\, 
 {\color{lightgray}\overline{6}}\, 
 {\color{lightgray}\overline{7}}\,\,\, 
 {\color{lightgray}7}\, 
 {\color{lightgray}6}\, 
 {\color{lightgray}\overline{7}} \,\,\,
 {\color{lightgray}7}\,
 \end{align*}
    \caption{The subword with $\longword^{\sigma}_{\le vert}$ for $\sigma = 82154763$.}
    \label{fig:vert_code_nantel}
\end{figure}

\begin{prop}[{\cite[Theorem 1.2 and \S 6.1]{GobetWilliams}}]
\label{prop:gobet_williams_bruhat}
For each $\sigma \in \NC_{n}$, we have $\ncrmin{\sigma}\in\rednc{\sigma}$. 
Moreover, given also any $\tau\in \NC_n$ we have 
\[
\tau\le \sigma
\quad\text{if and only if}\quad
\ncrmin{\tau}\subset\ncrmin{\sigma}
 \quad\text{if and only if}\quad
 \operatorname{vert}_i(\tau)\le \operatorname{vert}_i(\sigma)\text{ for all }1\leq i\leq  n.
\]
\end{prop}

\begin{rem}
    Proposition~\ref{prop:gobet_williams_bruhat} and the statements in \cite{GobetWilliams} differ in a few superficial ways, the most notable being that the results of~\cite{GobetWilliams} describe the set $\{\sigma^{-1}  \;|\; \sigma \in \NC_{n}\}$, and accordingly each reduced word appears in the opposite order (see~\cite[Rem.~3.5]{BeGa23} for a more comprehensive translation).
\end{rem}

Now recall that by Proposition~\ref{prop:noncrossing_commutation_class}, every element of $\rednc{\sigma}$ can be obtained from $\ncrmin{\sigma}$ using commutation relations only.  
We use this to reduce many properties of noncrossing reduced words to properties of $\ncrmin{\sigma}$, and we now collect some of these properties for later use.

\begin{lem}
\label{lem:canoncial_word_facts}
Let $\sigma \in \NC_{n}$.
\begin{enumerate}[label=(\arabic*)]
\item \label{8.8(1)} For any $i \in \desnc{\sigma}$ with $i > 1$, any pair of letters $i-1$ and $i+1$ appearing in a noncrossing reduced word for $\sigma$ must have an $i$ in between them (in this statement we ignore all bar decorations).

\item \label{8.8(2)} If $\sigma = \ForToNC(F)$, then $\sylmin{\sigma} \in \mc{R}(\longword; F)$.

\item \label{8.8(3)} Every $\pi \in \mc{R}(\longword; F)$ is contained in $\sylcont{\pi}$.
\end{enumerate}
\end{lem}
\begin{proof}
Take $i \in \desnc{\sigma}$ with $i > 1$.  As commutation relations cannot move an $i$ past an $i-1$ or $i+1$, it is sufficient by \Cref{prop:noncrossing_commutation_class} to verify Statement~\ref{8.8(1)} for $\ncrmin{\sigma}\in\rednc{\sigma}$.  
This will follow from the fact that either $\operatorname{vert}_{i-1}(\sigma)<\operatorname{vert}_{i}(\sigma)$ or both $\operatorname{vert}_{i-1}(\sigma)=\operatorname{vert}_{i}(\sigma)$ and $\operatorname{vert}_{i+1}(\sigma) < \operatorname{vert}_{i}(\sigma)$, which we verify in cases.  
Let $C_{1}, C_{2}, \ldots, C_{s}$ denote the cycles in $\sigma$ which nest $i$ ordered so that $C_{j+1}$ is nested in $C_j$ for all $j$, and denote $C$ for the cycle containing $i$, which we note is nested in $C_s$. Then $\operatorname{vert}_{i}(\sigma) = 2s+\delta_{i\ne \max C}$. 
Each of $i-1$ and $i+1$ must be nested in $C_1,\ldots,C_{s-1}$, and the only additional cycles that can possibly nest them are $C_s$ or $C$.
But by~\Cref{prop:NCsidec} we know that $i=\min C$ and $i+1$ is not minimal in the cycle that contains it, which implies that neither $i-1$ nor $i+1$ can be nested by $C$. 
If $i-1\in C_s$, then $i-1$ is nested precisely by $C_1,\ldots,C_{s-1}$, and so we conclude $\operatorname{vert}_{i-1}(\sigma) \le 2s - 1<2s\le \operatorname{vert}_{i}(\sigma)$. 
If on the other hand $i-1\not\in C_s$, then $i-1$ must be the largest element of its cycle (as otherwise this cycle nests $i$ below $C_s$) so $\operatorname{vert}_{i-1}(\sigma)\le 2s\le \operatorname{vert}_{i}(\sigma)$. We note for use in a later proof that at this point we have established in all cases that whenever $i\in \desnc{\sigma}$, we have
\begin{equation}
\label{eqn:helpfulineqs}\operatorname{vert}_{i}(\sigma)-2\le\operatorname{vert}_{i-1}(\sigma)\le \operatorname{vert}_{i}(\sigma)\text{ and }\operatorname{vert}_{i+1}(\sigma)\le \operatorname{vert}_i(\sigma)+1.
\end{equation}
Here it remains to analyze the case $\operatorname{vert}_{i-1}(\sigma)=\operatorname{vert}_{i}(\sigma)$. But this is only possible if $\operatorname{vert}_{i}(\sigma) = 2s$, and so $i=\max C$. Furthermore by \Cref{prop:NCsidec} we have $i+1$ is not the minimal element of its cycle, so $C$ must be nested in this cycle. We conclude that $i+1$ is not nested by all of the cycles $C_1,\ldots, C_s$ that nest $C$ and so $i+1$ is nested by exactly $C_1,\ldots, C_{s-1}$ which implies $\operatorname{vert}_{i+1}(\sigma) \le 2s-1<2s\le \operatorname{vert}_{i}(\sigma)$.

For the remaining statements, fix a subword $\pi \subset \longword$ whose values are the canonical labels of $\internal{F}$ each appearing exactly once, and consider the function $\operatorname{ht}_{\pi}: \internal{F} \to [2n-1]$ given by
\[
\operatorname{ht}_{\pi}(i) = \begin{cases}
2r -1 & \text{if $\pi$ contains $i^{(r)}$} \\
2r & \text{if $\pi$ contains $\overline{i}^{(r)}$}
\end{cases}
\]
In terms of the planar array for $\longword$, the function $\operatorname{ht}_{\pi}$ sends $i$ to the index of the row from which $\pi$ selects $i$.   
The function $\operatorname{ht}_{\pi}$ completely determines $\pi$, and $\pi$ is a Sylvester word of $F$ if and only if, for each parent-child pair $i$ and $j$ in $\internal{F}$, we have one of:
\begin{enumerate}[label = (\roman*)]
\item $\operatorname{ht}_{\pi}(i) < \operatorname{ht}_{\pi}(j)$, or

\item $j < i$ and $\operatorname{ht}_{\pi}(j) = \operatorname{ht}_{\pi}(i)$ is even, or

\item $i < j$ and $\operatorname{ht}_{\pi}(j) = \operatorname{ht}_{\pi}(i)$ is odd.
\end{enumerate}

We verify directly that the function $i \mapsto \operatorname{vert}_{i}(\sigma)$ satisfies these conditions.  
For a parent and child within the same component of $\ForToNC(F)$, both elements are non-maximal in some cycle of $\sigma$ so the values are equal and even (case (ii)).  
For the left child of a vertex which is removed by $\ForToNC$ during left edge deletion, the parent is the maximum value of a cycle which contains the child, so the values differ by one (case (i)).  
For a removed vertex and its right child, we have two maximal elements in disjoint cycles of $\sigma$ which are nested beneath the same set of larger cycles, so the values are equal and odd (case (iii)).  
Finally, for a removed vertex and a parent which is not removed by $\ForToNC$, the parent is in a cycle which nests the child and the child is maximal in its (disjoint) cycle, so the values differ by one (case (i) again).  This proves Statement~\ref{8.8(2)}.  

For Statement~\ref{8.8(3)} we show that any $\pi$ satisfying the conditions above has $\operatorname{ht}_{\pi}(i) \ge \operatorname{vert}_{i}(\sigma)$ for all $i \in \internal{F}$.  
This follows from the preceding argument: $\operatorname{vert}_{i}(\sigma)$ assigns the minimum allowed value at all internal nodes, only increasing when moving between a parent and child pair $i$ and $j$ with $\operatorname{ht}_{\pi}(i)$ odd and $i > j$ or $\operatorname{ht}_{\pi}(i)$ even and $j > i$ as required.
\end{proof}

We can visualize Lemma~\ref{lem:canoncial_word_facts}\ref{8.8(2)} by constructing the vine diagram corresponding to the Sylvester word $\sylmin{\sigma}$, as shown in Figure~\ref{fig:sylmin_example}.  
The reader will readily identify this vine diagram with the heap in Figure~\ref{fig:NC_heap} by treating the midpoint of \textit{every} box in $\ncrmin{\sigma}$ as a vertex (rather than just those with ``T''-tiles), illustrating the  visual intuition behind ~\Cref{prop:heapclaims}.
\begin{figure}
\begin{center}
\begin{tikzpicture}[scale = 0.6, baseline = 0]
\path[fill = purple!20!blue!20!white, xshift = -1cm] (1, 0) -- (1, -1) -- (2, -1) -- (2, -2) -- (3, -2) -- (3, -1) -- (4, -1) -- (4, -3) -- (5, -3) -- (5, -2) -- (6, -2) -- (6, -3) -- (7, -3) -- (7, -2) -- (8,-2) -- (8, 0) -- cycle;
\draw[color = gray] (-1, 0) grid (8, -4);
\draw[thin] (0, 0.5) node[left] {$\operatorname{vert}(\sigma) = $};
\foreach \x/\vc in {1/1, 2/2, 3/1, 4/3, 5/2, 6/3, 7/2, 8/0}{
    \draw[thin] (\x-0.5, 0.5) node {$\vc$};
    }
\draw[thin] (0, -4.5) node[left] {$\sigma =$};
\foreach \x/\sigmai in {1/8, 2/2, 3/1, 4/5, 5/4, 6/7, 7/6, 8/3}{
    \draw[thin] (\x-0.5, -4.5) node {$\sigmai$};
    }
\foreach \x/\y in {0/1, 0/3, 1/3, 2/3, 3/3, 5/3, 7/3, 8/1, 8/3}{
    \draw[thick] (\x-1, 0.5-\y) to[out = 0, in = -90] (\x-0.5, 1-\y);
    \draw[thick] (\x, 0.5-\y) to[out = -180, in = 90] (\x-0.5, -\y);
    }
\foreach \x/\y in {0/2, 0/4, 1/2, 1/4, 2/4, 3/2, 3/4, 4/4, 5/4, 6/4, 7/4, 8/2, 8/4}{
    \draw[thick] (\x-1, 0.5-\y) to[out = 0, in = 90] (\x-0.5, -\y);
    \draw[thick] (\x, 0.5-\y) to[out = -180, in = -90] (\x-0.5, 1 -\y);
    }
\foreach \x/\y in {1/1, 2/2, 3/1, 4/3, 5/2, 6/3, 7/2, 2/1, 4/1, 4/2, 5/1, 6/1, 6/2, 7/1}{
    \draw[thick] (\x, 0.5-\y) -- (\x-1, 0.5-\y);
    \draw[thick] (\x-0.5, 1-\y) -- (\x-0.5, -\y);
    }
\draw[thick, red, xshift = -1cm] (1, 0) -- (1, -1) -- (2, -1) -- (2, -2) -- (3, -2) -- (3, -1) -- (4, -1) -- (4, -3) -- (5, -3) -- (5, -2) -- (6, -2) -- (6, -3) -- (7, -3) -- (7, -2) -- (8,-2) -- (8, 0);
\end{tikzpicture}
\hspace{1cm}
\begin{tikzpicture}[scale = 0.6, baseline = 0]
\path[fill = purple!20!blue!20!white, xshift = -1cm] (1, 0) -- (1, -1) -- (2, -1) -- (2, -2) -- (3, -2) -- (3, -1) -- (4, -1) -- (4, -3) -- (5, -3) -- (5, -2) -- (6, -2) -- (6, -3) -- (7, -3) -- (7, -2) -- (8,-2) -- (8, 0) -- cycle;
\draw[color = gray] (-1, 0) grid (8, -4);
\draw (0, 0.5) node[left] {$\operatorname{vert}(\sigma) = $};
\foreach \x/\vc in {1/1, 2/2, 3/1, 4/3, 5/2, 6/3, 7/2, 8/0}{
    \draw[thin] (\x-0.5, 0.5) node {$\vc$};
    }
\foreach \x/\y in {0/1, 0/3, 1/3, 2/3, 3/3, 5/3, 7/3, 8/1, 8/3}{
    \draw[thick] (\x-1, 0.5-\y) to[out = 0, in = -90] (\x-0.5, 1-\y);
    \draw[thick] (\x, 0.5-\y) to[out = -180, in = 90] (\x-0.5, -\y);
    }
\foreach \x/\y in {0/2, 0/4, 1/2, 1/4, 2/4, 3/2, 3/4, 4/4, 5/4, 6/4, 7/4, 8/2, 8/4}{
    \draw[thick] (\x-1, 0.5-\y) to[out = 0, in = 90] (\x-0.5, -\y);
    \draw[thick] (\x, 0.5-\y) to[out = -180, in = -90] (\x-0.5, 1 -\y);
    }
\foreach \x/\y in {1/1, 2/2, 3/1, 4/3, 5/2, 6/3, 7/2}{
    \draw[thick] (\x, 0.5-\y) -- (\x-1, 0.5-\y);
    \draw[thick] (\x-0.5, 0.5-\y) -- (\x-0.5, -\y);
    }
\foreach \x/\y in {2/1, 4/1, 4/2, 5/1, 6/1, 6/2, 7/1}{
    \draw[thick] (\x, 0.5-\y) -- (\x-1, 0.5-\y);
    }
\end{tikzpicture}
\end{center}
\caption{The (permutation) vine diagram associated to the subword $\ncrmin{\sigma}$ and the (forest) vine diagram associated to the subword $\sylmin{\sigma}$ for $\sigma = 82145763$.}
\label{fig:sylmin_example}
\end{figure}

\begin{proof}[Proof of~\Cref{prop:heapclaims}]
We first prove that every element covers at most one other element in $\preceq$.  
To this end suppose that we have $1 \le a < b < c \le \ell$ with $a$ and $b$ both covered by $c$ in $\preceq$.  
By definition of the heap order, we must have $|\omega_{a} - \omega_{c}| = |\omega_{b} - \omega_{c}| = 1$, and by assumption we must have $\omega_{a} \neq \omega_{c}$.
Without loss of generality we take $\omega_{a} = \omega_{c} -1$ and $\omega_{b} = \omega_{c} + 1$.  
Then by~\Cref{lem:canoncial_word_facts}\ref{8.8(1)} the letter $\omega_{a}$ must appear in some position $c'$ between positions $a$ and $b$, implying that $a \prec c' \prec c$, a contradiction.

For the second claim, $\sylmin{\sigma}$ is a Sylvester word for $F$, so we know that both graphs are defined on the same underlying set, and we directly verify that each edge in $F$ corresponds to a sequence of covers in the heap order.   
For each non-maximal element $i$ of a cycle in $\sigma$, let $k$ be the next-largest element.  
Then $\operatorname{vert}_{i}(\sigma)$ is odd, $\operatorname{vert}_{k}(\sigma) > \operatorname{vert}_{i}(\sigma)$ for all $j$ between $i$ and $k$, and $\operatorname{vert}_{k}(\sigma) = \operatorname{vert}_{i}(\sigma)$ or $\operatorname{vert}_{i}(\sigma) - 1$.  
Therefore the final occurrence of $i$ and $j$ in $\ncrmin{\sigma}$ appears in a sequence $j \mathbf{b} (j+1) (j+2) \cdots i$ in which $\mathbf{b}$ contains only letters strictly greater than $j$, giving a sequence of covers in $\preceq$ from $j$ to $i$.  
Following a similar argument, we can construct a chain of covers in $\preceq$ from the largest element of each nested cycle to largest element of the nesting cycle which is still smaller than the nested cycle.  As these are precisely the internal edges of $F$, the proof is complete.  
\end{proof}

\subsection{Proof of AJS--Billey for double forest polynomials}
\label{subsec:ajs_billey_forests}

In this section we prove Theorem~\ref{thm:AJS_Billey_for_Forests}.  
We begin by introducing an inductive description of the interaction of the trimming operator $\eope{i}$ and the evaluation operator $\ev_{\sigma}$, proceed to the proof, and conclude with some examples.
  
\begin{defn}
    For a subset $S\subset [n]$ and $a\in [n]$ which may or may not lie in $S$, we write $$S(\wh{a})=\{i-\delta_{i>a}\suchthat i\in S\setminus a\}.$$
    For a $\mc{P}\in \ncpart_{n}$, let $\mc{P}(\wh{a})$ be the noncrossing partition obtained by replacing each block $S$ with $S(\wh{a})$. 
    If $\sigma\coloneqq \sigma(\mc{P})$, we let $\sigma_{\wh{a}}\in \NC_{n-1}$ be the  noncrossing partition corresponding to $\mathcal{P}(\wh{a})$. 
    
    For a noncrossing reduced word $\omega$, let $\omega({\wh{a}})$ be obtained by removing all instances of $a$ in $\omega$ and then decreasing all remaining elements which are larger than $a$ by $1$.
\end{defn}

\begin{eg}
Take $\sigma = 82154763 \in \NC_{8}$ from Example~\ref{eg:noncrossing_reduced_words}, with noncrossing reduced word $\omega = (7, 6, 5, 4, 3, 2, 1, 4, 5, 4, 6, 7, 6, 2) $.  
Then with $a = 4$ we have $\omega(\wh{a}) = (6, 5, 4, 3, 2, 1, 4, 5, 6, 5, 2)$,
\[
\mathcal{P}(\sigma)= \begin{tikzpicture}[scale = 0.75, baseline = 0.75*-0.2]
\foreach \x in {1, ..., 8}{\draw[fill] (\x - 1, 0) node[inner sep = 2pt] (\x) {$\scriptstyle \x$};}
\foreach \i\j in {1/3, 3/8, 4/5, 6/7}{\draw[thick] (\i) to[out = 35, in = 145] (\j);}
\end{tikzpicture}
,\quad\text{and}\qquad
\mathcal{P}(\sigma_{\wh{a}})= \begin{tikzpicture}[scale = 0.75, baseline = 0.75*-0.2]
\foreach \x in {1, ..., 7}{\draw[fill] (\x - 1, 0) node[inner sep = 2pt] (\x) {$\scriptstyle \x$};}
\foreach \i\j in {1/3, 3/7, 5/6}{\draw[thick] (\i) to[out = 35, in = 145] (\j);}
\end{tikzpicture}
\]
so that $\sigma_{\wh{a}} = 7214653$.
\end{eg}

\begin{prop}
\label{prop:Removenoncrossing}
For all $\sigma \in \NC_{n}$, $a \in \desnc{\sigma}$, and $\omega\in \rednc{\sigma}$ we have $\omega({\wh{a}})\in \rednc{\sigma_{\wh{a}}}$.
\end{prop}
\begin{proof}
Proposition~\ref{prop:noncrossing_commutation_class} states that there exists a sequence of noncrossing reduced words $\ncrmin{\sigma} = \psi^{(0)}, \psi^{(1)}, \ldots, \psi^{(k)}, \psi^{(k+1)} = \omega$ such that each $\psi^{(i)}$ and $\psi^{(i+1)}$ differ by a single commutation move.  
Using induction on $k$, we prove the statement that $\omega(\wh{a}) = \psi^{(k+1)}(\wh{a})\in \rednc{\sigma_{\wh{a}}}$.  

For the base case, let $\omega = \ncrmin{\sigma}$.  
We first compute the vertical code from the combinatorial description of $\mc{P}(\sigma_{\wh{a}})$.
Proposition~\ref{prop:NCsidec} implies that any arcs in $\mc{P}$ connected to $a$ do not contribute to the vertical code of $\mc{P}$, so we have
\[
\operatorname{vert}_{b}(\sigma_{\wh{a}}) = \begin{cases}
\operatorname{vert}_{b}(\sigma) & \text{if $b < a$,} \\
\operatorname{vert}_{b+1}(\sigma) & \text{if $b \ge a$.} 
\end{cases}
\]
Thus $\ncrmin{\sigma_{\wh{a}}} = \ncrmin{\sigma}(\wh{a})$; pictorially we can see this as taking the subset of the planar array for $\longword$ corresponding to the subword $\ncrmin{\sigma}$, removing all boxes in column $a$, and the shifting all boxes to the right of this column one step to the left.  

Now assume that $\psi^{(k)}(\wh{a}) \in \rednc{\sigma_{\wh{a}}}$.  
Let $i$ denote the index at which the commutation between $\psi^{(k)}$ and $\omega$ takes place, so that $\omega_{i} = \psi^{(k)}_{i+1}$ and $\omega_{i+1} = \psi^{(k)}_{i}$ and $\omega_{j} = \psi^{(k)}_{j}$ for all $j \notin \{i, i+1\}$.  
If $a \in \{\omega_{i}, \omega_{i+1}\}$ then $\omega(\wh{a}) = \psi^{(k)}(\wh{a})$ and the proof is complete.  
Otherwise $|\omega_{i} - \omega_{i+1}| > 1$, then $|\omega_{i} - \delta_{\omega_{i} > a} - \omega_{i+1} + \delta_{\omega_{i+1}>a}| > 1$ unless possibly $\{\omega_i,\omega_{i+1}\}=\{a-1,a+1\}$, which cannot happen by by Lemma~\ref{lem:canoncial_word_facts}~\ref{8.8(1)} (as the $a$ between $a-1$ and $a+1$ initially in $\psi^{(k+1)}$ remains between $a-1$ and $a+1$ after each commutation). Thus $\omega(\wh{a})$ and $\psi^{(k)}(\wh{a})$ differ by a single commutation.  
Thus by Proposition~\ref{prop:noncrossing_commutation_class}, $\omega(\wh{a}) \in \rednc{\sigma_{\wh{a}}}$.
\end{proof}

\begin{proof}[Proof of Theorem~\ref{thm:AJS_Billey_for_Forests}]
We proceed by induction on the length $\ell = \ell(\sigma)$.  
The case $\ell=0$ is trivial, so we assume $\ell\ge 1$. 
Throughout, we set $a \coloneqq i_{\ell}$, the final letter of our fixed noncrossing reduced word $\omega \in \rednc{\sigma}$.  
While there may be other instances of the letter $a$ in $\omega$, we are solely concerned with the final instance in $\omega$ in the following.

Since $\omega\in \rednc{\sigma}$, we have $\sigma s_{a} \in \NC_{n}$. 
Lemma~\ref{lem:NCsi} implies $a\in\{\sigma(a),\sigma(a+1)\}$, so the $\eope{a}$ recursion then gives 
\begin{align}
\label{eq:two_cases_recursion}
\ev_{\sigma}\forestpoly{F} = 
\begin{cases}
\ev_{\sigma s_{a}}\forestpoly{F} + 
(t_{\sigma(a)}-t_{\sigma(a+1)}) \left((\ev_{\sigma_{\wh{a}}} \forestpoly{F/a})(\wh{\tl}_{a})\right) & \text{if $a \in \qdes{F}$}, \\
\ev_{\sigma s_{a}}\forestpoly{F} & \text{otherwise,}
\end{cases}
\end{align}
where  $(\ev_{\sigma_{\wh{a}}} \forestpoly{F/a})(\wh{\tl}_{a})$ is obtained from $\ev_{\sigma_{\wh{a}}} \forestpoly{F/a}$  by replacing all instances of $t_b$ where $b \geq a$ with $t_{b+1}$.  
We now apply our inductive hypothesis to each case on the right hand side above.

Consider first the case where $a \notin \qdes{F}$. 
This means that no word in $\sylv{F}$ ends in $a$.
Let $\omega' \coloneqq  (i_{1}, \dots, i_{\ell-1})$, which must belong to $\rednc{\sigma s_{a}}$.  
We have by induction that
\begin{align}
    \ev_{\sigma s_a}\forestpoly{F}=\sum_{\pi\in \mc{R}(\omega';F)}\weight_{\pi}^{\circ}
\end{align}
As the $j$th letter of $\omega$ and $\omega'$ agree for all $j < \ell$, $\mc{R}(\omega';F) = \mc{R}(\omega;F)$ and we thus conclude
\[
\ev_{\sigma}(\forestpoly{F})=\ev_{\sigma s_{a}}(\forestpoly{F}) 
= \sum_{\pi\in \mc{R}(\omega';F)}\weight_{\pi}^{\circ}
=\sum_{\pi\in \mc{R}(\omega;F)}\weight_{\pi}^{\circ}.
\]

Now suppose $a \in \qdes{F}$, so that every word in $\sylv{F}$ contains $a$.
Like before, the sum of $\weight_{\pi}^{\circ}$ over all subwords $\pi\subset \omega$ that do not use the last letter in $\omega$ gives $\ev_{\sigma s_a}\forestpoly{F}$, while the remaining terms are divisible by $(t_{\sigma(a)}-t_{\sigma(a+1)})$, so 
\[
\sum_{\pi\in \mc{R}(\omega;F)} \weight_{\pi}^{\circ} 
= \ev_{\sigma s_a}\forestpoly{F} + (t_{\sigma(a)}-t_{\sigma(a+1)})\sum_{\substack{\pi\in \mc{R}(\omega;F)\\ \text{final $a \in \pi$}}}  \prod_{\substack{i_p\in \pi\\ p\neq \ell}} (t_{\sigma^{(p-1)}(i_p+1)}-t_{\sigma^{(p-1)}(i_p)}).
\]
Considering the second term above, we claim that the map $\pi \to \pi(\wh{a})$ gives a bijection from the set $\{ \pi\in \mc{R}(\omega;F) \suchthat \text{final $a \in \pi$}\}$ to $ \mc{R}(\omega(\hat{a});F/a)$.
From Observation~\ref{obs:SylvesterTrimming}, removing the last letter from $\pi$ and decrementing all letters strictly larger than $a$ by $1$ produces a Sylvester word $\pi(\wh{a})$ for $F/a$.
Furthermore, $\pi(\wh{a})$ is a subword of $\omega(\wh{a})$, which by Proposition~\ref{prop:Removenoncrossing} belongs to $\rednc{\sigma_{\wh{a}}}$.

To describe the effect of our bijection on the weight of each subword, let $g:\NN\to\NN$ be the map sending $x\mapsto x+\delta_{x\ge a}$.  Thus $t_{g(i)}$ is the $i$th element of the variable set $\hat{\tl}_{a}$.  
For each fixed $1 \le p < \ell$ with $i_p\ne a$ (which includes all $i_p$ such that $i_p\in \pi$ since only the final letter of $\pi$ is $a$), we can write $(i_{1} i_{2} \cdots i_{p})(\wh{a}) = j_{1}j_{2} \cdots j_{q}$ so that $g(j_1)\cdots g(j_q)$ is the sublist of $i_1\cdots i_p$ where we omit all instances of $a$, and $i_p=g(j_q)$.
By~\Cref{prop:Removenoncrossing}, $j_{1}j_{2} \cdots j_{q-1}\in \rednc{(\sigma_{\wh{a}})^{(q-1)}}$, and we claim that 
\begin{equation}
\label{eq:outstandingclaim}
\sigma^{(p-1)}(i_{p}) = g\Big((\sigma_{\wh{a}})^{(q-1)}(j_{q})\Big)
\qquad\text{and}\qquad
\sigma^{(p-1)}(i_{p} + 1) = g\Big((\sigma_{\wh{a}})^{(q-1)}(j_{q}+1)\Big),
\end{equation}
which we prove in Lemma~\ref{lem:outstandingclaim} below.

Once this has been proved, it follows that 
\begin{align*}
\sum_{\substack{\pi\in \mc{R}(\omega;F)\\ \text{final $a \in \pi$}}} \weight_{\pi}^{\circ}
&= (t_{\sigma(a)}-t_{\sigma(a+1)}) \sum_{\pi(\wh{a}) \in \mc{R}(\omega(\hat{a});F/a)} \quad \prod_{j_q\in \wh{\pi}} \Big(t_{g\big((\sigma_{\wh{a}})^{(q-1)}(j_q+1)\big)}-t_{g\big((\sigma_{\wh{a}})^{(q-1)}(j_q)\big)}\Big) \\
&=(t_{\sigma(a)}-t_{\sigma(a+1)})\;\big(\ev_{\sigma_{\wh{a}}} \forestpoly{F/a}\big)(\wh{\tl}_{a}),
\end{align*}
which in view of the previous considerations completes the proof.
\end{proof}

\begin{lem}
\label{lem:outstandingclaim}
If $\sigma \in \NC_{n}$, $a \in \desnc{\sigma}$, and $\omega = i_{1} \cdots i_{\ell} \in \rednc{\sigma}$ with $i_{\ell} = a$, then for any $1 \le p < \ell$ with $i_{p} \neq i_{\ell}$, we have:
\begin{itemize}
    \item $i_p\ne a-1$ and $\sigma^{(p-1)} g(j) = g \big((\sigma_{\wh{a}})^{(q-1)}(j)\big)$ for all $j$, or
    \item $i_p\ne a+1$ and $\sigma^{(p-1)} s_{a} g(j) = g \big((\sigma_{\wh{a}})^{(q-1)}(j)\big)$ for all $j$,
\end{itemize}  
where $g(x) = x + \delta_{x \ge i_{\ell}}$ and $(i_{1} i_{2} \cdots i_{p})(\wh{a}) = j_{1}j_{2} \cdots j_{q}$ as in the proof of Theorem~\ref{thm:AJS_Billey_for_Forests} above.
\end{lem}
This lemma implies Equation~\eqref{eq:outstandingclaim} by taking $j = j_{q}$ and $j_{q}+1$: if $j_{q} \ne a-1$ then $g(j_{q} + 1) = g(j_{q}) + 1 = i_{p} + 1$, and if $j_{q} \neq a$ then $s_{a} g(j_{q}) = g(j_{q}) = i_{p}$ and $s_{a} g(j_{q}+1) = g(j_{q}) + 1 = i_{p} + 1$.
\begin{proof}
Let $a = i_{\ell}$.  We begin by observing that $g s_{a-1} = s_{a} s_{a-1} s_{a} g$ and $g s_{r} = s_{g(r)} g$ for all $r \neq a-1$.  
Therefore if $i_1'\cdots i_{s}'$ is a list in which all instances of $a-1$ and $a$ occur as a disjoint union of triples of the form $a(a-1)a$, then $s_{i_1'}\cdots s_{i_{s}'} g=gs_{j_1}\cdots s_{j_{q-1}}$, where $j_1\cdots j_{q-1}$ is the associated depletion $(i_1'\cdots i_{s'}')(\wh{a})$.  
We therefore construct a (possibly nonreduced) word $i_1'\cdots i_{s}'$ for $\sigma^{(p-1)}$ or $\sigma^{(p-1)}s_{a}$ whose depletion is exactly $j_{1} \cdots j_{q}$.  
This is accomplished by modifying either $i_{1} \cdots i_{p}$ or $i_{1} \cdots i_{p}a$ using two operations which do not change the associated permutation or depletion: commutation moves of the form $ab\mapsto ba$ with $|a-b|\ge 2$ and insertions or removing instances of $aa$.

Consider the subsequence of $i_{1} i_{2} \cdots i_{p}$ containing all instances of $a-1$ and $a$.  
Commutation relations do not affect this subsequence, so it is a truncation of the analogous subsequence for $\ncrmin{\sigma}$.  
Because $a\in \desnc{\sigma}$, we have the inequalities $\operatorname{vert}_{a}(\sigma) -2 \le \operatorname{vert}_{a-1}(\sigma) \le \operatorname{vert}_{a}(\sigma)$ and $\operatorname{vert}_{a+1}(\sigma) \le \operatorname{vert}_{a}(\sigma) + 1$ observed in the proof of Lemma~\ref{lem:canoncial_word_facts}\ref{8.8(1)} in equation \eqref{eqn:helpfulineqs}. It follows that this subsequence begins with $a$, has some number of repetitions of $a-1, a-1, a, a$, and then possibly ends with a prefix of $a-1, a-1, a, a$ or $a-1, a, a$.  Moreover every $a+1$ in our word appears between two consecutive $a$'s in this sequence, except possibly before the first $a$, or when the subsequence ends in $a-1, a$, after the last $a$.  

We now proceed to modify the word $i_{1} i_{2} \cdots i_{p-1}$ as described above (adding an extra $a$ at the end in one of the cases).  
First, for each $a, a-1, a-1, a$ or $a, a-1, a$ in our subsequence of $a$'s and $a-1$'s, we can use commutation relations to move the $a$'s until they are adjacent to the $a-1$'s in the original list.  
Next, any two $a-1$'s with no $a$ in between them must be separated entirely in the original list by letters $b$ with $|a-b|\ge 2$, so we can introduce two copies of $a$ at any point in this interval; we do so directly next to each $a-1$. 
The result is a word for $\sigma^{(p-1)}$ in which all $a$'s and $a-1$'s occur in triples of adjacent letters $a(a-1)a$, possibly followed by a ``remainder'' of $a$ or $a(a-1)$ (also adjacent). 
If there is no remainder, we have obtained a word of the desired form for $\sigma^{(p-1)}$; in this case note that the final letter $i_{p}$ cannot be $a-1$ as there is a single $a$ following the previous instance of $a-1$ in our word.

If there is a remainder of either $a$ or $a-1$, then the letters between this remainder and $i_{p}$ (including $i_{p}$) are not $a+1$, as this violates the pattern of $a$, $a-1$, and $a+1$ described above.  
We can therefore move the letter $a$ between the end of our word to the end of the remainder using commutation relations.  
If the remainder is $a$, we move it to the end so that our word ends in $aa$, which we can remove.  
If the remainder is $a(a-1)$, we move the $a$ from the end of our word to the end of the remainder to form an $a(a-1)a$. 
\end{proof}

\subsection{Consequences of the AJS--Billey formula for forests}
\label{sec:AJS_consequences}

We now show that our evaluations of double forest polynomials share two important properties with the actual AJS--Billey formula.  First, recall from the introduction that $a(\tl)$ is Graham-positive if it lies in $\ZZ_{\ge 0}[t_2-t_1,t_3-t_2,\ldots]$.
The previous examples suggest that evaluations of double forest polynomials at noncrossing partitions are Graham-positive. 
Our next result establishes this in general.

\begin{cor}
\label{cor:graham_positivity_forests}
The evaluation of a forest polynomial $\forestpoly{F}$ at a noncrossing partition $\sigma$ is Graham-positive.
\end{cor}
\begin{proof}
Fix $\omega = (i_{1}, \ldots, i_{\ell})\in \rednc{\sigma}$.
For $1 < p \le \ell$, we have that $i_p$ is not a descent for $\sigma^{(p-1)}$. 
So $\sigma^{(p-1)}(i_p+1) > \sigma^{(p-1)}(i_p)$.  
It follows that every factor in $\weight_{\pi}$ is of the form $t_{b} - t_{a}$ for $b > a$, where $\pi\in\mc{R}(\omega;F)$.
\end{proof}

Now, recall that evaluations of double Schubert polynomials are upper triangular with respect to the Bruhat order, in that $\ev_{\sigma}(\schub{v}) = 0$ if and only if $v \not\le \sigma$.  
This is immediate from the subword criterion for Bruhat comparison, but our analogous result requires some more legwork.  
Recall the bijection $\ForToNC: \indexedforests_{n} \to \NC_{n}$ given in Section~\ref{subsec:relating_nc_and_forests}.

\begin{thm}
\label{thm:AJSB_triangularity}
Let $\sigma\in \NC_n$ and $F\in \indexedforests_n$. Then $\ev_{\sigma}(\forestpoly{F}) = 0$ if and only if $\ForToNC(F) \not\le \sigma$ in the Bruhat order.
\end{thm}
\begin{proof}[Proof of~\Cref{thm:AJSB_triangularity}]
By \Cref{prop:gobet_williams_bruhat}, and the fact that a sum of Graham-positive nonzero expressions is nonzero, we know that $\ncrmin{\sigma}$ is a noncrossing reduced word for $\sigma$. By Theorem~\ref{thm:AJS_Billey_for_Forests} it suffices to show that
\[
\ForToNC(F) \leq \sigma \Longleftrightarrow \text{$\mc{R}(\ncrmin{\sigma};F)$ is nonempty}.
\]

First suppose that $\ForToNC(F) \le \sigma$.
By Proposition~\ref{prop:gobet_williams_bruhat} we know that $ \ncrmin{\ForToNC(F)}\subset \ncrmin{\sigma}$ and by \Cref{lem:canoncial_word_facts}\ref{8.8(2)} we know that $\sylmin{\ForToNC(F)}$ is a Sylvester word for $F$, so $\sylmin{\ForToNC(F)}\in \mathcal{R}(\ncrmin{\sigma}; F)$.  

Now suppose that $\mathcal{R}(\ncrmin{\sigma}; F)\neq \emptyset$. 
Then there exists a $\pi\in\sylv{F}$ that is a subword of $\ncrmin{\sigma}$. By~\Cref{lem:canoncial_word_facts}\ref{8.8(3)} this subword has the additional property that $\pi \subset \sylcont{\ForToNC(F)}$, so we must have $\ncrmin{\ForToNC(F)}\subset \ncrmin{\sigma}$.
Thus $\ForToNC(F) \le \sigma$ by \Cref{prop:gobet_williams_bruhat}.
\end{proof}

The AJS--Billey formula also shows that double Schubert polynomials have the property
\[
\ev_{w}\schub{w}(\xl;\tl)=\prod_{(a, b) \in \operatorname{Inv}(w)}(t_{w(a)}-t_{w(b)}).
\]
A similar formula exists for $\ev_{\sigma_F}(\forestpoly{F})$.  For $\sigma \in \NC_{n}$, define the \emph{noncrossing inversion set} of $\sigma$ to be 
\[
\invnc{\sigma}\coloneqq  \{(i,j) \in [n]^{2} \;|\; \text{$i<j$, $\;\sigma(i\,j) \in \NC_{n}\;$ and $\;\sigma(i)>\sigma(j)$}\ \}.
\]

\begin{eg}
    For  $\sigma=82154763\in \NC_{8}$ in \Cref{eg:noncrossing} we have 
\[
\invnc{\sigma} = \{(1,3), (1, 8), (2, 3), (4, 5), (4, 8), (6, 7), (6, 8)\}.
\]
\end{eg}

\begin{prop}
\label{prop:invNC}
For $\sigma \in \NC_{n}$, $(i,j) \in \invnc{\sigma}$ if and only if $i$ is the smallest element in its cycle and either
\begin{enumerate}
\item $j$ belongs to the same cycle as $i$, or

\item $\sigma(j) < i<j$ and $j$ is minimal with respect to this property.  
\end{enumerate}
\end{prop}
\begin{proof}
Let $A$ and $B$ denote the parts of $\sigma$ containing $i$ and $j$, respectively.  
If $A = B$, then $\sigma(i)>\sigma(j)$ occurs only when $i=\min A$. 
In this case, clearly $\sigma(i,j)\in \NC_{n}$ as it splits the cycle $A$ into the product $c_{\{1,\ldots,\sigma(j)\}\cap A}\,c_{\{\sigma(j)+1,\ldots,n\}\cap A}$ of two disjoint cycles. 

If, on the other hand, $A \neq B$ then $\sigma(i)>\sigma(j)$ if and only if $A$ is nested in $B$, and $\sigma(j)<i$. 
In this case $c_Ac_B(i,j)=c_{A\sqcup B}$ exactly when $i=\min A$, and $\sigma (i,j)$ is noncrossing exactly when there is no cycle $C$ of $\sigma$ nested between $B$ and $A$, in that $A$ is nested in $C$ and $C$ is nested in $B$.  
The latter is equivalent to $j$ being minimal with respect to the property that $\sigma(j)<i<j$.
\end{proof}

\begin{thm}
\label{thm:forestatownperm}
For $F\in \indexedforests_n$ and $\sigma=\ForToNC(F)$ we have  
\[
\ev_{\ForToNC(F)}(\forestpoly{F}) 
= \prod_{(a,b) \in \invnc{\sigma}} (t_{\sigma(a)} - t_{\sigma(b)}).
\]
\end{thm}
\begin{proof}
We apply \Cref{thm:AJS_Billey_for_Forests} with $\omega=\ncrmin{F}$, which lies in $\rednc{\sigma}$ by \Cref{prop:gobet_williams_bruhat}.

By~\Cref{lem:canoncial_word_facts}\ref{8.8(3)}, apart from $\sylmin{\ForToNC(F)}$, every Sylvester subword of $F$ in $\longword$ contains a letter that is not in $\ncrmin{\ForToNC(F)}$, so $\mathcal{R}(\omega; F)$ contains a unique element which gives the same Sylvester word as $\sylmin{\ForToNC(F)}$.  
By \Cref{thm:AJS_Billey_for_Forests} we know that this is of the form $\prod_{i=1}^{|F|} (t_{c_i}-t_{d_i})$
for some sequences $c_1,\ldots,c_{|F|}$ and $d_1,\ldots,d_{|F|}$ with $c_i>d_i$ for all $i$. 

Now, for any polynomial $f$, permutation $w$ and $i\ne j$ we have the divisibility relations $(t_{w(i)}-t_{w(j)}) \big| (\ev_{w} f-\ev_{w(i,j)}f)$. 
So taking $w=\sigma$ and noting that $\ev_{\sigma(i,j)}\forestpoly{F}=0$ for $(i,j)\in \invnc{\sigma}$ by \Cref{thm:AJSB_triangularity} we deduce that $\prod_{(a,b)\in \invnc{\sigma}}(t_{\sigma(a)}-t_{\sigma(b)})$ divides $\ev_{\sigma}(\forestpoly{F})$. 
\end{proof}

\begin{rem}
The set $\invnc{\sigma}$ has a combinatorial description in terms of indexed forests.  
For $F\in \suppfor{n}$ and $v\in \internal{F}$, we define the spread of $v$ as the pair $(i,j)$ such that $i$ is the leftmost leaf descendant of $v$ and $j$ is the rightmost leaf descendant of $v$. 
Define the \emph{spread set} of $F$ as
\[
    \operatorname{Sp}(F)=\{(i,j)\suchthat i<j\text{ spread of some }v\in \internal{F}\}.
\]
Then one can directly verify that $\operatorname{Sp}(F)=\invnc{\ForToNC(F)}$, so an alternative statement of~\Cref{thm:forestatownperm} is 
\[
    \ev_{\ForToNC(F)}(\forestpoly{F}) =\prod_{(a,b)\in \operatorname{Sp}(F)}(t_{\sigma(a)}-t_{\sigma(b)}).
\]
\end{rem}

\section{Monoids}
\label{sec:monoids}
In this section we recall combinatorial monoids which structure the composition of equivariant quasisymmetric divided difference operators and equivariant Bergeron--Sottile maps.  
In previous work~\cite{NST_a, NST_c}, it was established that in the non-equivariant setting the analogous compositions are intimately related to rewriting rules in an augmented version of the positive Thompson monoid.  
Elements of this monoid are parameterized by the indexed forest and nested forests described in Section~\ref{sec:preliminaries}.  
In order to generalize the results of~\cite{NST_a, NST_c} to the equivariant setting, and in particular to account for the variable depletion in Theorem~\ref{thm:ForestDesiderata}, we undertake an even finer study of this monoid and quotients thereof here.

Throughout we denote by $\reseq$ be the set of strings of letters from the alphabet $$\bigcup_{i=1}^{\infty} \{\rletter{i}^-,\rletter{i}^+,\eletter{i}\}.$$
This set has the structure of a free monoid under concatenation.
For example, the product of $\rletter{2}^-\rletter{3}^+\eletter{2} \in \reseq$ and $\rletter{1}^+\rletter{1}^- \in \reseq$ is $\rletter{2}^-\rletter{3}^+\eletter{2}\rletter{1}^+\rletter{1}^-\in \reseq$. 

\subsection{The marked nested forest monoid}
\label{subsec:mnfor}

We quickly recall some notions from \cite{NST_c}. 
A \emph{marked nested forest} \cite[Definition 3.9]{NST_c} is a nested forest $\wh{F}$ in $\nfor$ alongside the extra data marking some finite subset of the roots nodes in the constituent trees of $\wh{F}$ by $\otimes$ such that the root nodes of indexed trees that are nested are \textit{necessarily} marked. 
Write $\mnfor$ for the set of marked nested forests. 
This has a natural monoid structure $\wh{F}\cdot \wh{G}$ given by joining the $i$th leaf of $\wh{F}$ to the $i$th unmarked root of $\wh{G}$ for all $i\in \NN$ and preserving all markings; see \cite[Definition 3.11]{NST_c}.
Just as $\indexedforests$ has a presentation by the Thompson monoid $\Th$,
in \cite{NST_c} it was shown that $\mnfor$ has a similar presentation by the \emph{augmented Thompson monoid} introduced in
\cite[Definition 3.7]{NST_c}.
\begin{thm}[{\cite[Theorem 3.12]{NST_c}}]
\label{thm:mnforpresentation}
    $\mnfor$ has a presentation given by the quotient of the free monoid $\langle \rletter{1}^-,\rletter{2}^-,\ldots,\eletter{1},\eletter{2},\ldots\rangle\subset \reseq$ by the formal commutation relations
    \begin{gather*}\eletter{i} \eletter{j}=\eletter{j} \eletter{i+1}\text{ if }i>j, \qquad \qquad 
    \eletter{i} \rletter{j}^-=\rletter{j}^- \eletter{i+1}\text{ if }i\ge j,\\
    \rletter{i}^- \eletter{j}=\eletter{j} \rletter{i+1}^-\text{ for }i >j, \qquad \qquad 
    \rletter{i}^- \rletter{j}^-=\rletter{j}^-  \rletter{i+1}^-\text{ if }i\ge j.
    \end{gather*}
    under the identification in Figure~\ref{fig:mnestfor}.
\end{thm}

\begin{figure}[!h]
    \centering
    \includegraphics[width=\linewidth]{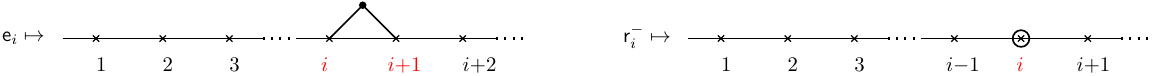}
    \caption{Generators for $\mnfor$}
    \label{fig:mnestfor}
\end{figure}

Note that if we restrict to the monoid generated by $\eletter{1},\eletter{2},\ldots$ then identifying these generators with $1,2,\ldots$ we recover the presentation $\langle i\cdot j=j\cdot (i+1)\suchthat i>j\rangle$ of  $\Th\cong \indexedforests$.

\subsection{The $\star$-monoid}
\label{subsec:partial_shift_monoid}

We define the \emph{$\star$-monoid $\mc{S}$} 
 as the submonoid of $\mnfor$ generated by $\rletter{1}^-,\rletter{2}^-,\ldots$.
Identifying these generators with the integers $1,2,\ldots$, this has a presentation 
$$
\mc{S}\coloneqq\langle i\suchthat i\star  j=j\star  (i+1)\text{ for }i\ge j\rangle.
$$
The relations allow any element  to be rewritten canonically in strictly increasing order, so the elements of $\mathcal{S}$ can be identified with finite subsets of $\NN$.
Under this correspondence if $A,B\subset \NN$ are finite subsets then
\[
A\star B =\{\overline{B}_i\suchthat i\in A\}\cup B.
\]
where $\overline{B}_{i}$ denotes the $i$th element of $\NN \setminus B$. 
For instance, if $A = \{1, 3, 4\} = 1 \star 3 \star 4$ and $B = \{2, 3, 6\} = 2 \star 3 \star 6$, then $A \star B = \{1, 2, 3, 5, 6, 7\}$.

Note that the defining relations of $\Th$ demand $i\cdot j=j\cdot (i+1)$ for $i>j$, so there is a natural map
$\indexedforests\to \mathcal{S}$.
We can describe this map in a straightforward way.
Given $F\in\indexedforests$ let $$L(F)=\{x\suchthat x,x+1\in \supp{F}\},$$
or equivalently as the set of canonical labels of the internal nodes of $F$. 
We note that this is in fact the definition of support in \cite{NT_forest}.

\begin{lem}
    For an indexed forest $F=i_1\cdots i_k$ we have $L(F)=i_1\star\cdots \star i_k$.
\end{lem}
\begin{proof}
    This follows inductively from the fact that $L(G\cdot i)=L(G)\star i$ for any indexed forest $G$.
\end{proof}

\subsection{Marked bicolored nested forests and marked $RE^{\pm}$-forests}

\begin{defn}
A \emph{marked bicolored nested forest} is a marked nested forest where each internal node has been colored either black or white $\{\bnode,\wnode\}$, and denote $\mbnfor$ for the set of all such forests. 
This has a monoid structure inherited from $\mnfor$.
\end{defn}

\begin{figure}[!ht]
    \centering
    \includegraphics[scale=1]{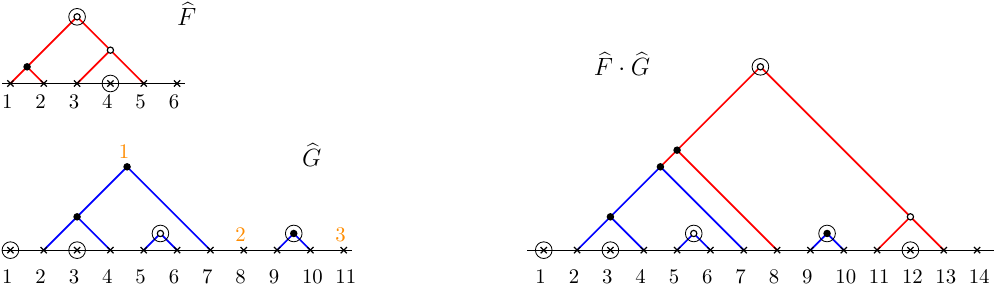}
    \caption{Two elements in $\mbnfor$ (left) and their product (right)}
    \label{fig:monoid_mbnestfor_eg}
\end{figure}

    \begin{prop}
    \label{fact:mbnrelations}
        The monoid structure on $\mbnfor$ can be described as the quotient of the free monoid $\reseq$ by the formal commutation relations
        \begin{gather*}
        \eletter{i} \eletter{j}=\eletter{j} \eletter{i+1}\text{ for }i> j\\
\rletter{i}^- \rletter{j}^-=\rletter{j}^- \rletter{i+1}^-\text{ for }i\ge j, \qquad \qquad 
\rletter{i}^+ \rletter{j}^+=\rletter{j}^+ \rletter{i+1}^+\text{ for }i>j,\\
\rletter{i}^+ \rletter{j}^-=\rletter{j}^- \rletter{i+1}^+\text{ for }i\ge j, \qquad \qquad 
\rletter{i}^- \rletter{j}^+=\rletter{j}^+ \rletter{i+1}^-\text{ for $i>j$},\\
\eletter{i} \rletter{j}^-=\rletter{j}^- \eletter{i+1}\text{ for }i\ge j, \qquad \qquad 
\rletter{i}^- \eletter{j}=\eletter{j} \rletter{i+1}^-\text{ for }i >j\\
\eletter{i} \rletter{j}^+=\rletter{j}^+ \eletter{i+1}\text{ for }i>j, \qquad \qquad 
\rletter{i}^+ \eletter{j}=\eletter{j}\rletter{i+1}^+\text{ for }i>j.
\end{gather*}
under the identification in Figure~\ref{fig:mbnestfor}.

\begin{figure}[!h]
    \centering
    \includegraphics[width=\linewidth]{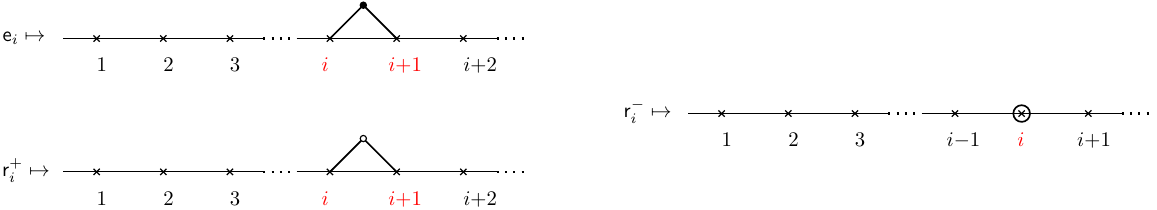}
    \caption{Generators for $\mbnfor$}
    \label{fig:mbnestfor}
\end{figure}
\end{prop}
\begin{proof}
    This follows from the presentation of \Cref{thm:mnforpresentation} as we replaced each $\eletter{i}$ generator with two copies of it, $\eletter{i}$ and $\rletter{i}^+$.
\end{proof}

\begin{defn}
\label{defn:markedRE}
    We define the \emph{marked $RE^{\pm}$-forests} $\mrefor$ to be the quotient of $\mbnfor$ under the further relations
    \begin{align*}\eletter{i} \rletter{i+1}^+=\rletter{i}^+ \eletter{i}\quad \text{ and }\quad\rletter{i}^+ \rletter{i+1}^+=\rletter{i}^+ \rletter{i}^+.\end{align*}
\end{defn}

The two relations in Definition~\ref{defn:markedRE} are graphically represented in Figure~\ref{fig:right_rotation}.
Note that these correspond to  \emph{rotations} on binary trees, which is how one obtains the \emph{Tamari lattice} \cite{Tam62}.
In our case we only need to perform these rotation moves when a white node is a right child. 
\begin{figure}[!ht]
    \centering
    \includegraphics[scale=1]{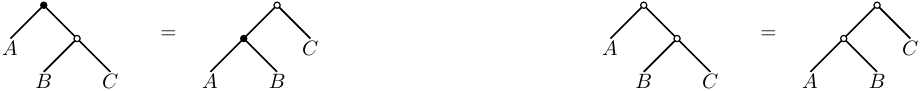}
    \caption{Additional relations in $\mrefor$}
    \label{fig:right_rotation}
\end{figure}

\begin{figure}[ht!]
    \centering
    \includegraphics[scale=0.8]{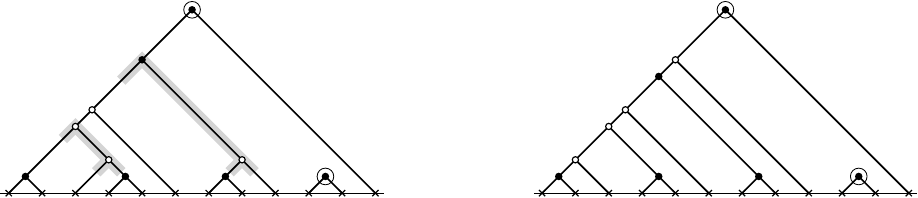}
    \caption{Two equivalent forests in $\mrefor$, with the right forest in normal form}
    \label{fig:refornormalform}
\end{figure}


Note that because $\mrefor$ is graded in $\mathbb{Z}^3$ by the number of $\rletter{}^+,\rletter{}^-,\eletter{}$, we can identify within $\mrefor$ the submonoids generated by only taking letters of a certain type. 
For example, the submonoid generated by $\eletter{}$ is $\Th\cong \indexedforests$, and the submonoid generated by $\eletter{}$ and $\rletter{}^-$ is $\mnfor$.

\begin{rem}
    Say that $F\in \mbnfor$ is in \emph{normal form} if no white internal node is a right child. It is clear that applying the additional relations in $\mrefor$ will eventually lead to such a forest, and in the sequel \cite{BGNST_Toappear} we will show that every element of $\mrefor$ in fact has a unique normal form.\footnote{As pointed out to us by F. Chapoton, this uniqueness can be directly proved using the Gr\"{o}bner bases for operads.}
\end{rem}

\section{Coefficient extraction and $\star$-composition}
\label{sec:coeffextraction}

This section solves the problem of writing $f \in \ZZ[\tl][\xl]$ as a $\ZZ[\tl]$-linear combination of double forest polynomials. 
For double Schubert polynomials we have the identity for any $f\in \poltx$ that
$$f=\sum a_w\schub{w}(\xl;\tl)\text{ \ for \  }a_w(\tl)=(\partial_w f)(\tl;\tl).$$
For double forest polynomials we will analogously define algebraic extraction operators $[\ev\star \eope{F}]:\poltx\to \ZZ[\tl]$
obtained by composing the $\eope{i}$ in a modified way we call ``$\star$-composition''
such that
$$f=\sum a_F(\tl)\forestpoly{F}(\xl;\tl)\text{ \ for \  }a_F(\tl)=[\ev \star \eope{F}]\forestpoly{F}.$$   

\subsection{$\star$-composition}
\label{subsec:star_composition}

Let $A\subset \NN$ be a finite subset and $\overline{A}=\NN\setminus A$ ordered naturally. 
For each $i \ge 1$, let $t_{i,A}\coloneqq t_{(\overline{A})_i}$, and $\wh{\tl}_A=(t_{1,A},t_{2,A},\ldots)$.
Now define 
\begin{align*}
\ev_A f(\xl;\tl)&=f(\wh{\tl}_A;\tl)\\
\rope{i,A}^-f(\xl;\tl)&=f(x_1,\ldots,x_{i-1},t_{i,A},x_i,x_{i+1},\ldots;\tl)\\
\rope{i,A}^+f(\xl;\tl)&=f(x_1,\ldots,x_{i-1},x_{i},t_{i,A},x_{i+1},\ldots;\tl)\\
\eope{i, A}f(\xl; \tl)&= \frac{\rope{i,A}^+f(\xl; \tl) -\rope{i,A}^-f(\xl; \tl) }{x_i-t_{i,A}}.
\end{align*}

\begin{defn}
We define a \emph{$\star$-word} to be a formal string $\Phi=X_{i_1}\star X_{i_2}\star \cdots \star X_{i_k}$ with each $X_i\in \{\rope{i}^+,\rope{i}^-,\eope{i}\}$, and we define the \emph{$\star$-weight} to be $\ul{\Phi}\coloneqq i_1\star \cdots \star i_k\in \mathcal{S}$. 
For $\rt\in \reseq$ we let $\Phi_{\rt}$ be obtained by replacing each letter with its corresponding operator and separating the consecutive letters with $\star$.

For $A\subset \NN$ a finite subset we define operations
\begin{align*}[X_{i_1}\star\cdots\star X_{i_k}]_A\coloneqq& X_{i_1,i_2\star\cdots i_k\star A}X_{i_2,i_3\star\cdots i_k\star A}\cdots X_{i_k,A}:\ZZ[\xl][\tl]\to \ZZ[\xl][\tl]\\
[\ev\star X_{i_1}\star\cdots \star X_{i_k}]_A\coloneqq&\ev_{i_1\star\cdots \star i_k\star A}[X_{i_1}\star\cdots \star X_{i_k}]_A:\ZZ[\xl][\tl]\to \ZZ[\tl].\end{align*}
We omit $A$ from the notation if $A$ is the empty set.
\end{defn}
 Note that from their definition we can recursively compute these operations as \begin{align*}
 [\Phi_1\star \Phi_2]_A&=[\Phi_1]_{\underline{\Phi_2}\star A}[\Phi_2]_A\\
 [\ev\star \Phi]_A&=[\ev]_{\underline{\Phi}\star A}[\Phi]_A.\end{align*}

\begin{eg}
    Let $\rt=\eletter{3}\rletter{2}^+\rletter{1}^-\eletter{2}$.
    Then the $\star$-word $\Phi_{\rt}$ is given by $\eope{3}\star\rope{2}^+\star \rope{1}^-\star \eope{2}$, and the $\star$-weight is given by $3\star 2\star 1 \star 2=\{1,2,4,6\}$.
    The corresponding operations for $A=\emptyset$ are
    \begin{align*}
        [\Phi_{\rt}]&=\eope{3,\{1,2,4\}}\rope{2,\{1,2\}}^+\rope{1,\{2\}}^-\eope{2}\\
        [\ev\star\Phi_{\rt}]&=\ev_{\{1,2,4,6\}}\eope{3,\{1,2,4\}}\rope{2,\{1,2\}}^+\rope{1,\{2\}}^-\eope{2}.
    \end{align*}
\end{eg}

Our next task is to make a connection between $\star$-composition and the monoids from Section~\ref{sec:monoids}.
To this end we introduce the notion of $\star$-compatibility.

\begin{defn}
\label{def:star_compatibility}
Given $\star$-words $B_1,\dots,B_k$ we say that a relation
\[
\sum_{1\leq i\leq k} a_i(\tl)[B_i]=0
\]
is \emph{$\star$-compatible} if we have the following equality of $\star$-weights: 
\[
    \ul{B_1}=\ul{B_2}=\cdots=\ul{B_k}.
\]
\end{defn}
\begin{prop}
\label{cor:GeneralRelations}
Given a $\star$-compatible relation $\sum a_{j}(\tl)[B_j]=0$, a finite set $A\subset\NN$, and $\star$-words $\Phi_1,\Phi_2$, the following hold:
\begin{align*}
\sum a_{j}(\wh{\tl}_{\ul{\Phi_2}\star A})[\Phi_1\star B_j\star \Phi_2]_A &=0,\\
\sum a_{j}(\wh{\tl}_{\ul{\Phi_2}\star A})[\ev\star \Phi_1\star B_j\star \Phi_2]_A &=0.
\end{align*}
\end{prop}
\begin{proof}
Let $\ul{B_1}=\ul{B_2}=\cdots=M\in \mathcal{S}$. The second identity follows from the first by applying $\ev_{\ul{\Phi_1}\star M\star \ul{\Phi_2}\star A}$, so we will show the first identity. 
It is clear that $\sum a_j(\wh{\tl}_{\ul{\Phi_2}\star A})[B_j]_{\ul{\Phi_2}\star A}=0$ for all $A$, as this amounts to applying the identity using the equivariant variables $\wh{\tl}_{\ul{\Phi_2}\star A}$, and treating $\{t_i\}_{i\in \ul{\Phi_2}\star A}$ as formal commuting variables. We thus compute
\begin{align*}
    \sum a_{j}(\wh{\tl}_{A\star \ul{\Phi_2}})[\Phi_1\star B_j\star \Phi_2]_A=&\sum a_{j}(\wh{\tl}_{A\star \ul{\Phi_2}})[\Phi_1]_{M\star \ul{\Phi_2}\star A}[B_j]_{\ul{\Phi_1}\star A}[\Phi_2]_A\\=&[\Phi_1]_{M\star \ul{\Phi_2}\star A}\left(\sum a_j(\wh{\tl}_{\ul{\Phi_2}\star A})[B_j]_{\ul{\Phi_2}\star A}\right)[\Phi_2]_A=0.\qedhere
\end{align*}
\end{proof}

\begin{thm}
\label{thm:starcompositionrefor}\label{thm:straighteningrules}
Each relation $a_r b_s=c_t d_u$ from \Cref{fact:mbnrelations} and \Cref{defn:markedRE} (i.e. the defining relations of $\mrefor$) gives a $\star$-compatible relation
$$[A_r\star B_s]=[C_t\star D_u].$$

In particular, if $\rt,\rt'\in \reseq$ induce the same element of $\mrefor$, then
    $[\Phi_\rt]=[\Phi_{\rt'}]$
    is a $\star$-compatible relation. Furthermore if $\rt,\rt'$ give the same element of $\mrefor$ up to the locations of markings of roots then $[\ev\star \Phi_\rt]_A=[\ev\star \Phi_{\rt'}]_A$ for all finite sets $A$.
\end{thm}


\begin{proof}
First we check the $\rope{i}^{\pm}\star \rope{j}^{\pm}$ identities. For $\epsilon_1,\epsilon_2\in \{+,-\}$, let $a=j+\delta_{\epsilon_1,+}$ and let $b=i+\delta_{\epsilon_2,+}$, then if $i\ge j$ and we do not have $(i,j,\epsilon_1,\epsilon_2)=(i,i,-,+)$ then
    \begin{align*}[\rope{i}^{\epsilon_1}\star \rope{j}^{\epsilon_2}]f=\rope{i,\{j\}}^{\epsilon_1}\rope{j}^{\epsilon_2}f=&f(x_1,\ldots,x_{a-1},t_j,x_a,\ldots,x_{b-1},t_{i+1},x_b,\ldots;\tl)\\=&\rope{j,\{i+1\}}^{\epsilon_2}\rope{i+1}^{\epsilon_1}f=[\rope{j}^{\epsilon_2}\star \rope{i+1}^{\epsilon_1}]f.\end{align*}
    We now check the $\eope{i}\star \eope{j}$ identity. The two sides of the equation expand out to the following.
\begin{align*}
    [\eope{i}\star\eope{j}]&=\eope{i,\{j\}}\eope{j}
    =\eope{i,\{j\}}\frac{\rope{j}^+-\rope{j}^-}{x_{j}-t_{j}}
    =\frac{(\rope{i,\{j\}}^+-\rope{i,\{j\}}^-)(\rope{j}^+-\rope{j}^-)}{(x_i-t_{i,\{j\}})(x_j-t_{j})}\\
    [\eope{j}\star\eope{i+1}]&=\eope{j,\{i+1\}}\eope{i+1}f=\eope{j,\{i+1\}}\frac{\rope{i+1}^+-\rope{i+1}^-}{x_{i+1}-t_{i+1}}
    =\frac{(\rope{j,\{i+1\}}^+-\rope{j,\{i+1\}}^-)(\rope{i+1}^+-\rope{i+1}^-)}{(x_j-t_{j,\{i+1\}})(x_i-t_{i+1})}. 
\end{align*}
We have $\overline{\{j\}}_i=i+1$ and $\overline{\{i+1\}}_j=j$ so the denominators are equal, and the numerators are equal because $\rope{i,\{j\}}^{\epsilon_1}\rope{j}^{\epsilon_2}=[\rope{i}^{\epsilon_1}\star \rope{j}^{\epsilon_2}]=[\rope{j}^{\epsilon_2}\star \rope{i+1}^{\epsilon_1}]=\rope{j,\{i+1\}}^{\epsilon_2}\rope{i}^{\epsilon_1}$.
    
The other identities are similar and we omit their verification. That $[\Phi_\rt]=[\Phi_{\rt'}]$ is a $\star$-compatible relation now follows from  \Cref{cor:GeneralRelations}.
    
    Finally, if $\rt$ produces an element of $\mrefor$ then by definition $\rletter{1}^-\rt$ turns the leftmost unmarked root into a marked root. Therefore if $\rt,\rt'$ give the same element of $\mrefor$ up to the locations of markings then there exist $a,b$ such that $(\rletter{1}^{-})^a\rt$ and $(\rletter{1}^{-})^b\rt'$ give the same element of $\mrefor$, and so $[\ev\star \Phi_\rt]_A=[\ev\star (\rope{1}^-)^{\star a}\star \Phi_\rt]_A=[\ev\star(\rope{1}^-) ^{\star b}\star \Phi_{\rt'}]_A=[\ev\star \Phi_{\rt'}]_A.$
\end{proof}

\begin{rem}
    If we define for $\Phi_i\in \{\rope{i}^+,\rope{i}^-,\eope{i}\}$ the endomorphism 
    \[
    \wt{\Phi_i}\in\End\left(\bigoplus_{\text{ finite }A\subset \NN}\poltx\right)
    \]
    by linearly extending $\wt{\Phi_i}(f\cdot 1_A)\coloneqq(\Phi_{i,A}f)\cdot 1_{i\star A}$, then
    \[
    \wt{\Phi_{i_1}}\cdots \wt{\Phi_{i_k}}(f\cdot 1_A)=([\Phi_{i_1}\star\cdots \star \Phi_{i_k}]_Af)\cdot 1_{i_1\star\cdots \star i_k\star A}.
    \]
In particular, $\star$-compatible relations descend to genuine relations among the $\wt{\Phi_i}$, and they therefore give a representation of $\mrefor$. In particular, the $\wt{\eope{i}}$ operations are a genuine representation of the Thompson monoid with $\wteope{i}\wteope{j}=\wteope{j}\wteope{i+1}$ for $i>j$.
\end{rem}

\subsection{Coefficient extraction via $\star$-compositions}
\label{subsec:coefficient_extraction}
By \Cref{thm:straighteningrules}, if $F\in \indexedforests$ and we have two factorizations $F=i_1\cdots i_k=j_1\cdots j_k$, then we have a $\star$-compatible identity
$$
[\eope{i_1}\star\cdots \star \eope{i_k}]=[\eope{j_1}\star \cdots \star \eope{j_k}],
$$
as both sides correspond to the same element of $\Th\subset \mrefor$.
Hence we are free to replace $\eope{i_1}\star\cdots \star \eope{i_k}$ with $\eope{j_1}\star\cdots \star \eope{j_k}$ in any $\star$-composition by \Cref{cor:GeneralRelations}.

\begin{defn}
    We extend the definition of $[\Phi]_A$ for a $\star$-word $\Phi$ to also allow $\Phi$ to contain letters $\eope{F}$ for $F\in \indexedforests$, which should be interpreted as $\eope{i_1}\star\cdots \star \eope{i_k}$ for any factorization $i_1\star\cdots\star i_k$. 
    In particular,
    \begin{align*}[\eope{F}]_A\coloneqq &\eope{i_1,i_2\star\cdots i_k\star A}\eope{i_2,i_3\star\cdots \star i_k\star A}\cdots \eope{i_k,A}\\
    [\ev\star \eope{F}]_A\coloneqq &\ev_{L(F)\star A}\eope{i_1,i_2\star\cdots i_k\star A}\eope{i_2,i_3\star\cdots \star i_k\star A}\cdots \eope{i_k,A}.
    \end{align*}
\end{defn}

\begin{eg}
Take $A = \{2, 3\}$ and let $F$ be the forest with factorization $1\cdot 1 \cdot 3$.  Then
\[
[\eope{F}]_{\{2,3\}} = \eope{1, \{1, 2, 3, 4\}} \eope{1, \{2, 3, 4\}} \eope{3, \{2, 3\}}
\]
\end{eg}

The monoid $\indexedforests$ is right-cancellative \cite[Proposition 2.6]{DeTe19}, meaning that the equation $G=H\cdot F$ has at most one solution $H$ for fixed $F,G$. We shall denote this solution by $G/F$, and write $F\le_R G$ if $G/F$ exists.
 
\begin{prop}
\label{prop:FGtrimming}
For $F,G\in\indexedforests$, we have 
\[
[\eope{F}](\forestpoly{G}(\xl;\tl))=\left\lbrace 
\begin{array}{ll} \forestpoly{G/F}(\xl;\wh{\tl}_{L(F)}) & F\le_R G\\ 0 & \text{otherwise.}
\end{array}\right.
\]
\end{prop}
\begin{proof}
    This follows by induction on $|F|$. 
    If we write $F=(F/i)\cdot i$ then we obtain
    \begin{align*}
    [\eope{F}]\,\forestpoly{G}
    =
    [\eope{(F/i)\cdot i}]\,\forestpoly{G}
    =
    [\eope{F/i}]_{\{i\}}\,\delta_{i\le_R G}\,\forestpoly{G/i}(\xl;\wh{\tl}_{i})
    =
    \delta_{i\le_R G}\,[\eope{F/i}]_{\{i\}}\,\forestpoly{G/i}(\xl;\wh{\tl}_{i})
    \end{align*}
    By applying the inductive hypothesis to the right-hand side we obtain
    \begin{align*}
    \delta_{i\le_R G}\,\delta_{F/i\le_R G/i}\,\forestpoly{((G/i)/(F/i))}(\xl;\wh{\tl}_{ L(F/i)\star i})=&\delta_{F\le_R G}\,\forestpoly{G/F}(\xl;\wh{\tl}_{L(F)}),
    \end{align*}
    thereby concluding the proof.
\end{proof}
In \Cref{fig:starcomp} we see two different ways of computing $[\eope{1\cdot 3}]=[\eope{2\cdot 1}]$ applied to $\forestpoly{F}$ for $F=1\cdot 1\cdot3=1\cdot 2\cdot 1$.
\begin{figure}[!h]
    \centering
    \includegraphics[width=\linewidth]{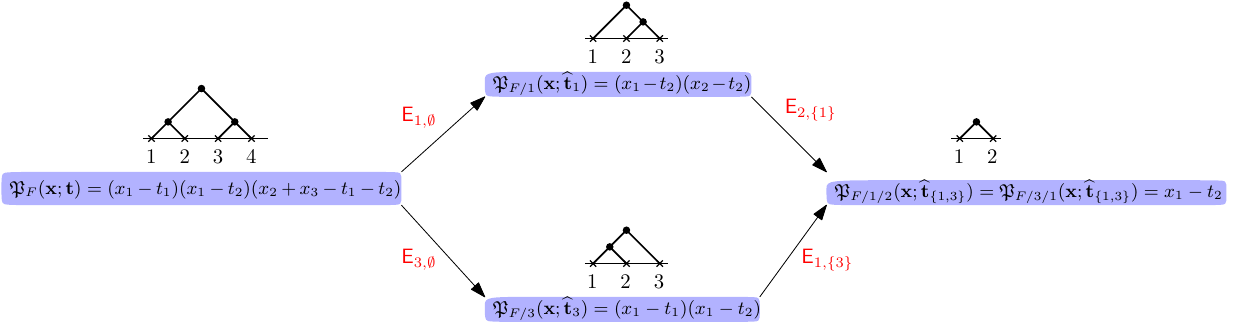}
    \caption{Applications of $\eope{i,A}$ operations to forest polynomials}
    \label{fig:starcomp}
\end{figure}
\begin{thm}
\label{thm:coeffextract}
    For
        $f\in \poltx$ we have
    $$f(\xl;\tl)=\sum_{F\in \indexedforests}a_F\forestpoly{F}(\xl;\tl) \text{ \ where \  }a_F=[\ev\star \eope{F}]f.$$
    \end{thm}
    \begin{proof}
        This follows directly from \Cref{prop:FGtrimming} as \begin{equation*}
            [\ev\star  \eope{F}]\,\forestpoly{G}=\ev_{L(F)}[\eope{F}]\,\forestpoly{G}=\delta_{F\le_R G}\ev_{L(F)}\forestpoly{G/F}(\xl;\wh{\tl}_{L(F)})=\delta_{F\le_R G}\,\delta_{G/F,\emptyset}=\delta_{G,F}.\qedhere
        \end{equation*}
    \end{proof}

\begin{eg}
We have the forest polynomial decomposition of $(x_1-t_1)^3$ given by
$$(x_1-t_1)^3=(t_2-t_1)^2\underbrace{(x_1-t_1)}_{\forestpoly{1}}+(t_3+t_2-2t_1)\underbrace{(x_1-t_1)(x_1-t_2)}_{\forestpoly{1\cdot 1}}+\underbrace{(x_1-t_1)(x_1-t_2)(x_1-t_3)}_{\forestpoly{1\cdot 1 \cdot 1}}.$$
Let us see how to extract the coefficient of $\forestpoly{1\cdot 1}$ using these operations. We have $[\ev\star \eope{1\cdot 1}]=\ev_{1\star 1}\eope{1,\{1\}}\eope{1}=\ev_{\{1,2\}}\eope{1,\{1\}}\eope{1}$. We compute
\begin{align*}
    E_{\{1\},1}E_1(x_1-t_1)^3=E_{\{1\},1}(x_1-t_1)^2=\frac{(x_1-t_1)^2-(t_2-t_1)^2}{x_1-t_2}
    =x_1+t_2-2t_1.
\end{align*}
Evaluating this expression at $(x_1,x_2,\ldots)=\wh{\tl}_{\{1,2\}}=(t_3,t_4,\cdots)$ gives us $t_3+t_2-2t_1$.
\end{eg}

\section{Graham positivity via positive straightening rules}
\label{sec:graham}
Recall that the equivariant generalized Littlewood--Richardson (henceforth LR) coefficients are the structure coefficients arising from multiplying double Schubert polynomials
$$\schub{u}(\xl;\tl)\schub{w}(\xl;\tl)=\sum_vc^v_{u,w}(\tl)\schub{v}(\xl;\tl).$$
They specialize to the usual generalized LR coefficients $c^v_{u,w}$ when $\tl=\bm{0}$. 
As was shown by Graham \cite{Gra01} through algebraic-geometric means, the equivariant LR coefficients $c^v_{u,w}(\tl)$ are Graham-positive.

Combinatorially, the essential difficulty in showing Graham-positivity of $c^v_{u,w}(\tl)$ is that the Leibniz rule $\partial_i(fg)=\partial_i(f)g+(s_i\cdot f)\partial_i(g)$ involves $s_i $, which does not preserve Schubert positivity.
Our positive straightening rules allow us to bypass these difficulties for forest polynomials.

Our next result shows that $\star$-compositions of equivariant Bergeron--Sottile operations  induce evaluations at noncrossing partitions.

\begin{lem}
\label{lem:RcomposeNC}
For $\epsilon_1,\ldots,\epsilon_{j_k}\in \{-,+\}$ there exists $n$ and $\sigma\in \NC_n$ such that
    $[\ev\star \rope{j_1}^{\epsilon_1}\star \cdots \star \rope{j_k}^{\epsilon_{j_k}}]f=\ev_\sigma f$.
\end{lem}
\begin{proof}
Suppose
$[\ev\star \rope{j_1}^{\epsilon_1}\star\cdots \star \rope{j_{k-1}}^{\epsilon_{k-1}}]=\ev_{\sigma'}$. 
Let $\sigma'_-$ be the  noncrossing partition obtained by inserting a trivial block in location $j_k$ and $\sigma'_+$ be the noncrossing partition obtained by joining parts $j_k$ and $j_k+1$ in $\sigma'_-$, or equivalently $\sigma'_+=\sigma'_- s_{j_k}$. 
Then
\begin{equation*}
    [\ev \star \rope{1}^{\epsilon_1}\star \cdots \star \rope{j_k}^{\epsilon_{k}}]f=[\ev \star \rope{1}^{\epsilon_1}\star \cdots \star \rope{j_{k-1}}^{\epsilon_{k-1}}]_{\{j_k\}}\rope{j_k}^{\epsilon_k}f=\ev_{\sigma_{\epsilon_k}'}f.\qedhere
\end{equation*}
\end{proof}

Using the relations in \Cref{thm:straighteningrules} we can move all $\eope{i}$ as far to the right past the $\rope{j}^{\epsilon}$ letters as possible until every subword $\eope{i}\star \rope{j}^{\epsilon}$ has either $(j,\epsilon)=(i+1,-)$ or $(i,+)$. 
It turns out there are two additional relations which let us move every $\eope{i}$ to the right of every $\rope{j}^{\epsilon}$.
\begin{prop}
\label{prop:additionalrelations}
    We  have the $\star$-compatible relations
\begin{align*}
[\eope{i}\star \rope{i+1}^-]&=[\rope{i}^+\star \eope{i}]+[\rope{i}^-\star \eope{i+1}]+(t_{i+1}-t_{i})[\eope{i}\star \eope{i+1}]\\
    [\eope{i}\star \rope{i}^+]&=[\rope{i}^+\star \eope{i}]+[\rope{i}^-\star\eope{i+1}]+(t_{i+1}-t_{i})[\eope{i}\star\eope{i}].
\end{align*}
\end{prop}\begin{proof}
    These are proved analogously to the relations in \Cref{thm:straighteningrules} so we omit the proof. 
\end{proof}
\begin{lem}
\label{lem:straightenmonadically}
    $[\Phi_{i_1}\star \cdots \Phi_{i_k}]_A$
    is a Graham-positive combination of expressions of the form
    $[\rope{j_1}^{\epsilon_1}\star \cdots \star\rope{j_\ell}^{\epsilon_\ell}\star \eope{K}]_A,$ and $[\ev\star \Phi_{i_1}\star \cdots \star \Phi_{i_k}]_A$ is a Graham-positive combination of expressions of the form $[\ev\star \rope{j_1}^{\epsilon_1}\star \cdots \star\rope{j_\ell}^{\epsilon_\ell}\star \eope{K}]_A$.
\end{lem}
\begin{proof}
The second one follows from the first since we can write
$$[\ev\star \Phi_{i_1}\star \cdots \star \Phi_{i_k}]_A=[\ev]_{i_1\star \cdots \star i_k\star A}[\Phi_{i_1}\star \cdots \star \Phi_{i_k}]_A,$$
straighten $[\Phi_{i_1}\star \cdots \star \Phi_{i_k}]_A$, and then recombine the $\ev$.

For the first part, we use induction on the number of $\Phi_i$ equal to $\rope{}^{\pm}$. 
When there are none we are already done. 
If there is a sub-expression $\eope{i_j}\star \rope{i_{j+1}}^{\pm}$ then we apply the corresponding straightening rule to move $\eope{}$ to the right of $\rope{}^{\pm}$ from \Cref{thm:starcompositionrefor} unless we are in the situations $\eope{i}\star \rope{i}^+$ or $\eope{i}\star \rope{i+1}^-$ in which case we use one of the additional relations \Cref{prop:additionalrelations} as modified appropriately by \Cref{cor:GeneralRelations}.

Every term in the resulting expression either decreases the number of $\rope{}^{\pm}$ or decreases the sum of the indices of the $\Phi_i$ equal to $\rope{}^{\pm}$. 
So this process must stop eventually.
\end{proof}
\begin{thm}
\label{thm:PartialApplicationGrahamPositive}    $[\Phi_{i_1}\star\cdots \star \Phi_{i_k}]\,\forestpoly{F}$ is a Graham-positive linear combination of forest polynomials $\forestpoly{G}(\xl;\wh{\tl}_{i_1\star\cdots \star i_k})$.
\end{thm}
\begin{proof}
    We want to show that $[\ev\star\eope{G}]_{i_1\star \cdots \star i_k}[\Phi_{i_1}\star\cdots \star \Phi_{i_k}]\forestpoly{F}=[\ev\star \eope{G}\star \Phi_{i_1}\star \cdots \star \Phi_{i_k}]\,\forestpoly{F}$
    is Graham-positive. 
    By \Cref{lem:straightenmonadically} it suffices to show this for $[\ev\star \rope{j_1}^{\epsilon_1}\star\rope{j_2}^{\epsilon_2}\star \cdots \star \rope{j_k}^{\epsilon_k}\star\eope{K}]\,\forestpoly{F}$. 
    We compute
    \begin{align*}
    [\ev\star \rope{j_1}^{\epsilon_1}\star\rope{j_2}^{\epsilon_2}\star \cdots \star \rope{j_k}^{\epsilon_k}\star\eope{K}]\,\forestpoly{F}
    =&[\ev\star \rope{j_1}^{\epsilon_1}\star\rope{j_2}^{\epsilon_2}\star \cdots \star \rope{j_k}^{\epsilon_k}]_{L(K)}\,\eope{K}\forestpoly{F}\end{align*}
    The right-hand side equals $0$ unless $K\le_R F$, in which case it equals
    \begin{align*}\left([\ev \star \rope{j_1}^{\epsilon_1}\star\rope{j_2}^{\epsilon_2}\star \cdots \star \rope{j_k}^{\epsilon_k}]\,\forestpoly{F/K}\right)(\wh{\tl}_{L(K)})
    =&(\ev_{\sigma}\forestpoly{F/K})(\wh{\tl}_{L(K)})\end{align*}
    for some noncrossing partition $\sigma$ by \Cref{lem:RcomposeNC}.
    Graham-positivity now follows from \Cref{thm:AJS_Billey_for_Forests}.
\end{proof}
\begin{prop}
\label{prop:ForestLeibniz}
For $f,g\in \ZZ[\tl][\xl]$  and $A\subset \NN$ we have the ``Leibniz rule''
$$\eope{i,A}(fg)=(\eope{i,A}f)(\rope{i,A}^-g)+(\rope{i,A}^+f)(\eope{i,A}g).$$
\end{prop}
\begin{proof}
This follows from the identity
    \begin{equation*}
     \frac{\rope{i,A}^+(fg)-\rope{i,A}^-(fg)}{x_i-t_{A,i}}=\frac{\rope{i,A}^+f-\rope{i,A}^-f}{x_i-t_{A,i}}\rope{i,A}^-(g)+\rope{i,A}^+f\frac{\rope{i,A}^+g-\rope{i,A}^-f}{x_i-t_{A,i}}.\qedhere
    \end{equation*}
\end{proof}
\begin{thm}
\label{thm:forestmultpos}
    The coefficients $c^H_{F,G}(\tl)$ in the expansion
    $$\forestpoly{F}(\xl;\tl)\forestpoly{G}(\xl;\tl)=\sum_H c^H_{F,G}(\tl)\,\forestpoly{H}(\xl;\tl)$$
    are Graham-positive.
\end{thm}
\begin{proof}
Let $i\in \qdes{H}$. 
Then by \Cref{prop:ForestLeibniz} we have \begin{align*}
c^H_{F,G}(\tl)&=[\ev\star \eope{H}]\,\forestpoly{F}\forestpoly{G}\\
&=[\ev\star \eope{H/i}]_{\{i\}}\left(\delta_{i\in \qdes{F}}\forestpoly{F}(\xl;\wh{\tl}_i)\,\rope{i}^-\forestpoly{G}+\delta_{i\in \qdes{G}}(\rope{i}^+\forestpoly{F})\,\forestpoly{G/i}(\xl;\wh{\tl}_i)\right).
\end{align*}
    As $\rope{i}^{-}\forestpoly{G}$ and $\rope{i}^{+}\forestpoly{F}$ are Graham-positive linear combinations of forest polynomials $\forestpoly{K}(\xl;\wh{\tl}_i)$ by \Cref{thm:PartialApplicationGrahamPositive}, the result follows by induction after applying the observation that
    \begin{equation*}
        [\ev\star \eope{H/i}]_{\{i\}}(f(\xl;\wh{\tl}_i))=([\ev\star \eope{H/i}]f)(\wh{\tl}_i).\qedhere
    \end{equation*}
\end{proof}

\begin{thm}
\label{thm:schubertexpandpos}
The coefficients $c^F_w(\tl)$ in the expansion
$$\schub{w}(\xl;\tl)=\sum_Fc^F_w(\tl)\,\forestpoly{F}(\xl;\tl)$$ are Graham-positive.
\end{thm}
\begin{proof}
We have
\begin{align*}
    c_w^F(\tl)=[\ev\star \eope{F}]\,\schub{w}.
\end{align*}
We prove the stronger statement that any expression $[\ev\star\Phi_{i_1}\star \cdots \Phi_{i_k}]\,\schub{w}$ is Graham-positive using induction on $\ell(w)$. 
By \Cref{lem:straightenmonadically}, it suffices to show this for $
[
\ev\star \rope{j_1}^{\epsilon_1}\star \cdots \star\rope{j_\ell}^{\epsilon_\ell}\star \eope{K}]_A$.
    
    If $K=\emptyset$, then this is equal to $\ev_{\sigma}\schub{w}$ for some $\sigma\in \NC_n$ by Lemma~\ref{lem:RcomposeNC}.
    The Graham-positivity now follows from the AJS--Billey formula {\cite{AJS,Bil99}}  recalled in \Cref{th:AJSB}. 
    
    So suppose $K\ne \emptyset$, and consider $i\in \qdes{K}$. 
    By writing $\eope{i}$ as $\rope{i}^{-}\partial_i$ we have
    \begin{align*}
    [\ev\star \rope{j_1}^{\epsilon_1}\star \cdots \star\rope{j_\ell}^{\epsilon_\ell}\star \eope{K}]\schub{w}=&[\ev\star \rope{j_1}^{\epsilon_1}\star \cdots \star\rope{j_\ell}^{\epsilon_\ell}\star \eope{K/i}]_{\{i\}} \rope{i}^-\partial_i\schub{w}.
    \end{align*}
    The right-hand side equals $0$ if $i\notin\des{w}$. 
    If $i\in \des{w}$, then it equals
    \begin{align*}
   [\ev\star \rope{j_1}^{\epsilon_1}\star \cdots \star\rope{j_\ell}^{\epsilon_\ell}\star \eope{K/i}]_{\{i\}}\, \rope{i}^-\schub{ws_i}
    =&[\ev\star \rope{j_1}^{\epsilon_1}\star \cdots \star\rope{j_\ell}^{\epsilon_\ell}\star \eope{K/i}\star \rope{i}^-]\,\schub{ws_i}.
    \end{align*}
    Since $\ell(ws_i)<\ell(w)$ when $i\in \des{w}$ the result follows by induction.
\end{proof}

As special cases of the preceding two results we obtain the following results.
\begin{cor}The following hold.
    \begin{enumerate}
    \item 
    The coefficients $a^e_{c,d}(\tl)$  in
    $$\slide{c}(x_1,\ldots,x_n;\tl)\slide{d}(x_1,\ldots,x_n;\tl)=\sum a^e_{c,d}(\tl)\slide{e}(x_1,\ldots,x_n;\tl)$$
    are Graham-positive.
    \item 
    The coefficients $b_{\lambda}^c(\tl)$ in $$s_\lambda(x_1,\ldots,x_n;\tl)=\sum b_{\lambda}^c(\tl)\slide{c}(x_1,\ldots,x_n;\tl)$$
    are Graham-positive.
    \end{enumerate}
\end{cor}
\begin{proof}
    These immediately follow from  \Cref{thm:forestmultpos} and \Cref{thm:schubertexpandpos} respectively -- we note that the left-hand sides of the two equations lie in $\eqsym{n}$ and hence the only double forest polynomials that can appear in their expansion are the double fundamental quasisymmetric polynomials by \Cref{thm:qsymbasis}.
\end{proof}

\pagebreak

\begin{longtable}{|c|l|} 
\hline
$F$ & \text{Expansion of $\forestpoly{F}$ with $y_{ij}\coloneqq x_i-t_j$ and $z_{ij}\coloneqq t_j-t_i$ }  \\
\hline
$\emptyset $ & $ 1 $\\
$1$ & $y_{11}$\\
$2$ & $y_{12}+z_{12}+y_{22}$\\
$3$ & $y_{13}+z_{13}+y_{23}+z_{23}+y_{33}$\\
$1\cdot 1$  & $y_{11}y_{12}$\\
$1\cdot 2$ & $y_{11}z_{12}+y_{11}y_{22}$\\
$1\cdot 3$ & $z_{13}y_{11}+y_{11}y_{23}+z_{23}y_{11}+y_{11}y_{33}+y_{11}y_{13}$\\
$2\cdot 2$ & $y_{12}y_{13}+z_{12}y_{13}+y_{13}y_{22}+z_{13}y_{22}+y_{22}y_{23}$\\
$2\cdot 3$ & $z_{13}y_{12}+y_{12}y_{23}+z_{23}y_{12}+y_{12}y_{33}+z_{12}z_{13}+z_{12}y_{23}+z_{12}z_{23}+z_{12}y_{33}+z_{23}y_{22}+$\\
{} & $y_{22}y_{23}$\\
$1\cdot 1\cdot 1$ & $y_{11}y_{12}y_{13}$\\
$1\cdot 1\cdot 2$ & $z_{12}y_{11}y_{13}+y_{11}y_{13}y_{22}$\\
$1\cdot 1\cdot 3$ & $z_{13}y_{12}y_{11}+y_{11}y_{12}y_{23}+z_{23}y_{11}y_{12}+y_{11}y_{12}y_{33}$\\
$1\cdot 2\cdot 2$ & $z_{13}y_{11}y_{22}+y_{11}y_{22}y_{23}$\\
$1\cdot 2\cdot 3$ & $z_{12}z_{13}y_{11}+z_{12}y_{11}y_{23}+z_{12}z_{23}y_{11}+z_{12}y_{11}y_{33}+z_{23}y_{11}y_{22}+y_{11}y_{22}y_{33}$\\
$2\cdot 2\cdot 2$ & $y_{12}y_{13}y_{14}+z_{12}y_{13}y_{14}+y_{13}y_{14}y_{22}+z_{13}y_{14}y_{22}+y_{14}y_{22}y_{23}+z_{14}y_{22}y_{23}+y_{22}y_{23}y_{24}$\\
$2\cdot 2\cdot 3$ & $z_{13}y_{12}y_{14}+y_{12}y_{14}y_{23}+z_{23}y_{12}y_{14}+y_{12}y_{14}y_{33}+z_{12}z_{13}y_{14}+z_{12}y_{14}y_{23}+z_{12}z_{23}y_{14}+$\\
{} & $z_{12}y_{14}y_{33}+z_{23}y_{14}y_{22}+y_{14}y_{22}y_{33}+z_{23}z_{14}y_{22}+z_{14}y_{22}y_{33}+z_{23}y_{22}y_{24}+y_{22}y_{24}y_{33}$\\
$2\cdot 2\cdot 4$ & $z_{14}   y_{12}   y_{13}+y_{12}   y_{13}   y_{24}+z_{24}   y_{13}   y_{12}+ y_{12}   y_{13}   y_{34}+z_{34}   y_{12}   y_{13}+y_{12}   y_{13}   y_{44}+z_{12} z_{14}y_{13}+$\\
{} & $
z_{12}   y_{13}   y_{24}+
z_{12}   z_{24}   y_{13}+z_{12}   y_{13}   y_{34}+z_{12}   z_{34}   y_{13}+z_{12}   y_{13}   y_{44}+z_{24}   y_{13}   y_{22}+y_{13}   y_{22}   y_{34}+ 
$\\
{} & $
z_{34}   y_{13}   y_{22}+y_{13}   y_{22}   y_{44}
+
z_{13}   z_{24}   y_{22}+z_{13}   y_{22}   y_{34}+z_{13}   z_{34}   y_{22}+z_{13}   y_{22}   y_{44}+z_{24}   y_{23}   y_{22}+$\\
{} & $y_{22}   y_{23}   y_{34}+z_{34}   y_{22}   y_{23}+y_{22}   y_{23}   y_{44}+
z_{14}y_{13}y_{22}+y_{13}y_{22}y_{24}+z_{13}z_{14}y_{22}+z_{13}y_{22}y_{24}$ \\
$2\cdot 3\cdot 4$ & $z_{13}z_{14}y_{12}+ z_{13}y_{12}y_{24}+z_{13}z_{24}y_{12}
   + z_{13}y_{12}y_{34}+z_{13}z_{34}y_{12}
   +   z_{13}     y_{12}     y_{44}
     + z_{24}     y_{12}     y_{23}+$\\
     {} & $y_{12}y_{23} y_{34}
   +z_{34}y_{12}y_{23}
 +y_{12}y_{23}y_{44}
 +z_{23} z_{24}y_{12}
   +  z_{23}y_{12}y_{34}
 +z_{23}z_{34}y_{12}
   +   z_{23}y_{12}y_{44}+$\\
 {} & $
    z_{34}y_{12}y_{33}
 +y_{12}y_{33}y_{44}
 +z_{12}z_{13}z_{14}
   +   z_{12}z_{13}y_{24}
    +  z_{12}z_{13}z_{24}
    +  z_{12}z_{13}y_{34}
    +  z_{12}z_{13}z_{34}+$\\
 {} & $
     z_{12}     z_{13}     y_{44}
 +z_{12}     z_{24}     y_{23}
   +   z_{12}     y_{23}     y_{34}
 +z_{12}     z_{34}     y_{23}
   +   z_{12}     y_{23}     y_{44}
    +  z_{12}     z_{23}     z_{24}
 +z_{12}     z_{23}     y_{34}+$\\
 {} & $
    z_{12}     z_{23}     z_{34}
 +z_{12}     z_{23}     y_{44}
 +z_{12}     z_{34}     y_{33}
   +   z_{12}     y_{33}     y_{44}+
 z_{23}     z_{24}     y_{22}
    +  z_{23}     y_{22}     y_{34}
 +z_{23}     z_{34}     y_{22}+$\\
 {} & $
     z_{23}y_{22}y_{44}+z_{34}y_{22}y_{33}
 +y_{22}y_{33}y_{44}
$\\
$1\cdot 1\cdot 1\cdot 1$ & $y_{11}y_{12}y_{13}y_{14}$\\
$1\cdot 1\cdot 1\cdot 2$ & $z_{12}y_{11}y_{13}y_{14}+y_{11}y_{13}y_{14}y_{22}$\\
$1\cdot 1\cdot 2\cdot 2$ & $z_{13}y_{11}y_{14}y_{22}+y_{11}y_{14}y_{22}y_{23}$\\
$1\cdot 1\cdot 3\cdot 3$ & $z_{14}  y_{12}  y_{11}  y_{23}+z_{23}  z_{14}  y_{11}  y_{12}+z_{14}  y_{12}  y_{11}  y_{33}+y_{11}  y_{12}  y_{23}  y_{24}+z_{23}  y_{11}  y_{12}  y_{24}+$\\
{} & $y_{11}  y_{12}  y_{24}  y_{33}+z_{24}  y_{11}  y_{12}  y_{33}+y_{11}  y_{12}  y_{33}  y_{34}
$\\
\hline
\caption{Expansions of $\forestpoly{F}$ from the vine model}
\label{table:forestpolys}
\end{longtable}

\pagebreak
\begin{longtable}{|c|l|} 
\hline
$w$ & \text{Double forest expansion of $\schub{w}$}  \\
\hline
$\id $ & $ 1 $\\
$21$ & $\forestpoly{1}$\\
$132$ & $\forestpoly{2}$\\
$231$ & $\forestpoly{1\cdot 2}$\\
$312$ & $\forestpoly{1\cdot 1}$\\
$321$ & $\forestpoly{1\cdot 1\cdot 2}+(t_3-t_2)\,\forestpoly{1\cdot 2}$\\
$1243$ & $\forestpoly{3}$\\
$1342$ & $\forestpoly{2\cdot 3}$\\
$1423$ & $\forestpoly{2\cdot 2}$\\
$1432$ & $\forestpoly{2\cdot 2\cdot 3} + (t_4-t_3)\,\forestpoly{2\cdot 3}+\forestpoly{1\cdot 2\cdot 2}$\\
$2143$ & $\forestpoly{1\cdot 3}$\\
$2341$ & $\forestpoly{1\cdot 2\cdot 3}$\\
$2413$ & $\forestpoly{1\cdot 2\cdot 2}+\forestpoly{1\cdot 1\cdot 2}$\\
$2431$ & $\forestpoly{1\cdot 2\cdot 2\cdot 3}+\forestpoly{1\cdot 1 \cdot 2\cdot  3}+(t_4-t_3)\,\forestpoly{1\cdot 2\cdot 3}+(t_2-t_1)\,\forestpoly{1\cdot 2\cdot 2}$\\
$3142$ & $\forestpoly{1\cdot 1\cdot 3}$\\
$3241$ & $\forestpoly{1\cdot 1\cdot 2\cdot 3}+(t_4-t_2)\,\forestpoly{1\cdot 2\cdot 3}$\\
$3412$ & $\forestpoly{1\cdot 1\cdot 2\cdot 2}+(t_4-t_2)\,\forestpoly{1\cdot 2\cdot 2}$\\
$3421$ & $\forestpoly{1\cdot 1\cdot 2\cdot 2\cdot 3}+(t_5-t_2)\,\forestpoly{1\cdot 2\cdot 2\cdot  3}+(t_2-t_1)\forestpoly{1\cdot 1\cdot  2\cdot 2}+(t_2-t_1)(t_4-t_2)\,\forestpoly{1\cdot 2\cdot 2}$\\
$4123$ & $\forestpoly{1\cdot 1\cdot 1}$\\
$4132$ & $\forestpoly{1\cdot 1\cdot 1\cdot 3}+(t_4-t_3)\,\forestpoly{1\cdot 1\cdot 3}$\\
$4213$ & $\forestpoly{1\cdot 1\cdot 1\cdot 2}+(t_4-t_2)\,\forestpoly{1\cdot 1\cdot 2}$\\
$4231$ & $\forestpoly{1\cdot 1\cdot 1\cdot 2\cdot 3}+(t_5+t_4-t_3-t_2)\,\forestpoly{1\cdot 1\cdot 2\cdot 3}+(t_4-t_3)(t_4-t_2)\,\forestpoly{1\cdot 2\cdot 3}$\\
$4312$ & $\forestpoly{1\cdot 1\cdot 1\cdot 2\cdot 2}+(t_5+t_4-t_3-t_2)\,\forestpoly{1\cdot 1\cdot 2\cdot 2}+(t_4-t_3)(t_4-t_2)\,\forestpoly{1\cdot 2\cdot 2}$\\
$4321$ & $\forestpoly{1\cdot 1\cdot 1\cdot 2\cdot 2\cdot 3}+(t_6+t_5-t_3-t_2)\,\forestpoly{1\cdot 1\cdot 2\cdot 2\cdot 3}+(t_2-t_1)\,\forestpoly{1\cdot 1\cdot 1\cdot 2\cdot 2}+(t_5-t_2)(t_5-t_3)\,\forestpoly{1\cdot 2\cdot 2\cdot 3}+$\\
${}$ & $(t_2-t_1)(t_5+t_4-t_3-t_2)\,\forestpoly{1\cdot 1\cdot 2\cdot 2}+(t_2-t_1)(t_4-t_2)(t_4-t_3)\forestpoly{1\cdot 2\cdot 2}$\\
$12354$ & $\forestpoly{4}$\\
$12453$ & $\forestpoly{3\cdot 4}$\\
$12534$ & $\forestpoly{3\cdot 3}$\\
$12543$ & $\forestpoly{3\cdot 3\cdot 4}+\forestpoly{2\cdot 3\cdot 3}+(t_5-t_4)\forestpoly{3\cdot 4}$\\
$13452$ & $\forestpoly{2\cdot 3\cdot 4}$\\
$13524$ & $\forestpoly{2\cdot 3\cdot 3}+\forestpoly{2\cdot 2\cdot 3}$\\
$14523$ & $\forestpoly{2\cdot 2\cdot 3\cdot 3}+\forestpoly{1\cdot 2\cdot 2\cdot 3}+(t_5-t_3)\forestpoly{2\cdot 3\cdot 3}$\\
$34512$ & $\forestpoly{1\cdot 1\cdot 2\cdot 2\cdot 3\cdot 3}+(t_6-t_2)\forestpoly{1\cdot 2\cdot 2\cdot 3\cdot 3}+(t_3-t_1)\forestpoly{1\cdot 1\cdot 2\cdot 2\cdot 3}+(t_3-t_1)(t_5-t_2)\forestpoly{1\cdot 2\cdot 2\cdot 3}$\\
\hline
\caption{Some double Schubert to double  forest expansions}
\label{table:Schubtofor}
\end{longtable}

\pagebreak
\begin{longtable}{|c|c|l|} 
\hline
$c$ & $d$ & \text{``Equivariant summand'' in $\slide{c}(\xl;\tl)\cdot \slide{d}(\xl;\tl)$}  \\
\hline

$0001$ & $0001$  & $ (t_5-t_4)\slide{0001} $\\

$0001$ & $0 0 1 1$ & $(t_5-t_3)\slide{0011}$\\
$0001$ & $0 0 0 2$ & $(t_6-t_4)\slide{0002}$\\

$0001$ & $0 0 0 3$ & $(t_7-t_4)\slide{0003}$\\
$0001$ & $0 0 1 2$ & $(t_6 + t_5 - t_4 - t_3
)\slide{0012}$\\
$0001$ & $0 0 2 1$ & $(t_6  - t_3
)\slide{0021}+(t_6-t_5)\slide{0012}$\\
$0001$ & $0111$ & $(t_5-t_2)\slide{0111}$\\

$0011$ & $0011$ & $(t_6-t_3)\slide{0012}+(t_4-t_3)\slide{0021}+(t_5 + t_4 - t_3 - t_2)\slide{0111}+(t_4 - t_3) (t_5 - t_3)\slide{0011}$\\
$0011$ & $0002$ & $(t_6-t_3)(\slide{0012}+\slide{0021})$\\

$0011$ & $0003$ & $(t_7-t_3)(\slide{0013}+\slide{0031}+\slide{0022})$\\
$0011$ & $0012$ & $(t_7-t_3)\slide{0013}+(t_5-t_3)\slide{0022}+(t_6+t_5-t_3-t_2)(\slide{0112}+\slide{0121})+$\\
${}$ & ${}$ & $(t_5-t_3)(t_6-t_3)\slide{0012}$\\
$0011$ & $0021$ & $(t_7+t_6 - t_5-t_3)\slide{0022}+(t_6-t_5)\slide{0112}+(t_4-t_3)\slide{0031}+$\\
${}$ & ${}$ & $(2t_6 +t_4 -t_5  - t_3 - t_2)\slide{0121}+(t_6 + t_4 - t_3 - t_2)\slide{0211}+(t_6-t_3)(t_6-t_5)\slide{0012}+
$\\
${}$ & ${}$ & $(t_4-t_3)(t_6-t_3)\slide{0021}$\\

$0011$ & $0111$ & $(t_6-t_2)\slide{0112}+(t_6+t_4-2t_2)\slide{0121}+(t_4-t_2)\slide{0211}+ (t_5 + t_4 - t_2 - t_1)\slide{1111} +$\\
${}$ & ${}$ & $(t_4-t_2)(t_5-t_2)\slide{0111}$\\

$0002$ & $0002$ & $(t_7 + t_6 - t_5 - t_4)\slide{0003}+(2t_6-t_5-t_4)\slide{0012}+(t_6-t_5)\slide{0021}+$\\
{} & {} & $(t_6-t_4)(t_6-t_5)\slide{0002}$\\
$0002$ & $0003$ & $(t_8 + t_7 - t_5 - t_4
)\slide{0004}+(2t_7-t_5-t_4)(\slide{0013}+\slide{0022})+(t_7-t_5)\slide{0031}+$\\
${}$ & ${}$ &
$(t_7-t_4)(t_7-t_5)\slide{0003}$\\

$0002$ & $0012$ & $(t_7 + t_6 - t_4 - t_3
)(\slide{0013}+\slide{0022})+(t_6-t_4)\slide{0112}+(t_6-t_2)\slide{0121}+$\\{} & {} & $(t_6-t_3)(t_6 - t_4)\slide{0012}
$
\\
$0002$ & $0021$ & $(t_7-t_5)\slide{0013}+(t_7 +t_6 - t_5 - t_3
)(\slide{0031}+\slide{0022})+(t_6-t_5)(\slide{0112}+\slide{0211})+$\\
${}$ & ${}$ & $(2t_6-t_5-t_3)\slide{0121}+(t_6-t_5)(t_6-t_3)(\slide{0012}+\slide{0021})$
\\
$0002$ & $0111$ & $(t_6-t_2)(\slide{0112}+\slide{0121}+\slide{0211})$\\
$0012$ & $0012$ & $(t_8-t_3)\slide{0014}+(t_7-t_2)\slide{0131}+(t_5-t_2)\slide{0221}+(t_5-t_4)(\slide{0212}+\slide{1112})+$\\
{} & {} & $(t_7-t_2)\slide{0131}+(t_5-t_2)\slide{0221}+(t_6 + t_5 - t_4 - t_3)\slide{0023}+(t_6+t_5-t_2-t_1)\slide{1121}+$\\
{} & {} & $(t_7 + t_6 + t_5 - t_4 - t_3 - t_2)(\slide{0113}+2\slide{0122})+(t_5-t_2)(t_6-t_2)\slide{0121}+$\\
{} & {} & $(t_6 + t_5 - t_4 - t_3)(t_7 - t_3))\slide{0013}+(t_5-t_4)(t_5-t_3)\slide{0022}+$\\
{} & {} & $(t_5 - t_4)(t_6 + t_5 - t_3 - t_2)\slide{0112}+(t_5 - t_3)(t_5-t_4)(t_6 - t_3)\slide{0012}$\\
\hline
\caption{Expansions of $\slide{c}\cdot \slide{d}$ where we only record summands indexed by padded compositions of size $<|c|+|d|$}
\label{table:fundamentalprod}
\end{longtable}

\bibliographystyle{hplain}
\bibliography{main.bib}

\begin{thebibliography}{10}

\bibitem{AJS}
H.~H. Andersen, J.~C. Jantzen, and W.~Soergel.
\newblock Representations of quantum groups at a {$p$}th root of unity and of
  semisimple groups in characteristic {$p$}: independence of {$p$}.
\newblock {\em Ast\'{e}risque}, (220):321, 1994.

\bibitem{AF24}
D.~Anderson and W.~Fulton.
\newblock {\em Equivariant cohomology in algebraic geometry}, volume 210 of
  {\em Cambridge Studies in Advanced Mathematics}.
\newblock Cambridge University Press, Cambridge, 2024.

\bibitem{ABB04}
J.-C. Aval, F.~Bergeron, and N.~Bergeron.
\newblock Ideals of quasi-symmetric functions and super-covariant polynomials
  for {$S_n$}.
\newblock {\em Adv. Math.}, 181(2):353--367, 2004.

\bibitem{Ber93}
N.~Bergeron and S.~Billey.
\newblock R{C}-graphs and {S}chubert polynomials.
\newblock {\em Experiment. Math.}, 2(4):257--269, 1993.

\bibitem{BeGa23}
N.~Bergeron and L.~Gagnon.
\newblock The excedance quotient of the {B}ruhat order, quasisymmetric
  varieties, and {T}emperley--{L}ieb algebras.
\newblock {\em J. Lond. Math. Soc. (2)}, 110(4):Paper No. e13007, 2024.

\bibitem{BGNST_Toappear}
N.~Bergeron, L.~Gagnon, P.~Nadeau, H.~Spink, and V.~Tewari.
\newblock The quasisymmetric flag variety, in preparation.

\bibitem{BS98}
N.~Bergeron and F.~Sottile.
\newblock Schubert polynomials, the {B}ruhat order, and the geometry of flag
  manifolds.
\newblock {\em Duke Math. J.}, 95(2):373--423, 1998.

\bibitem{biane:hal-04891027}
P.~Biane.
\newblock {Maximal elements in the Hurwitz graph}.
\newblock working paper or preprint, January 2025.

\bibitem{Bil99}
S.~C. Billey.
\newblock Kostant polynomials and the cohomology ring for {$G/B$}.
\newblock {\em Duke Math. J.}, 96(1):205--224, 1999.

\bibitem{BJS93}
S.~C. Billey, W.~Jockusch, and R.~P. Stanley.
\newblock Some combinatorial properties of {S}chubert polynomials.
\newblock {\em J. Algebraic Combin.}, 2(4):345--374, 1993.

\bibitem{BW83}
A.~Bj\"{o}rner and M.~Wachs.
\newblock On lexicographically shellable posets.
\newblock {\em Trans. Amer. Math. Soc.}, 277(1):323--341, 1983.

\bibitem{Bor53}
A.~Borel.
\newblock Sur la cohomologie des espaces fibr\'{e}s principaux et des espaces
  homog\`enes de groupes de {L}ie compacts.
\newblock {\em Ann. of Math. (2)}, 57:115--207, 1953.

\bibitem{DeTe19}
P.~Dehornoy and E.~Tesson.
\newblock Garside combinatorics for {T}hompson's monoid {$F^+$} and a hybrid
  with the braid monoid {$B^+_\infty$}.
\newblock {\em Algebr. Comb.}, 2(4):683--709, 2019.

\bibitem{FK96}
S.~Fomin and A.~N. Kirillov.
\newblock The {Y}ang-{B}axter equation, symmetric functions, and {S}chubert
  polynomials.
\newblock In {\em Proceedings of the 5th {C}onference on {F}ormal {P}ower
  {S}eries and {A}lgebraic {C}ombinatorics ({F}lorence, 1993)}, volume 153,
  pages 123--143, 1996.

\bibitem{GobetWilliams}
T.~Gobet and N.~Williams.
\newblock Noncrossing partitions and {B}ruhat order.
\newblock {\em European J. Combin.}, 53:8--34, 2016.

\bibitem{GK21}
R.~Goldin and A.~Knutson.
\newblock Schubert structure operators and {$K^*_T(G/B)$}.
\newblock {\em Pure Appl. Math. Q.}, 17(4):1345--1385, 2021.

\bibitem{GKM98}
M.~Goresky, R.~Kottwitz, and R.~MacPherson.
\newblock Equivariant cohomology, {K}oszul duality, and the localization
  theorem.
\newblock {\em Invent. Math.}, 131(1):25--83, 1998.

\bibitem{Gra01}
W.~Graham.
\newblock Positivity in equivariant {S}chubert calculus.
\newblock {\em Duke Math. J.}, 109(3):599--614, 2001.

\bibitem{HNT05}
F.~Hivert, J.-C. Novelli, and J.-Y. Thibon.
\newblock The algebra of binary search trees.
\newblock {\em Theoret. Comput. Sci.}, 339(1):129--165, 2005.

\bibitem{KM05}
A.~Knutson and E.~Miller.
\newblock Gr\"{o}bner geometry of {S}chubert polynomials.
\newblock {\em Ann. of Math. (2)}, 161(3):1245--1318, 2005.

\bibitem{KnUd23}
A.~Knutson and G.~Udell.
\newblock Interpolating between classic and bumpless pipe dreams.
\newblock {\em S\'{e}m. Lothar. Combin.}, 89B:Art. 89, 12 pp., 2023.

\bibitem{Lam18}
T.~Lam, S.J. Lee, and M.~Shimozono.
\newblock Back stable {S}chubert calculus.
\newblock {\em Compos. Math.}, 157(5):883--962, 2021.

\bibitem{LS82}
A.~Lascoux and M.-P. Sch\"{u}tzenberger.
\newblock Polyn\^{o}mes de {S}chubert.
\newblock {\em C. R. Acad. Sci. Paris S\'{e}r. I Math.}, 294(13):447--450,
  1982.

\bibitem{Mat64}
H.~Matsumoto.
\newblock G\'{e}n\'{e}rateurs et relations des groupes de {W}eyl
  g\'{e}n\'{e}ralis\'{e}s.
\newblock {\em C. R. Acad. Sci. Paris}, 258:3419--3422, 1964.

\bibitem{NST_c}
P.~Nadeau, H.~Spink, and V.~Tewari.
\newblock The geometry of quasisymmetric coinvariants, 2024, 2410.12643.

\bibitem{NST_a}
P.~Nadeau, H.~Spink, and V.~Tewari.
\newblock Quasisymmetric divided differences, 2024, 2406.01510.

\bibitem{NST_2}
P.~Nadeau, H.~Spink, and V.~Tewari.
\newblock Schubert polynomial expansions revisited, 2024, 2407.02375.

\bibitem{NT_forest}
P.~Nadeau and V.~Tewari.
\newblock Forest polynomials and the class of the permutahedral variety.
\newblock {\em Adv. Math.}, 453:Paper No. 109834, 2024.

\bibitem{PeSa23}
O.~Pechenik and M.~Satriano.
\newblock James reduced product schemes and double quasisymmetric functions,
  2023, 2304.11508.

\bibitem{ShWa16}
J.~Shareshian and M.~L. Wachs.
\newblock Chromatic quasisymmetric functions.
\newblock {\em Adv. Math.}, 295:497--551, 2016.

\bibitem{EC2}
R.~P. Stanley.
\newblock {\em Enumerative combinatorics. {V}ol. 2}, volume 208 of {\em
  Cambridge Studies in Advanced Mathematics}.
\newblock Cambridge University Press, Cambridge, [2024] \copyright 2024.
\newblock Second edition [of 1676282], With an appendix by Sergey Fomin.

\bibitem{Stem96}
J.~R. Stembridge.
\newblock On the fully commutative elements of {C}oxeter groups.
\newblock {\em J. Algebraic Combin.}, 5(4):353--385, 1996.

\bibitem{Tam62}
D.~Tamari.
\newblock The algebra of bracketings and their enumeration.
\newblock {\em Nieuw Arch. Wisk. (3)}, 10:131--146, 1962.

\bibitem{Tits69}
J.~Tits.
\newblock Le probl\`eme des mots dans les groupes de {C}oxeter.
\newblock In {\em Symposia {M}athematica ({INDAM}, {R}ome, 1967/68), {V}ol. 1},
  pages 175--185. Academic Press, London-New York, 1969.

\bibitem{TymSur}
J.~Tymoczko.
\newblock Billey's formula in combinatorics, geometry, and topology.
\newblock In {\em Schubert calculus---{O}saka 2012}, volume~71 of {\em Adv.
  Stud. Pure Math.}, pages 499--518. Math. Soc. Japan, [Tokyo], 2016.

\bibitem{Vien86heaps}
G.~X. Viennot.
\newblock Heaps of pieces. {I}. {B}asic definitions and combinatorial lemmas.
\newblock In {\em Graph theory and its applications: {E}ast and {W}est
  ({J}inan, 1986)}, volume 576 of {\em Ann. New York Acad. Sci.}, pages
  542--570. New York Acad. Sci., New York, 1989.

\end{thebibliography}
\end{document}